\documentclass[a4paper,12pt,leqno]{amsart}
\usepackage[a4paper, top=2.4cm, bottom=2cm, left=3cm, right=3cm]{geometry}

\usepackage{amsmath,amssymb,latexsym,amsthm,tikz, enumerate,amscd,hyperref}
\usepackage{graphicx,tikz}

\usetikzlibrary{positioning,arrows}
\usepackage[utf8]{inputenc} 
\usepackage[T1]{fontenc}
\usepackage{hyperref}
\usepackage{theoremref}

\numberwithin{equation}{section}
\theoremstyle{plain}
\newtheorem{lem}[equation]{Lemma}

\newtheorem{prop}[equation]{Proposition}
\newtheorem{thm}[equation]{Theorem}
\newtheorem{cor}[equation]{Corollary}

\newtheorem{conj}[equation]{Conjecture}
\theoremstyle{definition}
\newtheorem{definition}[equation]{Definition}

\newtheorem{remark}[equation]{Remark}

\newtheorem{claim}{Claim}

\newtheorem*{MT}{Main Theorem}
\newtheorem*{claim*}{Claim}
\newtheorem*{que*}{Question}
\newtheorem*{remark*}{Remark}
\newtheorem*{definition*}{Definition}

\newcommand{\type}{\operatorname{Type}}

\newcommand{\od}{\widehat{\Sigma}}
\newcommand{\Si}{\Sigma}

\newcommand{\bC}{\mathfrak C}

\newcommand{\ca}{\mathcal {A}} 
\newcommand{\prj}{\operatorname{Proj}} 
\newcommand{\s}{\operatorname{Sal}} 

\newcommand{\vertex}{\operatorname{Vert}} 
\newcommand{\lk}{\operatorname{lk}}

\newcommand{\cp}{\mathcal {P}} 
\newcommand{\ce}{\mathcal {E}}

\newcommand{\whC}{\widehat C}

\newcommand{\supp}{\operatorname{Supp}}
\newcommand{\mc}{\operatorname{Mincut}}

\newcommand{\ad}{\textbf{d}}
\newcommand{\uls}{\underline{s}}
\newcommand{\bS}{\mathsf S}

	{ 
	\end{enumerate}
}

\begin{document}

\title{Bestvina metric and tree reduction for $K(\pi,1)$-conjecture}

\author{Jingyin Huang}
\address{Department of Mathematics, The Ohio State University, 100 Math Tower, 231 W 18th Ave, Columbus, OH 43210, U.S.}
\email{huang.929@osu.edu}
\maketitle

\begin{abstract}
We give a dimension independent reduction of the $K(\pi,1)$-conjecture for Artin groups to the case of Coxeter diagrams that are trees. This yields new classes of Artin groups satisfying the conjecture and serves as a foundation for subsequent work. The reduction is achieved by constructing actions of Artin groups on a ``tower'' of Garside groupoids. A key ingredient is a method for constructing these groupoids using Bestvina's asymmetric metric, which may be of independent interest.
\end{abstract}
\section{Introduction}

\subsection{Main results}
The $K(\pi,1)$-conjecture, due to Arnol'd, Brieskorn, Pham, and Thom, predicts that for each Artin group $A_S$, the space of regular orbits of a canonical action of the associated Coxeter group is a $K(A_S,1)$-space. We refer to \cite{paris2012k} for more background on this conjecture. It has been established for several important classes of Artin groups, notably spherical, type FC, affine etc \cite{deligne,charney1995k,charney2004deligne,callegaro2010k,paolini2021proof,juhasz2023class,goldman2022k,haettel2023new,huang2023labeled,huang2024,goldman2025deligne,huang2025353}. Nevertheless, our understanding of the conjecture remains limited in dimensions $\ge 4$.

The main goal of this article is to provide a general mechanism for approaching this problem: we show that the $K(\pi,1)$-conjecture for arbitrary Artin groups can be reduced to properties of Artin groups whose Coxeter diagrams are trees. The reduction is dimension independent and gives a systematic way of passing from Coxeter diagrams with arbitrarily complicated collections of cycles to the tree case.



\smallskip

We begin with the Artin complex, a natural analogue of the classical Coxeter complex. Let $A_S$ be an Artin group with standard generating set $S$. The \emph{Artin complex} $\Delta_S$ has vertices corresponding to left cosets $gA_{S\setminus{s}}$, where $g\in A_S$ and $s\in S$, and a collection of vertices spans a simplex whenever the corresponding cosets have non-empty common intersection.

A vertex corresponding to $gA_{S\setminus{s}}$ is said to have type $\hat s=S\setminus{s}$. A \emph{special $4$-cycle} in $\Delta_S$ is an embedded $4$-cycle in the $1$-skeleton $\Delta_S^1$ whose vertex types alternate as $\hat s\hat t\hat s\hat t$ for some $s,t\in S$. Such a cycle has a \emph{center} if its four vertices are adjacent to a common vertex of $\Delta_S$.
 


\begin{MT}(Corollary~\ref{cor:cycle reduction all})
	\label{thm:cycle 1}
	Suppose that for each Artin group $A_S$ with Coxeter diagram being a tree, $A_S$ satisfies the $K(\pi,1)$-conjecture and any special 4-cycle in $\Delta_S$ has a center. Then any Artin group satisfies the $K(\pi,1)$-conjecture.
\end{MT}

Several useful strengthening of this theorem are proved in Section~9, including Corollary~\ref{cor:cycle reduction single} which is a local version of Theorem~\ref{thm:cycle 1}; and Corollary~\ref{cor:G2} and Corollary~\ref{cor:G2 single} which are about reducing to a smaller class of tree diagrams using \cite{huang2023labeled}.
%
%
%
%
%
%
%
%

We conjecture that this 4-cycle condition always holds.
\begin{conj}(\cite[Conj~9.18]{huang2023labeled})
	\label{conj:bowtie free}
	Suppose $A_S$ is an Artin group with its Coxeter diagram being a tree. Then any special 4-cycle in $\Delta_S$ has a center.
\end{conj}

We already have strong evidence towards this conjecture: it holds whenever all edges of the Coxeter diagram have label $\ge 4$  \cite[Cor.~9.14]{huang2023labeled}. The conjecture is also known when $A_S$ is spherical, or more generally when its Coxeter diagram is a tree satisfying the conditions of \cite[Thm.~1.1]{huang2023labeled},  or of dimension $\le 3$ \cite{huang2025353}. In general, proving Conjecture~\ref{conj:bowtie free} reduces to understanding certain pairs of commuting elements in $A_S$ \cite[Prop.~8.3]{huang2023labeled}.

%
%
%


As an immediate consequence of the Main Theorem, we obtain the following unconditional theorem, which is a substantial generalization of \cite[Thm.~1.1]{huang2023labeled}. 

\begin{thm}(=Corollary~\ref{cor:spherical combine})
	\label{cor combine0}
	Let $\Lambda$ be a Coxeter diagram. Suppose that, after removing all edges of $\Lambda$ with labels at least $6$, every connected component of the resulting diagram is either spherical or 3-dimensional. Then the Artin group $A_\Lambda$ satisfies the $K(\pi,1)$-conjecture.
\end{thm}

The strength of the Main Theorem is also illustrated by subsequent joint work with Hoda \cite{HH26}, where it provides the essential first step in a substantially broader higher dimensional result. Starting from the reduction of arbitrary Coxeter diagrams to trees established here, the remaining tree case is further reduced to several special families of linear and tripod-shaped diagrams in \cite{HH26}. A combination of these results leads to the $K(\pi,1)$-conjecture for every Artin group whose spherical parabolic subgroups avoid type $D_n$ for $n\geq 4$ and the exceptional types in \cite{HH26}. 

\medskip

The Main Theorem is based on constructing actions of Artin groups on suitable
Garside groupoids. Garside groups and groupoids have played a central role in
earlier approaches to the $K(\pi,1)$-conjecture; see, for instance,
\cite{deligne,mccammond2017artin,paolini2021proof}. In the generality of the
Main Theorem, the previous works
\cite{huang2023labeled,huang2024,huang2025353} suggest a natural candidate on
which an Artin group acts and which one might hope to realize as a Garside groupoid. However, this breaks down at a very
basic level as it does not pass one of the first requirements in constructing a Garside
structure: the underlying relation does not define a
poset; see Section~\ref{subsec:proof} for a more detailed discussion. This
obstruction prevents the approach of \cite{huang2023labeled,huang2024,huang2025353} from extending
directly to general Coxeter diagrams, and is one of the main reasons why those papers are restricted either to dimension at most $3$ or to Coxeter diagrams
with only one or two cycles, such as those in
\cite[Figure~3]{huang2023labeled}. In the present article, we introduce a new complex which overcomes this obstruction and makes it possible to treat arbitrary Coxeter diagrams containing cycles.




A second new feature is the use of Garside groupoids with a tower structure.
More precisely, we construct actions of Artin groups on Garside groupoids such
that the stabilizer of each morphism is itself a smaller Artin group. These
smaller Artin groups, in turn, admit analogous actions whose stabilizers are
yet smaller Artin groups. This produces a tower of Garside groupoids and,
consequently, an inductive framework for proving $K(\pi,1)$-results, as well
as potentially other structural properties of Artin groups. Such groupoids
can be constructed whenever the Coxeter diagram contains a cycle, so the base
cases of the induction are precisely Artin groups whose Coxeter diagrams are
trees. More precisely, we establish Theorem~\ref{thm:Bestvina complex} below,
by combining Theorem~\ref{thm:cycle reduction},
Proposition~\ref{prop:mincutAn}, and Corollary~\ref{cor:bestvina}.




\begin{thm}
	\label{thm:Bestvina complex}
Let $A_S$ be an Artin group such that its Coxeter diagram $\Lambda$ contains a cycle. Suppose that for each induced subdiagram $\Lambda'\subset\Lambda$ that is a tree, the associated Artin complex $\Delta_{S'}$ satisfies that each special 4-cycle in $\Delta_{S'}$ has a center. Then $A_S$ acts cocompactly on the Bestvina complex of a Garside groupoid, such that stabilizer of each simplex is isomorphic to a proper parabolic subgroup of $A_S$.
\end{thm}

Roughly speaking, the Bestvina complex is the ``central quotient'' of the Garside groupoid. It was originally defined for spherical Artin groups \cite{Bestvina1999}, then it is generalized to Garside groups \cite{charney2004bestvina} and Garside groupoids \cite[Sec 8]{bessis2006garside}. 

Bestvina's asymmetric metric plays an unexpected role in the construction of Garside groupoids in this theorem. Typically, this metric is viewed as a structure arising from a Garside groupoid. Here we use it in the reverse direction: geometric properties of the metric help guide the construction of the Garside groupoids needed for the reduction theorem. This perspective seems new and may be of independent interest in both Garside theory and metric geometry; see Section~\ref{subsec:proof} for more discussion.

\subsection{Discussion of proofs}
\label{subsec:proof}
We only discuss the proof of Theorem~\ref{thm:Bestvina complex}, which is the most important ingredient towards other main theorems. Using a combination of tools from Garside theory and non-positive curvature, we will gradually reduces Theorem~\ref{thm:Bestvina complex} to a problem about triangulation of a 2-dimensional disk.

\medskip
\noindent
\textbf{Background on Bestvina complexes}\ \ \ 
Let $G=\mathbb Z^n$ with its standard basis as generators. Let $G^+$ be the positive monoid of $G$ (i.e. $G^+$ is made of elements all of whose coordinates are non-negative). We define a partial order on $G$: $g\le h$ if $h=gk$ with $k\in G^+$. This is indeed a partial order, and $(G,\le)$ is a lattice (i.e. every pair of element has a least common upper bound, and greatest common lower bound). 
Let $\Delta$ be the element $(1,1,\ldots,1)$ in $G$. Non-trivial elements of $G$ that are lower bounded by identity and upper bounded by $\Delta$ are called \emph{atoms} of $G$.
The \emph{Bestvina complex} $X$ for $G$, has a vertex for each cosets of the subgroup $\langle \Delta\rangle$. Two different vertices are joined by an edge if the two cosets are differed by an atom. The flag completion of this 1-skeleton is the Bestvina complex for $G$. 

The complex $X$ has extra structure, inherit from $(G,\le)$, which can be formulated in multiple ways. Here we adopt a treatment in \cite{haettel2021lattices}. 
Note that a maximal chain from $0$ to $\Delta$ in $(G,\le)$ descents to a cycle in $X$ which span a maximal simplex of $X$. Thus the liner order in this maximal chain gives to a cyclic order on the vertex set of this maximal simplex of $K$. Use group action, we obtain cyclic order on the vertex set of each maximal simplex of $X$, and these cyclic orders are compatible in the intersection of  two maximal simplices. As each cyclically order set with one element removed has a canonical induced linear order, we have a well-defined partial order $<_x$ on the vertex set of $\lk(x,X)$ for each vertex $x\in X$. Note that $(\lk(x,X),<_x)$ is isomorphic (as a poset) to $(\hat x,\Delta \hat x)$, where $\hat x$ is any lift of $x$ in $\mathbb Z^n$, and $(\hat x,\Delta\hat x)$ denotes the collection of elements of $(G,\le)$ that is $>\hat x$ and $<\Delta\hat x$.

For any Garside group, and more generally Garside groupoid, one can create a Bestvina complex in a similar fashion with cyclic orders on its maximal simplices. The vertex set of this complex is made of elements in the central quotient. Conversely, give a simplicial complex $X$ with cyclic orders on its maximal simplices, as long as $X$ is simply-connected and the relation $<_x$ in $\lk^0(x,X)$ satisfies that it is transitive (hence $<_x$ is a partial order),  any upper bounded pair $(\lk^0(x,X),<_x)$ has the join, and any lower bounded pair $(\lk^0(x,X),<_x)$ has the meet, then $X$ is the Bestvina complex of some Garside groupoid - this is a consequence of the classical local-to-global characterization of Garside structure, see \cite[Thm 1.3]{haettel2024lattices} for an explanation. It also follows from Garside theory that $X$ is contractible. This is an important local-to-global contractibility criterion used in this article.

Given a Coxeter group $W_S$ with generating set $S$, we can define its Coxeter complex $\bC_S$ and types of vertices in $\bC_S$ in the same fashion as Artin complexes. Now assume $W_S$ is of type $\widetilde A_{n-1}$. Then its Coxeter diagram $\Lambda$ is a cyclic, which gives a natural cyclic order on $S$. By considering types of vertices in $\bC_S$, we obtain a cyclic order on each maximal simplex of $\bC_S$. Then  $\bC_S$ is isomorphic to the Bestvina complex $X$ of $\mathbb Z^n$, such that the isomorphism respects cyclic orders on maximal simplices, see e.g. \cite[Sec 3.2]{hirai2020uniform}. Similarly, we can consider the Artin complex $\Delta_S$ associated with the Artin group $A_S$ of type $\widetilde A_{n-1}$. Again the Coxeter diagram being a cycle induces cyclic orders on maximal simplices of $\Delta_S$.
By work of Crisp-McCammond and Haettel \cite{haettel2021lattices}, the complex $\Delta_S$ with the cyclic orders on its maximal simplices satisfies all the conditions required for it to be a Bestvina complex of a Garside groupoid. 

\medskip
\noindent
\textbf{Failure of poset property}\ \ \ 
Now we consider the more general situation that $\Lambda$ is any Coxeter diagram containing an induced cycle $\Lambda'\subset \Lambda$. Let $S,S'$ be the vertex set of $\Lambda,\Lambda'$ respectively. Let $\Delta_S$ be the Artin complex of $A_S$, and
let $\Delta_{S,S'}$ be the induced subcomplex of $\Delta_S$ spanned by vertices of type $\hat s$ with $s\in S'$. As $S'$ is the vertex set of a cycle, we can endow $S'$ with a natural cyclic order. This gives a cyclic order on the vertex set of each maximal simplex of $\Delta_{S,S'}$. In \cite{huang2023labeled} we studied the question of whether $\Delta_{S,S'}$ is the Bestvina complex of a Garside groupoid. It is shown that for some hyperbolic type Artin groups whose Coxeter diagrams contain a cycle, $\Delta_{S,S'}$ is indeed the Bestvina complex of some Garside groupoid, which leads to $K(\pi,1)$-conjecture for these Artin groups. However, in general, $\Delta_{S,S'}$ can not be a Bestvina complex, since the relation $<_x$ defined as before is not transitive except for very specific Coxeter diagrams, so $(\lk^0(x,\Delta_{S,S'}),<_x)$ is not a poset!

\medskip
\noindent
\textbf{A different complex}\ \ \ 
In order to construct Bestvina complexes for Artin groups with cycles in their Coxeter diagrams in full generality, the first main ingredient is a different complex, which we called \emph{minimal cut complex}, which can be used to fix this fundamental issue of failure of poset property. We demonstrate the complex in a simple example. Let $\Lambda$ be the diagram in Figure~\ref{fig:example}, with $P$ being the thickened oriented path from $b$ to $a$ in $\Lambda$. Let $A_\Lambda$ be the associated Artin group. Let $\Lambda_P$ be the subdiagram of $\Lambda$ obtained by removing interior points of the path $P$ from $\Lambda$. Let $a,b$ be the two vertices in $\Lambda_P$ as in Figure~\ref{fig:example}. Given two sets of vertices $A,B$ of $\Lambda_P$, a \emph{cut} of $\Lambda_P$ between $A$ and $B$ is a set $T$ of vertices of $\Lambda_P$ such that any path from a point in $A\setminus T$ to a point in $B\setminus T$ has non-empty intersection with $T$ (we assume this property holds true automatically if $A\setminus T=\emptyset$ or $B\setminus T=\emptyset$). A \emph{minimal cut} of $\Lambda_P$ between $A$ and $B$ is a cut such that any proper subset of this cut is not a cut. Let $\mc_{\Lambda_P}(\{a\},\{b\})$ be the collection of minimal cuts separating $\{a\}$ and $\{b\}$. We say two such minimal cuts are \emph{comparable}  if one of them is a cut of $\Lambda_P$ between $\{a\}$ and the remaining one.

\begin{figure}
	\centering
	\includegraphics[scale=1]{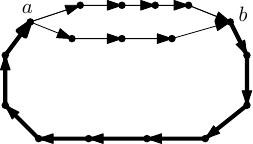}
	\caption{The diagram $\Lambda$.}
	\label{fig:example}
\end{figure}

Let $\mathcal C_P$ be the collection of sets of vertices of $\Lambda$ such that an element of $\mathcal C_P$ is either an element of $\mc_{\Lambda_P}(\{a\},\{b\})$ (which is a \emph{type II element}), or is made of a single vertex that is an interior vertex of $P$ (which is a\emph{type I} element). Two different elements are \emph{comparable}, if either one of them is a type I element, or both of them are type II and they are comparable. Note that orientation of edges in $\Lambda$ as in Figure~\ref{fig:example} induces a cyclic order on any maximal collection of elements in $\mathcal C_P$ that are pairwise comparable.

Now we define a complex $\Delta^P_\Lambda$ upon which the Artin group $A_\Lambda$ acts. Vertices of $\Delta^P_\Lambda$ are in 1-1 correspondence with left cosets of form $gA_{\hat T}$ with $g\in A_\Lambda$ and $T\in \mathcal C_P$, where $A_{\hat T}$ denotes the subgroup of $A_\Lambda$ generated by all generators except those in $T$. A vertex corresponding to $gA_{\hat T}$ has \emph{type} $\hat T$. Two vertices of type $\hat T_1$ and $\hat T_2$ are adjacent, if $T_1$ and $T_2$ are comparable, and the corresponding cosets have non-empty intersection. Then $\Delta^P_\Lambda$ is the flag completion of its 1-skeleton. Given a maximal simplex in $\Delta^P_\Lambda$, by considering type of vertices in this simplex and the previous paragraph, we obtain a cyclic order on its vertex set. The advantage of this new complex is that the relation $<_x$ is now transitive, and it can be defined more generally whenever $\Lambda$ contains a cycle, for an appropriate choice of path $P$ inside $\Lambda$. There is a natural action $A_\Lambda\curvearrowright \Delta^P_\Lambda$, whose vertex stabilizers are proper parabolic subgroups of $A_\Lambda$.

\medskip
\noindent
\textbf{Labeled 4-cycle property and the lattice condition}\ \ \ 
In order to show $\Delta_P$ is a Bestvina complex, it remains to verify that for vertex $x\in \Lambda_P$, the poset $(\lk^0(x,\Delta^P_\Lambda),<x)$ satisfies that each upper bounded pair has the join and each lower bounded pair has the meet. We will refer these two properties as the \emph{lattice condition}, as they will imply if we add a greatest element and a smaller element to $(\lk^0(x,\Delta^P_\Lambda),<x)$, then we obtain a lattice. Our second ingredient relates the lattice condition to \emph{labeled 4-cycle property}. 
\begin{definition}
	\label{def:labaledintro}
	Let $\Lambda$ be a Coxeter diagram, and let $\Delta_\Lambda$ be the associated Artin complex.
	We say $\Delta_\Lambda$ satisfies \emph{labeled 4-cycle condition}, if for any embedded 4-cycle $x_1x_2x_3x_4$ in $\Delta_\Lambda$ with $x_1,x_3$ having the same type $\hat s$ and $x_2,x_4$ having the same type $\hat t$, and any connected induced subgraph $\Lambda'\subset \Lambda$ containing both $s$ and $t$, there is a vertex $z$ such that
	\begin{enumerate}
		\item $z$ is adjacent to each $x_i$ for $1\le i\le 4$. 
		\item $z$ has type $\hat r$ with $r\in \Lambda'$.
	\end{enumerate}
\end{definition}
This condition was defined in the special case when $\Lambda$ is a tree in \cite{huang2023labeled} before we have the correct formulation in full generality here. This property is related to the lattice condition in the following way.

\begin{prop}
	\label{prop:mincutAnintro}
	Let $\Lambda$ be a Coxeter diagram. 	
	Suppose all proper induced subdiagrams $\Lambda'$ of $\Lambda$, $\Delta_{\Lambda'}$ satisfies the labeled 4-cycle condition.  Then all vertex links of complex $\Delta^P_\Lambda$ satisfy the lattice condition. Hence $\Delta^P_\Lambda$ is a Bestvina complex.
\end{prop} 

There is a serious gap between the assumption of Theorem~\ref{thm:Bestvina complex} and the assumption of Proposition~\ref{prop:mincutAnintro}, as assumption of the former only implies labeled 4-cycle condition for $\Delta_{\Lambda'}$ with $\Lambda'$ being an induced \textbf{tree} subdiagram of $\Lambda$. This gap can be filled by
 repeatedly applying the following propagation theorem of labeled 4-cycle property, which is the core of this article. Due to the relation between labeled 4-cycle property and lattice condition discussed before, the following can also viewed as a propagation theorem for the lattice condition.

\begin{prop}
	\label{prop:propintro}
	Suppose $\Lambda$ is a connected Coxeter diagram which is not a tree. Suppose all proper induced subdiagrams $\Lambda'$ of $\Lambda$, $\Delta_{\Lambda'}$ satisfies the labeled 4-cycle condition. Then $\Delta_{\Lambda}$ satisfies the labeled 4-cycle condition.
\end{prop}

\medskip
\noindent
\textbf{Propagation via Bestvina non-positive curvature}\ \ \ 
Now we discuss the proof of Proposition~\ref{prop:propintro}. Given a 4-cycle $\eta=x_1x_2x_3x_4$ in $\Delta_\Lambda$ of type $\hat s\hat t\hat s\hat t$. Our goal is to find an appropriate center of this 4-cycle, i.e. a vertex adjacent to each $x_i$ satisfying Definition~\ref{def:labaledintro}. 

The assumption of Proposition~\ref{prop:propintro} and Proposition~\ref{prop:mincutAnintro} implies that for an appropriate choice of the path $P$ in $\Lambda$, the complex $X=\Delta^P_\Lambda$ is a Bestvina complex. Between each pair of two vertices, there is a preferred shortest edge path in $X^1$ between them introduced by Bestvina in \cite{Bestvina1999}, which we called the \emph{B-geodesic}. An induced subcomplex $Y$ of $X$ is \emph{B-convex}, if for any two vertices of $Y$, the $B$-geodesic between them is contained in $Y$. While the property of being $B$-convex seems to be a global property that is hard to check, we showed in \cite{huang2025353} that it suffices to check that $\lk(y,Y)$ is locally convex in $\lk(y,X)$ for each vertex $y\in Y$ in an appropriate sense, which makes $B$-convex subsets easier to identify.

We reduce the study the 4-cycle $\eta$ in $\Delta_\Lambda$ to the study of a configuration of four $B$-convex subcomplexes of $X$ as follows. Given a vertex $x\in X$ corresponding to $gA_{\hat T}$ for $g\in A_\Lambda$ and $T$ being a set of vertices of $\Lambda$, we can realize $x$ as the barycenter of the simplex in $\Delta_\Lambda$ spanned by vertices corresponding to $\{gA_{\hat s}\}_{s\in T}$. Thus each vertex of $X$ can be viewed as the barycenter of a simplex in $\Delta_\Lambda$. Given $x_i\in \Delta_\Lambda$, let $Y_i$ be the full subcomplex of $X$ spanned by the collection of vertices of $X$ that are contained in the same simplex of $\Delta_\Lambda$ as $x_i$. It turns out that each $Y_i$ is $B$-convex in $X$, and $X_i\cap Y_{i+1}\neq\emptyset$ for $i\in \mathbb Z/4\mathbb Z$. Moreover, the task of finding a center of $\eta$ can be reduced to showing that such configuration of $Y_1,Y_2,Y_3,Y_4$ is degenerate in the sense that either $Y_1\cap Y_3\neq\emptyset$ or $Y_2\cap Y_4\neq\emptyset$. 

Let $\Xi$ be the collection of quadruple $(y_1,y_2,y_3,y_4)$ of vertices in $\Delta^P_\Lambda$ with $y_i\in Y_i\cap Y_{i+1}$ for all $i\in\mathbb Z/4\mathbb Z$. Each such quadruple determines a 4-gon in $X$ by considering the $B$-geodesic from $y_i$ to $y_{i+1}$. The convexity of $Y_i$ implies that each side of this 4-gon is contained in one of $Y_i$. Such 4-gon is called an \emph{admissible 4-gon}. As $X$ is simply-connected, each such 4-gon bound a triangulated disk in $X$. The task of showing the configuration $Y_1,Y_2,Y_3,Y_4$ is degenerate can be further reduced to finding an admissible 4-gon that bounds a degenerate triangulated disk.

In order to find such 4-gon, we look at the collection $\Xi'$ of quadruples in $\Xi$ that minimize the distance sum $\sum_{i\in \mathbb Z/4\mathbb Z}d(y_i,y_{i+1})$, here $d$ denotes the path metric on $X^1$ with edge length one. There is another metric $\ad$ on $X^0$, introduced by Bestvina \cite{Bestvina1999}, which is an asymmetric metric, i.e. $\ad(x,y)\neq\ad(y,x)$ in general. The metric $\ad$ has the advantage of preserving more information from the Garside structure, and it has a form of non-positive curvature \cite[Prop 3.12]{Bestvina1999}, which will play a key role in our argument.

Among quadruples in $\Xi'$, we select a member which also minimizes a carefully designed distance sum of some of the $y_i$'s with respect to Bestvina's asymmetric metric $\ad$, and consider the associated 4-gon. Most of the work here is to show that, one can produce a filling of this 4-gon using $B$-geodesics, such that the resulting triangulated disk from such filling has a very specific combinatorial structure, due to Bestvina's non-positive curvature on this asymmetric metric. Moreover, this triangulated has extra decorations coming from Garside theory, and information from the ambient Artin groups, which further limits the possibility of such triangulation, eventually leads to degeneracy of the disk.

We caution the reader that it is not true that whenever we have four $B$-convex subcomplexes $\{X_i\}_{i=1}^4$ in a Bestvina complex $X$ with $X_i\cap X_{i+1}\neq\emptyset$, then either $X_1\cap X_3\neq\emptyset$ or $X_2\cap X_4\neq\emptyset$. Here we only work with specific convex subcomplexes coming from the 4-cycle $\eta$ as explained before. The Bestvina complex $X=\Delta^P_\Lambda$ depends on the choice of the path $P$ in $\Lambda$. For different choices of $P$, we obtain different Bestvina complexes with four convex subcomplexes coming from $\eta$. It is also not true that these four convex subcomplexes always behave in the way we expect for any choice of $P$ - in the end, the choice of $P$ depends on the type of vertices of $\eta$.

\subsection{Structure of the article}
In Section~\ref{sec:complexes} and Section~\ref{sec:relations}, we collect some preliminary material. Section~\ref{sec:minimal cuts} is a pure graph theoretical discussion on properties of minimal cuts in graphs which are needed in later sections. Section~\ref{sec:labeled 4-cycle} is about the labeled 4-cycle property for Artin complexes $\Delta_\Lambda$, and a companion of this property, called the strong labeled 4-cycle property. We will discuss the relation of these two properties. Section~\ref{sec:garside} collects some background on Garside theory, and has a discussion on Bestvina's asymmetric metric and the non-positive curvature aspect of this metric. In Section~\ref{sec:minimal cut complex}, we define the minimal cut complex in full generality, and prove Proposition~\ref{prop:mincutAnintro}. Section~\ref{sec:propagation} is devoted to the proof of the key propagation result Proposition~\ref{prop:propintro}. In Section~\ref{sec:contractibility} we deduce all the results on $K(\pi,1)$-conjecture in the introduction.
	
\subsection{Acknowledgment}
We thank Piotr Przytycki for valuable discussion related to this article. The author is partially supported by a Sloan fellowship and NSF grant DMS-2305411. The author thanks Chinese Academy of Sciences for hospitality, where part of the work was undertaken.
The author thanks the Isaac Newton Institute for Mathematical Sciences, Cambridge, for support and hospitality during the programme Operators, Graphs, Groups, where part of the work was undertaken. This work was partially supported by EPSRC grant EP/Z000580/1.

\section{Complexes for Artin and Coxeter groups}
\label{sec:complexes}
\subsection{Artin complexes and relative Artin complexes}
\label{subsec:rel Artin complex}
A \emph{Coxeter diagram} $\Lambda$ is a finite simple graph with vertex set $S=\{s_i\}_i$ and labels $m_{ij}=3,4,\ldots ,\infty$ for each edge $s_is_j$. If $s_is_j$ is not an edge, we define $m_{ij}=2$. The Artin group $A_\Lambda$ is the group with generator set $S$ and relations $s_is_js_i\cdots=s_js_is_j\cdots$ with both sides alternating words of length $m_{ij}$, whenever $m_{ij}<\infty$. The Coxeter group $W_\Lambda$ is obtained from $A_\Lambda$ by adding relations $s_i^2=1$.

The \emph{pure Artin group} $PA_\Lambda$ is the kernel of the obvious homomorphism $A_\Lambda\to W_\Lambda$.
We say that $A_\Lambda$ is \emph{spherical}, if $W_\Lambda$ is finite. Recall that any $S'\subset S$ generates a subgroup of $A_\Lambda$ isomorphic to $A_{\Lambda'}$, where $\Lambda'$ is the subdiagram of $\Lambda$ induced on $S'$. Such subgroup is called a \emph{standard parabolic subgroup}. The following is a consequence of \cite[Lem 4.7 and Prop 4.5]{godelle2012k}.

\begin{prop}
	\label{prop:intersection}
Let $\{g_iA_{S_i}\}_{i=1}^n$ be a finite collection of left cosets of standard parabolic subgroups of $A_\Lambda$. Suppose these cosets pairwise have non-empty intersection. Then the common intersection of them is non-empty.
\end{prop}

The \emph{Artin complex} $\Delta_\Lambda$, introduced in \cite{CharneyDavis} and further studied in \cite{godelle2012k,cumplido2020parabolic}, is a simplicial complex defined as follows. For each $s\in S$, let $A_{\hat s}$ be the standard parabolic subgroup generated by $\hat s=S\setminus\{s\}$. The vertices of $\Delta_\Lambda$ correspond to the left cosets of $\{A_{\hat s}\}_{s\in S}$. Moreover, vertices span a simplex if the corresponding cosets have non-empty common intersection. It follows from \cite[Prop~4.5]{godelle2012k} that $\Delta_\Lambda$ is a flag complex.
The \emph{Coxeter complex} $\bC_\Lambda$ is defined analogously, where we replace $A_{\hat s}$ by $W_{\hat s}<W_\Lambda$ 
generated by $\hat s$. A vertex of $\bC_\Lambda$ or $\Delta_\Lambda$ corresponding a left coset of $W_{\hat s}$ or $A_{\hat s}$ has \emph{type} $\hat s$. 
We have that $\bC_\Lambda$ is the quotient of $\Delta_\Lambda$ under the action of $PA_\Lambda$.

Let $\Delta'_\Lambda$ be the barycentric subdivision of the Artin complex $\Delta_\Lambda$. Given a vertex $x\in \Delta'_\Lambda$ which is the barycenter of a simplex $\sigma\in \Delta_\Lambda$ with vertices of $\sigma$ have type $\hat s_1,\ldots,\hat s_k$, the \emph{type} of this vertex is defined to be $\hat T=S\setminus S'$ such that $\hat T=\cap_{i=1}^k\hat s_i$. Note that vertices of type $\hat T$ in $\Delta'_\Lambda$ are in 1-1 correspondence with left cosets in $A_\Lambda$ of form $gA_{S\setminus T}$, where $A_{S\setminus T}$ is the standard parabolic subgroup generated by elements in $S\setminus T$. Two vertices of $\Delta'_\Lambda$ are adjacent if and only if the left coset associated with one vertex is contained the left coset associated with another vertex. Similarly, we can define types of vertices in the barycentric subdivision $\bC'_\Lambda$ of $\bC_\Lambda$.

\begin{thm}\cite[Thm 3.1]{godelle2012k}
	\label{thm:kpi1}
Suppose $A_S$ is not spherical.	If $\Delta_S$ is contractible and each $\{A_{\hat s}\}_{s\in S}$ satisfies the $K(\pi,1)$-conjecture, then $A_S$ satisfies the $K(\pi,1)$-conjecture.
\end{thm}

Let $A_\Lambda$ be an Artin group with generating set $S$, and $S'\subset S$. The \emph{$(S,S')$-relative Artin complex $\Delta_{S,S'}$} (introduced in \cite{huang2023labeled}) is defined to be the induced subcomplex of the Artin complex $\Delta_S$ of $A_S$ spanned by vertices of type $\hat s$ with $s\in S'$. We also write $\Delta_{S,S'}$ as $\Delta_{\Lambda,\Lambda'}$ where $\Lambda,\Lambda'$ are Coxeter diagrams for $A_S,A_{S'}$ respectively.
Links of vertices in relative Artin complexes can be computed as follows.
\begin{lem}(\cite[Lem 6.4]{huang2023labeled})
	\label{lem:link}
	Let $\Delta=\Delta_{\Lambda,\Lambda'}$, and let $v\in \Delta$ be a vertex of type $\hat s$ with $s\in \Lambda'$. Let $\Lambda_s$ and $\Lambda'_s$ be the induced subgraph of $\Lambda$ and $\Lambda'$ respectively spanned all the vertices which are not $s$. 
	Then there is a type-preserving isomorphism between $\lk(v,\Delta)$ and $\Delta_{\Lambda_s,\Lambda'_s}$.
	Moreover, Let $I_s$ be the union of connected components of $\Lambda_s$ that contain at least one component of $\Lambda'_s$. Then $\Lambda'_s\subset I_s$ and there is a type-preserving isomorphism between $\lk(v,\Delta)$ and $\Delta_{I_s,\Lambda'_s}$.
\end{lem}

The following is a direct consequence of definition.

\begin{lem}
	\label{lem:c}
	Suppose $|S'|\ge 2$. Then $\Delta_{S,S'}$ is connected. 
\end{lem}

\begin{lem}(\cite[Lem 6.2]{huang2023labeled})
	\label{lem:sc}
	Suppose $|S'|\ge 3$. Then $\Delta_{S,S'}$ is simply-connected. 
\end{lem} 

The following is a special case of \cite[Lem 11.7 (2)]{huang2025353}.
\begin{lem}
	\label{lem:dr}
	Let $T$ be a subset of vertices of a Coxeter diagram $\Lambda$. Suppose $\Delta_{\Lambda\setminus R}$ is contractible for each nonempty subset $R$ of $T$. Then $\Delta_{\Lambda}$ deformation retracts onto $\Delta_{\Lambda,\Lambda\setminus T}$.
\end{lem}

\subsection{Davis complexes}
\label{subsec:complex}

\begin{definition}[Davis complex]
	Given a Coxeter group $W_\Lambda$, let $\mathcal{P}$ be the poset of left cosets of spherical standard parabolic subgroups in $W_\Lambda$ (with respect to inclusion) and let $b\Si_\Lambda$ be the geometric realization of this poset (i.e.\ $b\Si_\Lambda$ is a simplicial complex whose simplices correspond to chains in $\mathcal{P}$). Now we modify the cell structure on $b\Si_\Lambda$ to define a new complex $\Sigma_\Lambda$, called the \emph{Davis complex}. The cells in $\Sigma_\Lambda$ are induced subcomplexes of $b\Si_\Lambda$ spanned by a given vertex $v$ and all other vertices which are $\le v$ (note that vertices of $b\Si_\Lambda$ correspond to elements in $\mathcal{P}$, hence inherit the partial order).
\end{definition}

If $\Lambda'\subset \Lambda$ is an induced subgraph, then $W_{\Lambda'}\to W_{\Lambda}$ induces an embedding $\Sigma_{\Lambda'}\to \Sigma_{\Lambda}$. The image of this embedding and its left translations are \emph{standard subcomplexes} of type $\Lambda'$. There is a correspondence between standard subcomplexes of type $\Lambda'$ in $\Sigma_{\Lambda}$ and left cosets of $W_{\Lambda'}$ in $W_{\Lambda}$.

\begin{definition}
	\label{def:associated subcomplex}
	For each vertex $x\in \bC'_\Lambda$ of type $\hat T$, let $gW_{S\setminus T}$ be the associated left coset in $W_\Lambda$. Recall that vertices of the Davis complex $\Si_\Lambda$ can be identified with $W_\Lambda$. The \emph{standard subcomplex of $\Si_\Lambda$ associated with the vertex $x\in \bC_\Lambda$} is defined to be the standard subcomplex of $\Si_\Lambda$ spanned by vertices in $gW_{S\setminus T}$.
\end{definition}

Lemma~\ref{lem:gate} below is standard, see e.g.\ \cite{bourbaki2002lie} or \cite{davis2012geometry}. Let $d$ be the path metric on the 1-skeleton of $\Si_\Lambda$, with each edge having length 1.
Lemma~\ref{lem:gate} describes nearest point projection into the vertex set of a standard subcomplex.
\begin{lem}
	\label{lem:gate}
	Let $F$ be a standard subcomplex of $\Si_\Lambda$ and let $x\in \Si_\Lambda$ be a vertex. Then the following hold.
	\begin{enumerate}
		\item For two vertices $x_1,x_2\in\mathcal R$, the vertex set of any geodesic in $\Si^1_\Lambda$ joining $v_1$ and $v_2$ is inside $F$. Moreover, there exists a unique vertex $x_F\in F$ such that $d(x,x_F)\le d(x,y)$ for any vertex $y\in F$, where $d$ denotes the path metric on the 1-skeleton of $\Sigma_\Lambda$. The vertex $x_F$ is called the \emph{projection} of $x$ to $F$, and is denoted $\prj_F(x)$. 
		\item For any vertex $y\in F$, there exists a shortest edge path $\omega$ in $\Si^1_\Lambda$ from $x$ to $y$ so that $\omega$ passes through $x_F$ and so that the segment of $\omega$ between $x_F$ and $y$ is contained in $F$.
	\end{enumerate}
\end{lem}

\begin{definition}
	\label{def:projection1}
	Let $F$ be a standard subcomplex of $\Si_\Lambda$. Lemma~\ref{lem:gate} gives a map $\Pi_F:\vertex\Si_\Lambda\to\vertex F$ which extends to a retraction $\Pi_F:\Si_\Lambda\to F$ as follows. Note that for each face $E$ of $\Si_\Lambda$, $\pi(\vertex E)$ is the vertex set of a face $E'\subset F$. Then we extends $\pi$ to a map $\pi'$ from the vertex set of $b\Si_\Lambda$ to the vertex set of $bF$, by sending the barycenter of $E$ to the barycenter of $E'$. As $\pi'$ map vertices in a simplex to vertices in a simplex, it extends linearly to a map $\Pi_F:b\Si_\Lambda\cong \Si_\Lambda\to bF\cong F$.
\end{definition}

\subsection{Oriented Davis complexes and Salvetti complexes}
\label{subsec:Sal}
Let $\mathcal P$ be the poset of faces of $\Si_\Lambda$ (under containment), and let $V$ be the vertex set of $\Si_\Lambda$. We now define the \emph{oriented Davis complex} $\widehat\Si_\Lambda$ as follows.
Consider the set of pairs $(F,v)\in \cp \times V$.  Define  an equivalence relation $\sim$ on this set by $$(F,v)\sim (F,v') \iff F=F' \text{\ and\ } \prj_F(v') = \prj_F(v).$$
Denote the equivalence class of $(F,v')$ by $[F,v']$ and let $\ce(\ca)$ be  the set of equivalence classes.   Note  that each equivalence class $[F,v']$ contains a unique representative of the form $(F,v)$, with $v\in \vertex F$.  The \emph{oriented Davis complex} $\widehat\Si_\Lambda$ is defined as the regular CW complex given by taking  $\Si_\Lambda\times V$ (i.e., a disjoint union of copies of $\Si_\Lambda$) and then identifying faces $F\times v$ and $F\times v'$ whenever $[F,v]=[F,v']$, i.e.,
\begin{equation}
	\widehat\Si_\Lambda=( \Si_\Lambda\times V)/ \sim \ .
\end{equation}
For example, for each edge $F$ of $\Si_\Lambda$ with endpoints $v_0$ and $v_1$, we get two $1$-cells $[F,v_0]$ and $[F,v_1]$ of $\widehat\Si_\Lambda$ glued together along their endpoints $[v_0,v_0]$ and $[v_1,v_1]$.  So, the $0$-skeleton of $\widehat\Si_\Lambda$ is equal to the $0$-skeleton of $\widehat\Si_\Lambda$ while its $1$-skeleton is formed from the $1$-skeleton of $\widehat\Si_\Lambda$ by doubling each edge.  
There is a natural map $\pi:\widehat\Si_\Lambda\to\Si_\Lambda$ defined by ignoring the second coordinate. 

The definition of oriented Davis complex traced back to work of Salvetti \cite{s87}, so it is also called Salvetti complex by  other authors. The naming ``oriented Davis complex'' comes from an article of J. McCammond \cite{mccammond2017mysterious}, clarifying the relation between Salvetti's work and Davis complex.



\begin{definition}
	\label{def:label}
	As the 1-skeleton of $\Si_\Lambda$ can be identified with the unoriented Cayley graph of $W_\Lambda$, each edge of $\Si_\Lambda$ is labeled by an element in the generating set $S$.
	We pull back the edge labeling from $\Si_\Lambda$ to $\od_\Lambda$ via the map $\pi:\od_\Lambda\to \Si_\Lambda$. For a subset $E$ in $\Sigma$, we define $\supp(E)$ to be the collection of labels of edges in $E$. 
\end{definition}

For each subcomplex $Y$ of $\Si_\Lambda$, we write $\widehat Y=p^{-1}(Y)$ and call $\widehat Y$ the subcomplex of $\od_\Lambda$ associated with $Y$.
A \emph{standard subcomplex} of $\widehat\Si_\Lambda$ is a subcomplex of $\od_\Lambda$ associated with a standard subcomplex of $\Si_\Lambda$. In other words, if $F\subset \Si_\Lambda$ is a standard subcomplex, then $\widehat F$ is the union of faces of form $E\times v$ in $\widehat \Si_\Lambda$ with $E\subset F$ and $v$ ranging over vertices in $\Si_\Lambda$.

\begin{lem}(\cite[Lem 3.12]{huang2023labeled})
	\label{lem:compactible}
	Let $E$ be a face of $\Si_\Lambda$ and let $F$ be a standard subcomplex of $\Si_\Lambda$.
	If $[E,v_1]=[E,v_2]$, then $[\prj_F(E),v_1]=[\prj_F(E),v_2]$.
\end{lem}

We will use the following important construction in \cite{godelle2012k}.
\begin{definition}
	\label{def:retraction}
	Let $F$ be a face in $\Si_\Lambda$. Then there is a retraction map $\Pi_{\widehat F}:\widehat\Si_\Lambda\to \widehat F$ defined as follows. Recall that $\widehat\Si_\Lambda=( \Si_\Lambda\times V)/ \sim$. For each $v\in V$, let $(\Si_\Lambda)_v$ be the union of all faces in $\widehat\Si_\Lambda$ of form $E\times v$ with $E$ ranging over faces of $\Si_\Lambda$. By Definition~\ref{def:projection1}, there is a retraction $(\Pi_F)_v:(\Si_\Lambda)_v\to F\times v$ for each $v\in V$. It follows from Lemma~\ref{lem:compactible} that these maps $\{(\Pi_F)_v\}_{v\in V}$ are compatible in the intersection of their domains. Thus they fit together to define a retraction $\Pi_{\widehat F}:\widehat\Si_\Lambda\to \widehat F$.
\end{definition}
The following lemma is a direct consequence of the definition.
\begin{lem}
	\label{lem:retraction property}
	Take standard subcomplexes $E,F\subset \Si_\Lambda$. Then $\Pi_{\widehat F}(\widehat E)=\widehat{\Pi_F(E)}$.
\end{lem}


\subsection{Generalized cycles}
\label{subsec:generalized cycles}
Let $\Lambda$ be a connected Dynkin diagram with its vertex set $S$.
Given two vertices $x,y$ in the barycentric subdivision $\Delta'_\Lambda$ of the Artin complex $\Delta_\Lambda$, we write $x\sim y$ if they are contained a common simplex of $\Delta_\Lambda$. Then $x\sim y$ if and only if the associated two left cosets have nonempty intersection. Let $\sigma_x$ be the simplex of $\Delta_\Lambda$ such that $x$ is the barycenter of $\sigma_x$. A collection of vertices $x_1x_2\cdots x_n$ in $\Delta'_\Lambda$ form a \emph{generalized cycle} if $x_i\sim x_{i+1}$ for $1\le i\le n-1$ and $x_n\sim x_1$.

\begin{lem}(\cite[Lem 2.2]{huang2025353})
	\label{lem:transitive}
	Given vertices $x_1,x_2,x_3$ of types $\hat S_1,\hat S_2,\hat S_3$ in $\Delta'_{\Lambda}$. Suppose that for any $x\in S_1\setminus S_2$ and $y\in S_3\setminus S_2$, $x$ and $y$ are in different components of $\Lambda\setminus S_2$. Suppose $x_1\sim x_2$ and $x_2\sim x_3$. Then $x_1\sim x_3$. 
\end{lem}

\begin{lem}
	\label{lem:4-cycle}
	Given vertices $\{x_i\}_{i\in \mathbb Z/4\mathbb Z}$ in $\Delta'_\Lambda$ such that $x_i\sim x_{i+1}$ for each $i$. Then there is a vertex $x'_2$ of $\Delta'_\Lambda$ of the same type as $x_4$ such that $x'_2\sim x_i$ for $1\le i\le 3$. 
	
		Let $\{x_i\}_{i\in \mathbb Z/5\mathbb Z}$ be vertices of $\Delta'_\Lambda$ such
	that $x_i\sim x_{i+1}$ for each $i$. Then there exist vertices $x'_2,x'_3$ of
	$\Delta'_\Lambda$ of the same types as $x_2,x_3$, respectively, such that
	\[
	x'_3\sim\{x'_2,x_4,x_5\}
	\quad\text{and}\quad
	x'_2\sim\{x'_3,x_1,x_5\}.
	\]
\end{lem}

\begin{proof}
	Let $\bC$ be the Coxeter complex for $W_\Lambda$, with its barycentric subdivision denoted by $\bC'$. Let $\bar\pi:\Delta'_\Lambda\to \bC'$ be the simplicial map induced by action of the pure Artin group on $\Delta_\Lambda$. Let $\bar z_i=\bar\pi(x_i)$ and let $C_i$ be the standard subcomplex of $\Si_\Lambda$ associated with $\bar z_i$ (in the sense of Definition~\ref{def:associated subcomplex}). Let $\whC_i$ be the standard subcomplex of $\od_\Lambda$ associated with $C_i$.
	Let $L_i$ be the left coset of $A_\Lambda$ associated with $x_i$. Then $L_i\cap L_{i+1}\neq\emptyset$. Take $\ell_i\in L_{i-1}\cap L_i$, and let $w_i\in A_\Lambda$ be such that $\ell_{i+1}=\ell_iw_i$. Then $w_1w_2w_3w_4$ is trivial in $A_\Lambda$. This gives a null-homotopic edge loop in $\s_\Lambda$, which lifts a null-homotopic edge loop $P$ in $\od_\Lambda$. We can assume $P$ is a concatenation of four paths $P_1P_2P_3P_4$ such that each $P_i\subset \whC_i$ corresponds to $w_i$.
	
	
	Let $E_{i,j}=\Pi_{C_i}(C_j)$ (see Definition~\ref{def:projection1}). Let $I=\supp(E_{2,4})$ (Definition~\ref{def:label}). Let $I'$ be the union of irreducible components of $I$ that are not contained in $\supp(C_1)$, and $I''=I\setminus I'$. As $C_1$ has nonempty intersection with both $C_2$ and $C_4$, \cite[Lem 5.10]{huang2023labeled} implies that $I'\subset \supp(C_2)\cap \supp(C_4)$. The definition of $I''$ implies that $I''\subset \supp(C_2)\cap \supp(C_1)$. Let $E_{2,4}=E'\times E''$ with $\supp(E')=I'$ and $\supp(E'')=I''$. Let $\widehat E_{2,4},\widehat E',\widehat E''$ be the associated complexes in $\od_\Lambda$.
	
	As $P$ is null-homotopic in $\od_\Lambda$, $\Pi_{\widehat C_2}(P)=Q_1Q_2Q_3Q_4$ is null-homotopic in $\whC_2$, where $Q_i=\Pi_{\widehat C_2}(P_i)$. As $P_2=Q_2$, so $P_2$ is homotopic rel endpoints in $\whC_2$ to $\bar Q_1\bar Q_4\bar Q_3$ where $\bar Q_i$ is the inverse path of $Q_i$. By Lemma~\ref{lem:retraction property}, $Q_1\subset \widehat E_{2,1}=\whC_2\cap \whC_1$, and $Q_3\subset \whC_2\cap \whC_3$. By previous paragraph, $\bar Q_4$ is homotopic rel endpoints in $\widehat E_{2,4}\subset \whC_2$ to $\bar Q''_4\bar Q'_4$ with $\supp(Q''_4)\subset I''$ and $\supp(Q'_4)\subset I'$. As $\bar Q_4$ (hence $\bar Q''_4$) starts with a vertex in $\whC_2\cap\whC_1$, and $I''\subset \supp(C_2)\cap \supp(C_1)$, we know $\bar Q''_4\subset \whC_2\cap \whC_1$. Thus we can assume $P_2$ is a concatenation of path in $\whC_2\cap \whC_1$ (i.e. $\bar Q_1\bar Q''_4$), $\bar Q'_4$, and a path in $\whC_2\cap \whC_3$. Up to adjusting the choice of $\ell_i\in L_{i-1}\cap L_i$, we can assume $P_2=\bar Q'_4$. Then $\supp(P_2)\subset \supp(\whC_2)\cap\supp(\whC_4)$ by the previous paragraph. Hence there is a left coset $L=gA_{S\setminus T_4}$ (assuming $x_4$ has type $\widehat T_4$) containing both $\ell_2$ and $\ell_3$. In particular, $L\cap L_i\neq\emptyset$ for $i=1,2,3$, and it suffices to choose $x'_2$ to be the vertex corresponding to $L$.
	
	Now we prove the second part of the lemma. We can produce $\bar z_i, L_i,\ell_i,C_i,\whC_i,w_i,P_i$ from $\{x_i\}_{i\in \mathbb Z/5\mathbb Z}$ as before. Let $I_{53}=\supp(E_{5,3})$ and $I_{52}=\supp(E_{5,2})$. Let $I'_{53}$ be the union of irreducible components of $I_{53}$ that are not contained in $\supp(C_4)$, and $I''_{53}=I_{53}\setminus I'_{53}$. As $C_4$ has nonempty intersection with $C_3$ and $C_5$, \cite[Lem 5.10]{huang2023labeled} implies that $I'_{53}\subset \supp(C_5)\cap\supp(C_3)$. The definition of $I''_{53}$ implies that $I''_{53}\subset \supp(C_5)\cap\supp(C_4)$. Let $E_{5,3}=E'_{5,3}\times E''_{5,3}$ with $\supp(E'_{5,3})=I'_{53}$ and $\supp(E''_{5,3})=I''_{53}$. Let $\widehat E_{5,3},\widehat E'_{5,3},\widehat E''_{5,3}$ be the associated complexes in $\od_\Lambda$. Similarly, using $C_1$ has nonempty intersection with $C_5$ and $C_2$, we obtain product decomposition $\widehat E_{5,2}=\widehat E'_{5,2}\times \widehat E''_{5,2}$ such that $\supp(\widehat E'_{5,2})\subset \supp(C_5)\cap\supp(C_2)$ and $\supp(\widehat E''_{5,2})\subset \supp(C_5)\cap\supp(C_1)$.
	
Note that $\Pi_{\whC_5}(P)=Q_1Q_2Q_3Q_4Q_5$ is null-homotopic in $\whC_5$, where $Q_i=\Pi_{\whC_5}(P_i)$. As $P_5=Q_5$, $P_5$ is homotopic rel endpoints in $\whC_5$ to $\bar Q_4\bar Q_3\bar Q_2\bar Q_1$, with $\bar Q_4\subset \whC_5\cap \whC_4$ and $\bar Q_1\subset \whC_5\cap \whC_1$. By previous paragraph, $\bar Q_3$ is homotopic rel endpoints in $\widehat E_{5,3}$ to $\bar Q''_3\bar Q'_3$ with $\supp(\bar Q''_3)\subset \supp(C_5)\cap\supp(C_4)$ and $\supp(\bar Q'_3) \subset\supp(C_5)\cap\supp(C_3) $; and $\bar Q_2$ is homotopic rel endpoints in $\widehat E_{5,2}$ to $\bar Q'_2\bar Q''_2$ with $\supp(\bar Q'_2)\subset \supp(C_5)\cap\supp(C_2)$ and $\supp(\bar Q''_2)\subset \supp(C_5)\cap\supp(C_1)$. Thus we can assume $P_5$ is a concatenation of $\bar Q_4\bar Q''_3\subset \whC_5\cap\whC_4$, $\bar Q'_3\bar Q'_2$, and $\bar Q''_2\bar Q_1\subset \whC_5\cap\whC_1$. Up to adjusting the choice of $\ell_i\in L_{i-1}\cap L_i$, we assume $P_5=\bar Q'_3\bar Q'_2$. This gives an element $\ell'\in L_5$ such that $\ell_5$ and $\ell'$ differ by an element in $A_{S\setminus T_3}$, and $\ell'$ and $\ell_1$ differ by an element in $A_{S\setminus T_2}$ (assume $x_i$ has type $\widehat T_i$). This gives the desired $x'_3$ and $x'_2$.
\end{proof}

\section{Relation on sets}
\label{sec:relations}
\subsection{Posets}
\label{subsec:posets}
A poset $P$ is called \emph{weakly graded} if there is a \emph{poset map} $r:P\to \mathbb Z$, i.e.
such that for every $x<y$ in $P$, we have $r(x)<r(y)$: the map $r$ is called
a \emph{rank map}. A poset $P$ is \emph{weakly boundedly graded} if there is a rank
map $r:P\to\mathbb Z$ with finite image. Note that a finite weakly graded poset is
always weakly boundedly graded.
An \emph{upper bound} for a pair of elements $a,b\in P$ is an element $c\in P$ such
that $a\le c, b\le c$. A \emph{minimal upper bound} for $a,b$ is an upper bound $c$ such
that there does not exist upper bound $c'$ of $a,b$ such that $c'<c$. The \emph{join} of
two elements $a, b$ in $P$ is an upper bound $c$ of them such that for any other
upper bound $c'$ of $a, b$, we have $c\le c'$. We define \emph{lower bound}, \emph{maximal lower
bound}, and \emph{meet} similarly. In general, the meet or join of two elements in $P$
might not exist. A poset $P$ is a \emph{lattice} if any pair of elements have a meet
and a join.

We will mainly interested in posets arising from relative Artin complexes and their variations. Let $\Lambda$ be a Coxeter diagram, and let $\Lambda'\subset\Lambda$ be a linear induced subdiagram. We choose a linear order of the vertices $\{s_i\}_{i=1}^n$ of $\Lambda'$, and define a relation on the vertex set of $\Delta_{\Lambda,\Lambda'}$ as follows: for vertices $w,v\in \Delta_{\Lambda,\Lambda'}$, $v<w$ if $v$ and $w$ are adjacent in $\Delta_{\Lambda,\Lambda'}$ and the type $\hat s_i$ of $v$ and the type $\hat s_j$ of $w$ satisfy $i<j$. This relation depends on the choice of the linear order on $\Lambda'$.

We say an induced subdiagram $\Lambda'$ of a Coxeter diagram $\Lambda$ is \emph{admissible}, if for any vertex $s\in \Lambda'$, if $s_1,s_2\in \Lambda'$ are vertices in different connected components of $\Lambda'\setminus\{s\}$, then they are in different components of $\Lambda\setminus\{s\}$. It is shown that \cite[Lem 6.7]{huang2023labeled} that if $\Lambda'$ an admissible linear subdiagram of $\Lambda$, then for any choice of the linear order on $\Lambda'$, the associated relation on $\Delta^0_{\Lambda,\Lambda'}$ is a poset. 

A \emph{bowtie} $x_1y_1x_2y_2$ consists of distinct elements of $P$ satisfying $x_i<y_j$ for all $i,j=1,2$. 
\begin{definition}
	A poset $P$ is \emph{bowtie free} if for any bowtie $x_1y_1x_2y_2$ there exists $z\in P$ such that $x_i\le z\le y_j$ for all $i,j=1,2$.
\end{definition}

\begin{definition}(\cite[Def 6.8]{huang2023labeled})
	\label{def:bowtie free}
Suppose $\Lambda'$ is an admissible linear subdiagram of an Coxeter diagram $\Lambda$ with consecutive vertices of $\Lambda'$ being $\{s_i\}_{i=1}^n$. We define $\Delta_{\Lambda,\Lambda'}$ is \emph{bowtie free} if the poset defined on its vertex set as above is bowtie free. The property of being bowtie free does not depend on the choice of one of the two linear orders on $\Lambda'$.
\end{definition}


\subsection{Partial cyclic order}
\label{subsec:pco}
\begin{definition}
	\label{def:partial cyclic}
A \emph{partial cyclic order} on a set $X$ is a relation $C\subset X^3$, written as $[a,b,c]$, that satisfies the following axioms:
\begin{enumerate}
	\item Cyclicity: if $[a,b,c]$ then $[b,c,a]$;
	\item Asymmetry: if $[a,b,c]$ then not $[c,b,a]$;
	\item Transitivity: if $[a,b,c]$ and $[a,c,d]$ then $[a,b,d]$.
\end{enumerate}
This partial cyclic order is a \emph{cyclic order} if it satisfies an extra condition: if $a,b$ and $c$ are mutually distinct, then either $[a,b,c]$ or $[c,b,a]$.
\end{definition}
It follows from Definition~\ref{def:partial cyclic} that if $[a,b,c]$, then $\{a,b,c\}$ are mutually distinct.

Given a set $X$ with a linear order $<$, the cyclic order on $X$ induced by $<$ is defined as follows:
\begin{center}
$[a,b,c]$ if and only if $a<b<c$ or $b<c<a$ or $c<a<b$.
\end{center}
For example, the natural linear order on $\mathbb Z$ induces a cyclic order on $\mathbb Z$, which we refer to the \emph{canonical cyclic order} on $\mathbb Z$.

Given a partial cyclic order on $X$ and an element $x\in X$, we can define a binary relation $< $ on $(X\setminus\{x\})^2$ such that $x_1< x_2$ if $[x_1,x_2,x]$. One readily verifies that $(X\setminus\{x\},<)$ is a poset.

Sometimes we can ``glue'' posets to form partial cyclic ordered sets. We will be specifically interested in the following situation. Let $X$ be a set such that $X=X_1\cup X_2$ and $X_1\cap X_2=\{a,b\}$. Suppose $X_1$ and $X_2$ are also posets such that $b$ is the maximal element in $X_2$ and $a$ is the minimal element in $X_2$; and $b$ is the minimal element in $X_1$ and $a$ is the maximal element in $X_1$. 

Now we define a relation $C\subset X^3$, written as $[x,y,z]$ made of pairwise distinct triple satisfying one of the following conditions and all cyclic permutation of such triples:
\begin{enumerate}
	\item $x,y,z$ are in $X_2$ and $x< y< z$;
	\item $x,y,z$ are in $X_1$ and $x< y< z$;
	\item $x,y$ are in $X_2$ and $z$ is in $X_1$, and $x< y$;
	\item $x$ is in $X_2$ and $y,z$ are in $X_1$, and $y< z$.
\end{enumerate}

The following lemma is a definition chase.
\begin{lem}
	This ternary relation on $X$ is a partial cyclic order.
\end{lem}

\section{Minimal cuts in graphs}
\label{sec:minimal cuts}
All the graphs in this section are finite and simplicial.
Let $\Lambda$ be a graph. Let $A,B,C$ be sets of vertices of $\Lambda$. We say \emph{$B$ separates $A$ from $C$}, if for any $x\in A\setminus C$ and $y\in B\setminus C$, $x$ and $y$ are in different connected components of $\Lambda\setminus C$.

\begin{definition}
	\label{def:mincut}
Let $\Lambda$ be a graph, with two sets of vertices $A$ and $B$. A \emph{minimal cut between $A$ and $B$} is a set $C$ of vertices in $\Lambda$ such that
\begin{enumerate}
	\item $C$ separates $A$ from $B$;
	\item any proper subset $C'$ of $C$ does not separate $A$ from $B$.
\end{enumerate}
If $A\setminus C=\emptyset$ or $B\setminus C=\emptyset$, then it is considered that property (1) is satisfied. 

We will use $\mc_\Lambda(A,B)$ to denote the collection of all minimal cuts between $A$ and $B$, and we write $\mc(A,B)$ if $\Lambda$ is clear. 
\end{definition}

The set $\mc_\Lambda(A,B)$ is always non-empty, although it could possibly contain the empty subset of $\Lambda$. Also if $C\in \mc_\Lambda(A,B)$, then $A\cap B\subset C$.

Given a path $\omega$ in $\Lambda$, viewed as a map $f:[a,b]\to \Lambda$. Let $T\subset \Lambda$ be a subset. We say $\omega$ is \emph{$T$-tight}, if $f^{-1}(T)$ is a single point. For multiple subsets $T_1,T_2,\ldots$ of $\Lambda$, we say $\omega$ is \emph{$(T_1,T_2,\ldots)$-tight} if $\omega$ is $T_i$-tight for each $i$.

\begin{remark}
	\label{rmk:visit}
If $C\in \mc_\Lambda(A,B)$ and $x\in C$, then there exist a path $\omega$ from $a\in A$ to $b\in B$ such that $\omega\cap  (C\setminus\{x\})=\emptyset$. Note that $\omega$ must visit $x$, possibly multiple times. Up to removing certain subpaths of $\omega$, we can assume $\omega$ visit $x$ once. This gives a path from $a\in A$ to $b\in B$ visiting $x$ once, but does not visit any other vertices of $C$. Moreover, we can assume this path meets $A$ only in its starting point $a$, and meets $B$ only in its ending point $b$. In other words, this path is $(A,B,C)$-tight.
\end{remark}

The following is a slightly reformulation of Definition~\ref{def:mincut}.
\begin{lem}
	\label{lem:reformulate}
Let $\Lambda,A,B$ be as before. Then $C\in \mc_\Lambda(A,B)$ if and only if
\begin{enumerate}
	\item any path from $a\in A$ to $b\in B$ must have non-empty intersection with $C$;
	\item for any $x\in C$, there exist $a\in A$, $b\in B$ and a path from $a$ to $b$ outside $C\setminus \{x\}$.
\end{enumerate}
\end{lem}

%

\begin{lem}
	\label{lem:comparable}
Let $\Phi_1,\Phi_2\in \mc_\Lambda(A,B)$. Let $\Phi^A_i$ be the union of components of $\Lambda\setminus \Phi_i$ containing at least one element in $A\setminus \Phi_i$. Similarly we define $\Phi^B_i$. Then the following are equivalent:
\begin{enumerate}
	\item $\Phi^A_1\subset \Phi^A_2$;
	\item $\Phi^B_1\supset \Phi^B_2$;
	\item $\Phi_1$ separates $\Phi_2$ from $A$;
	\item $\Phi_2$ separates $\Phi_1$ from $B$.
\end{enumerate} 
\end{lem}

\begin{proof}
We first prove $(1)\Rightarrow(2)$.
We argue by contradiction and take $x\in \Phi^B_2\setminus \Phi^B_1$. Let $\omega$ be a shortest edge path from $x$ to a vertex $b'$ in $B\setminus \Phi_2$ outside $\Phi_2$. By the choice of $x$, $\omega\cap\Phi_1\neq\emptyset$. Let $y\in \Phi_1$ be the point in $\omega$ closest to $b'$, and let $\omega_0$ be the subpath of $\omega$ from $y$ to $b'$. 

By Remark~\ref{rmk:visit}, there is a path $\omega'$ from $a'\in A$ to $b'\in B$ which is $(A,B,\Phi_1)$-tight and meets $\Phi_1$ at $y$. Let $\omega''$ be the subpath of $\omega'$ from $a'$ until it hits $y$ the first time. Then either $\omega''$ is the trivial path (if $a'\in \Phi_1$) or $\omega''$ is contained in $\Phi^A_1$ except its endpoint $y$ (if $a'\notin \Phi_1$). In either case $\omega''\setminus\{y\}\subset\Phi^A_1\subset \Phi^A_2$. Then $\omega''\cup\omega_0$ is a path from $a'$ to $b'$ avoiding $\Phi_2$. This is a contradiction. 


$(2)\Rightarrow(1)$ is similar. Now we prove $(1)\Rightarrow(3)$. Suppose by contradiction that $x\in \Phi_2$ and $y\in A$ are connected by a path $\omega$ avoiding $\Phi_1$. We assume $x\notin B$, otherwise we have a contradiction immediately. Similar to the previous paragraph (using Remark~\ref{rmk:visit}), there is a path $\omega'$ from $b\in B$ to $x$ such that $\omega'\setminus\{x\}\subset \Phi^B_2$. As  $\Phi^B_2\subset \Phi^B_1$ by $(1)\Rightarrow(2)$, we obtain that $\omega\cup \omega'$ is a path from a point in $A$ to a point in $B$ avoiding $\Phi_1$, contradiction. Similarly, we can prove $(2)\Rightarrow(4)$.

For $(3)\Rightarrow(1)$, note that $(3)$ implies that $\Phi_2\setminus \Phi_1$ is disjoint from $\Phi^A_1$. As $\Phi_1\cap\Phi^A_1=\emptyset$, we obtain $\Phi^A_1\cap \Phi_2=\emptyset$. It follows that $\Phi^A_1\subset \Phi^A_2$. $(4)\Rightarrow(2)$ is similar.
\end{proof}

In the case of Lemma~\ref{lem:comparable}, we will say $\Phi_1$ and $\Phi_2$ are \emph{comparable} in $\mc_\Lambda(A,B)$.

\begin{lem}
	\label{lem:comparable1}
In the case of Lemma~\ref{lem:comparable} (1), the following hold:
\begin{enumerate}
	\item $\Phi_1\subset \Phi_2^A\cup \Phi_2$;
	\item if $\Phi'_1=\Phi_1\cap \Phi^A_2$, then $\Phi_1=\bar\Phi_1$ where $\bar\Phi_1$ is $\Phi'_1$ together with all vertices in $\Phi_2$ that can be connected to a vertex in $A$ by a $\Phi_2$-tight path in $(\Phi_2^A\cup\Phi_2)\setminus \Phi'_1$;
	\item $\Phi_1^A$ can be alternatively characterized as points in $\Phi_2^A$ that can be connected to a point in $A$ by a path in $\Phi_2^A\setminus \Phi'_1$.
\end{enumerate}
\end{lem}
\begin{proof}
	\label{rmk:comparable}
 For the first assertion, for any $x\in \Phi_1$, by Remark~\ref{rmk:visit}, there is an $(A,\Phi_1)$-tight path $\omega$ from $a\in A$ to $x\in\Phi_1$.  Then $\omega\setminus\{x\}\subset \Phi_1^A\subset \Phi_2^A$. If $x\notin \Phi_2$, then $x\in \Phi_2^A$. For Assertion 2, $\bar\Phi_1\subset\Phi_1$ is a consequence of $\Phi^A_1\subset \Phi^A_2$, and $\Phi_1\subset\bar\Phi_1$ is a consequence of Remark~\ref{rmk:visit} and Assertion 1. Assertion 3 is clear.
 \end{proof}




\begin{lem}
	\label{lem:latticecut}
We endow elements in $\mc_\Lambda(A,B)$ with a relation $<$ such that $\Phi_1<\Phi_2$ if $\Phi^A_1\subset \Phi^A_2$ and $\Phi_1\neq\Phi_2$. Then $(\mc_\Lambda(A,B),<)$ is a lattice. Moreover, given $\Phi_1,\Phi_2\in (\mc_\Lambda(A,B),<)$, both the join and meet of $\Phi_1$ and $\Phi_2$ are contained in $\Phi_1\cup\Phi_2$.
\end{lem}

\begin{proof}
The relation is transitive by definition.	
We will only prove $\Phi_1$ and $\Phi_2$ have the meet. Existence of join will be similar, due to the symmetry between $A$ and $B$ and Lemma~\ref{lem:comparable}. 

Let $E$ be the collection of points in $\Lambda$ that can be connected to a point in $A$ by a (possibly trivial) path outside $\Phi_1\cup\Phi_2$. Let $\partial E=\bar E\setminus E$ where $\bar E$ is the closure of $E$ in $\Lambda$ (it is possible that $E=\emptyset$, in which case $\partial E=\emptyset$). Set $$\Phi=\partial E\cup (\Phi_1\cap A)\cup(\Phi_2\cap A).$$
Note that $\partial E\cap B\subset \Phi_i\cap B$ for $i=1,2$. As $A\cap B\subset \Phi_i$ for $i=1,2$, we obtain $\Phi\cap B\subset\Phi_i\cap B$ for $i=1,2$. Moreover, $\Phi\subset\Phi_1\cup\Phi_2$. 


We claim $\Phi\in \mc_\Lambda(A,B)$. We first show any path $\omega\subset\Lambda$ from $a\in A\setminus \Phi$ to $b\in B\setminus \Phi$ has non-trivial intersection with $\Phi$. Note that $a\notin \Phi_i$ for $i=1,2$. If $\omega\cap\Phi=\emptyset$, then $\omega\subset E$ and $\omega\cap (\Phi_1\cup\Phi_2)=\emptyset$, contradicting that $\Phi_1\in \mc_\Lambda(A,B)$.


Next we show for any $x\in \Phi$, there exist $a_0\in A$, $b_0\in B$ and a path $\omega$ from $a_0$ to $b_0$ such that $\omega$ has empty intersection with $\Phi\setminus\{x\}$. Indeed, by definition of $\Phi$, we can find a path $\omega_1$ from $a_0\in A$ to $x$ intersecting $\Phi$ only at its endpoint $x$ (it is possible that $a_0=x$ and $\omega_1$ is the trivial path). Note that $x\in \Phi_1\cup\Phi_2$. Without loss of generality, we assume $x\in \Phi_1$. By Remark~\ref{rmk:visit}, we can find a path $\omega_2$ from $x$ to $b_0\in B\setminus \Phi_1$ such that $\omega_2$ only visits $\Phi_1$ at its starting point $x$. Let $\omega=\omega_1\cup\omega_2$. Then $\omega\cap \Phi_1=\{x\}$. Let $y\in \omega\cap \Phi_2$. By the choice of $\omega_1$, $y\in \omega_2$. If $y\in \Phi$ and $y\neq x$, then we can find a path $\omega'$ from $a'_0\in A$ to $y$ such that $\omega'$ meets $\Phi$ only in its endpoint $y$. As $y$ is in the interior of $\omega_2$, the subpath $\omega''$ of $\omega_2$ from $y$ to $b_0$ satisfies $\omega''\cap \Phi_1=\emptyset$. Then $\omega'\cup \omega''$ is a path from $a'_0\in A$ to $b_0\in B$ outside $\Phi_1$, contradiction. It follows that if $y\in \omega\cap \Phi_2$ satisfies $y\neq x$, then $y\notin \Phi$. As $\Phi\subset\Phi_1\cup\Phi_2$, it follows that $\omega$ meets $\Phi$ only at $\{x\}$, hence $\omega$ has empty intersection with $\Phi\setminus\{x\}$, as desired.
	
Lastly, we show if $\Phi_0\in \mc_\Lambda(A,B)$ satisfies $\Phi_0\leq \Phi_i$ for $i=1,2$, then $\Phi_0\le \Phi$. Indeed, as $\Phi^A_0\subset\Phi^A_i$ for $i=1,2$, we know points in $\Phi^A_0$ can be connected to a point in $A$ via a path outside $\Phi_1\cup \Phi_2$. In particular, $\Phi^A_0\subset E$. On the other hand, $E\subset \Phi^A$. Thus $\Phi^A_0\subset\Phi^A$.
\end{proof}

\begin{lem}
	\label{lem:squeez}
Given $\Phi_1,\Phi_2\in\mc_\Lambda(A,B)$ with $\Phi_1<\Phi_2$. If $\Phi\in \mc_\Lambda(\Phi_1,\Phi_2)$, then $\Phi\in \mc_\Lambda(A,B)$ and $\Phi_1\le \Phi\le\Phi_2$.
\end{lem}

\begin{proof}
Note that any path $\omega$ from $a\in A\setminus \Phi$ to $b\in B\setminus \Phi$ must have non-empty intersection with $\Phi$. Indeed, suppose $\omega$ meets $\Phi_1$ in vertex $u_1$, and meets $\Phi_2$ in vertex $u_2$. Then the subpath of $\omega$ between $u_1$ and $u_2$ must meet $\Phi$. 

We claim any $a\in A\setminus \Phi_1$ and any $c\in \Phi\setminus \Phi_1$ are in different connected components of $\Lambda\setminus \Phi_1$.  As $\Phi\in\mc_\Lambda(\Phi_1,\Phi_2)$,  there exist $u_1\in \Phi_1$, $u_2\in \Phi_2$ and a path $\omega_0$ from $u_1$ to $u_2$ outside $\Phi\setminus\{c\}$. Note that $c\in \omega_0$. 
Up to passing a sub-path, we can assume $\omega_0$ meets $\Phi_1$ only at its starting point $u_1$.
As $c\notin \Phi_1$ and $\Phi_1\cap\Phi_2\subset \Phi$, we obtain $\Phi_1\cap\Phi_2\subset \Phi\setminus\{c\}$. In particular, $u_2\notin \Phi_1$, and $\omega_0$ is not the trivial path. As $c\in \omega_0$, $c$ and $u_2$ are in the connected component of $\Lambda\setminus \Phi_1$. As $a$ and $u_2$ are in different components of $\Lambda\setminus \Phi_1$ by Lemma~\ref{lem:comparable}, the claim follows. Similarly, any $c\in \Phi\setminus \Phi_2$ and $b\in B\setminus \Phi_2$ are in different connected components of $\Lambda\setminus \Phi_2$.

Given $c\in \Phi$, we show there exists $a_0\in A$, $b_0\in B$ and a path $\omega$ from $a_0$ to $b_0$ avoiding $\Phi\setminus\{c\}$. By Remark~\ref{rmk:visit}, there exist $u_1\in \Phi_1,u_2\in\Phi_2$, and a $(\Phi_1,\Phi,\Phi_2)$-tight path $\omega_0$ from $u_1$ to $u_2$ visiting at $\Phi$ once at $c$. By Remark~\ref{rmk:visit}, there is an $(A,\Phi_1)$-tight path $\omega_1$ from $a_0\in A$ to $u_1$, and a $(\Phi_2,B)$-tight path $\omega_2$ from $u_2$ to $b_0\in B$. By previous paragraph, $\omega_1\setminus\{u_1\}$ and any point in $\Phi\setminus\Phi_1$ are in different connected components of $\Lambda\setminus \Phi_1$. In particular $(\omega_1\setminus\{u_1\})\cap \Phi=\emptyset$. Similarly, $(\omega_2\setminus\{u_2\})\cap\Phi=\emptyset$. Thus the concatenation $\omega_1\omega_0\omega_2$ is the path as desired. It follows that $\Phi\in \mc_\Lambda(\Phi_1,\Phi_2)$. The previous paragraph implies that $\Phi_1\le \Phi\le \Phi_2$.
\end{proof}

\begin{lem}
	\label{lem:separate}
	Given $\Phi_1,\Phi_2,\Phi_3\in\mc_\Lambda(A,B)$ with $\Phi_1<\Phi_2<\Phi_3$. Then $\Phi_2$ separates $\Phi_1$ from $\Phi_3$.
\end{lem}

\begin{proof}
Take $v\in \Phi_1\setminus \Phi_2$ and $u\in \Phi_3\setminus \Phi_2$. By Remark~\ref{rmk:visit}, there is a path $\omega$ from $a\in A$ to $v$ such that $\omega\setminus\{v\}\subset \Phi^A_1$ (we do allow $\omega=\{v\}$), and there is a path $\omega'$ from $u$ to $b\in B$ with $\omega'\setminus\{u\}\subset \Phi^B_3$. Lemma~\ref{lem:comparable} implies that $\omega\subset \Phi^A_2$ and $\omega'\subset \Phi^B_2$. If $v$ and $u$ are connected by a path $\omega_0$ outside $\Phi_2$, then the concatenation $\omega\omega_0\omega'$ connects $a\in A$ to $b\in B$ outside $\Phi_2$, contradiction. This finishes the proof.
\end{proof}

\section{The labeled 4-cycle condition}
\label{sec:labeled 4-cycle}
\begin{definition}
	\label{def:labaled}
We say $\Delta_\Lambda$ satisfies \emph{labeled 4-cycle condition}, if for any embedded 4-cycle $x_1x_2x_3x_4$ in $\Delta_\Lambda$  of vertex type $\hat s\hat t \hat s\hat t$, and any connected induced subgraph $\Lambda'\subset \Lambda$ containing both $s$ and $t$, there is a vertex $z$ such that
\begin{enumerate}
\item $z$ is adjacent to each $x_i$ for $1\le i\le 4$. 
\item $z$ has type $\hat r$ with $r\in \Lambda'$.
\end{enumerate}
\end{definition}


Recall generalized cycles are defined in Section~\ref{subsec:generalized cycles}.

\begin{definition}
	\label{def:strong labaled}
	We say $\Delta_\Lambda$ satisfies \emph{strong labeled 4-cycle condition}, if the following holds.
	Let $A,B$ be two sets of vertices of $\Lambda$. Then for any generalized 4-cycle $x_1x_2x_3x_4$ in $\Delta'_\Lambda$ of type $\hat A,\hat B,\hat A,\hat B$ respectively, there exists $C\in \mc_\Lambda(A,B)$ and a vertex $y\in \Delta'_\Lambda$ of type $\hat C$ such that $y\sim x_i$ for each $i$.
\end{definition}

\begin{definition}
	\label{def:admissible}
Let $\mathcal Q$ be a subset of $2^{V\Lambda}$, i.e. the power set of the vertex set of $\Lambda$, endowed with a relation $<$. We say $(\mathcal Q,<)$ is \emph{admissible} if the following are true:
\begin{enumerate}
	\item the relation $<$ is a poset;
	\item if $A,B\in \mathcal Q$ are comparable with respect to the relation $<$, then any element of $\mc_\Lambda(A,B)$ is contained in $\mathcal Q$; moreover, if $C\in \mc_\Lambda(A,B)$, then either $A\le C\le B$ or $B\le C\le A$;
	\item if $A<B<C$ in $\mathcal Q$, then $B$ separates $A$ from $C$;
	\item two elements $A,B$ of $\mathcal Q$ with a common upper bound have the join $C$ with $C\subset A\cup B$, and two elements $A,B$ of $\mathcal Q$ with a common lower bound have the meet $C$ with $C\subset A\cup B$.
\end{enumerate}
\end{definition}

\begin{lem}
	\label{lem:mincut admissible}
	Let $\mathcal Q=\mc_\Lambda(A,B)$ endowed with the relation $<$ in Lemma~\ref{lem:latticecut}. Then $(\mathcal Q,<)$ is admissible.
\end{lem}

\begin{proof}
 Definition~\ref{def:admissible} (1) (4) follow from Lemma~\ref{lem:latticecut}. Definition~\ref{def:admissible} (2) follows from Lemma~\ref{lem:squeez}.  Definition~\ref{def:admissible} (3) follows from Lemma~\ref{lem:separate}.
\end{proof}

\begin{lem}
	\label{lem:admissible}
Let $(\mathcal Q,<)$ be admissible as in Definition~\ref{def:admissible}. Let $\mathcal P$ be the collection of vertices in $\Delta'_\Lambda$ with type $\hat C$ such that $C\in \mathcal Q$. For $x\neq y\in \mathcal P$ of type $\hat T_x,\hat T_y$, we define $x<y$ if $x\sim y$ and $T_x<T_y$. 

Suppose $\Delta_\Lambda$ satisfies the strong labeled 4-cycle condition. Then the relation $<$ on $\mathcal P$ is a weakly boundedly graded poset such that any upper bounded pair has the meet and any lower bounded pair has the join.
\end{lem}

\begin{proof}
It follows from Lemma~\ref{lem:transitive} and Definition~\ref{def:admissible} (3) that $\mathcal P$ is a poset.	
As $\mathcal Q$ is a poset with finitely many elements, $\mathcal Q$ is weakly graded. There is a natural map $f:\mathcal P\to \mathcal Q$ by considering types of elements in $\mathcal P$. Note that $x<y$ implies $f(x)<f(y)$. Thus $\mathcal P$ is weakly graded.

Let $r:\mathcal P\to \mathbb Z$ be a rank function that factors through a rank function $\mathcal Q\to \mathbb Z$ that is also denoted by $r$. We first show that given $x,y\in \mathcal P$ with an upper bound such that $T_x=T_y$, $x$ and $y$ have the join. Let $z\in \mathcal P$ such that $r(z)$ is smallest among all the common upper bounds of $x$ and $y$. We claim $z$ is the join of $x$ and $y$. Let $w$ be an upper bound of $x$ and $y$. By Lemma~\ref{lem:4-cycle}, there is vertex $z'$ of $\Delta'_\Lambda$ of the same type as $z$ such that $w\sim z'$, $x\sim z'$ and $y\sim z'$. 

We first consider the case $z=z'$. Then $z\sim w$. Let $T_u$ be the meet of $T_z$ and $T_w$ in $\mathcal Q$ - it exists by Definition~\ref{def:admissible} (4) and $T_u\subset T_z\cup T_w$.
Let $u$ be the vertex of $\Delta'_\Lambda$ such that $u\sim z$, $u\sim w$. As $z,w,x$ are contained in a single simplex of $\Delta_\Lambda$, we obtain $u\sim x$. Similarly, $u\sim y$. As $T_u$ is an upped bound of $\{T_x,T_y\}$, we know $u$ is an upped bound of $\{x,y\}$. By the choice of $u$, $r(u)=r(T_u)\le r(T_z)=r(z)$. On the other hand, $r(z)\le r(u)$. So $r(T_u)=r(T_z)$, implying $T_u=T_z$. As $u\sim z$, we obtain $u=z$. As $T_u\le T_w$, we obtain $z\le w$, as desired.

Now we consider the case $z\neq z'$. We will show this case is impossible by producing an upper bound $z_0\in \mathcal P$ of $\{x,y\}$ such that $r(z_0)<r(z)$. By the strong labeled 4-cycle condition, there is a vertex $z_0\in \Delta'_\Lambda$ with $z_0\sim \{ x,y, z,z'\}$ and $T_{z_0}\in \mc_\Lambda(T_x,T_z)$. By Definition~\ref{def:admissible} (2), $T_{z_0}\in \mathcal Q$ and $T_x\le T_{z_0}\le T_z$. Thus $z_0\in \mathcal P$ and $\{x,y\}\le z_0\le z$. As $z\neq z'$, $z_0<z$, hence $r(z_0)<r(z)$, as desired.

Similarly we know for any $x,y\in \mathcal P$ with a lower bound such that $T_x=T_y$, $x$ and $y$ have the meet. Now we consider the case $T_x\neq T_y$ and $x,y$ have a common upper bound. We define $z$ as before, and claim $z$ is the join of $x$ and $y$. Let $w$ and $z'$ be as before. The case $z=z'$ follows from the same argument as before. Suppose $z\neq z'$. As $T_{z}=T_{z'}$, we know $z$ and $z'$ have the meet, denoted $z_0$. Then $z_0$ is a common upper bound of $\{x,y\}$, and $r(z_0)<r(z)$, which is a contradiction.
\end{proof}

The following is a consequence of Lemma~\ref{lem:mincut admissible} and Lemma~\ref{lem:admissible}.

\begin{cor}
	\label{cor:strong}
	Let $\Delta_\Lambda, A, B$ be as in Definition~\ref{def:strong labaled}. Let $\mathcal P$ be the collection of vertices in $\Delta'_\Lambda$ which have type $\hat C$ for some $C\in \mc_\Lambda(A,B)$. For $x\neq y\in \mathcal P$ of types $\hat T_x,\hat T_y$, we define $x<y$ if $x\sim y$ and $T^A_x\subset T^A_y$, where $T^A_x$ is defined in Lemma~\ref{lem:comparable}.
	
	Suppose $\Delta_\Lambda$ satisfies the strong labeled 4-cycle condition. Then the relation $<$ on $\mathcal P$ is a weakly boundedly graded poset such that any upper bounded pair has the meet and any lower bounded pair has the join.
\end{cor}

\begin{lem}
	\label{lem:graph}
Let $\Lambda$ be a connected graph, and $X$ be a finite set of vertices in $\Lambda$. Then there exists $x\in X$ such that $X\setminus \{x\}$ are contained in the same connected component of $\Lambda\setminus\{x\}$. Moreover, if $|X|>1$ and we are given $x_0\in X$, then we can choose such $x$ with $x\neq x_0$.
\end{lem}

\begin{proof}
Take arbitrary $x\in X$. It suffices to show that if one of the components $C$ of $\Lambda\setminus\{x\}$ contains exactly $k$ elements from $X\setminus\{x\}$ with $k<|X|-1$, then we can find $x'\in X\setminus\{x\}$, such that $\Lambda\setminus \{x'\}$ has a component containing all elements in $(C\cap X)\cup\{x\}$. Indeed, as $k<|X|-1$, we know $\Lambda\setminus\{x\}$ has a component $C'\neq C$ such that $C'$ has at least one element from $X$. 
Let $x'\in X\cap C'$. We now show $(C\cap X)\cup\{x\}$ is contained in a single component of $\Lambda\setminus\{x'\}$. As $\Lambda$ is connected, each element in $C$ can be connected to $x$ via a path in $\Lambda$. Up to passing subpaths, we can assume such path does not visit $x$ except at its endpoint. Thus each element in $C$ can be connected to $x$ via a path that is outside $C'$, hence outside $\{x'\}$. Moreover, $x$ is connected to $x'$ via a path visiting $x'$ only at its endpoint. Thus $(C\cap X)\cup\{x\}$ is contained in a single component of $\Lambda\setminus\{x'\}$, as desired.
\end{proof}

\begin{prop}
	\label{prop:strong}
Let $\Lambda$ be a connected Coxeter diagram such that for any induced subdiagrams $\Lambda'$ of $\Lambda$, $\Delta_{\Lambda'}$ satisfies the labeled 4-cycle property. Then $\Delta_\Lambda$ satisfies the strong labeled 4-cycle property. 
\end{prop}

\begin{proof}
We induct on the number of vertices in $\Lambda$. The base case when $\Lambda$ has one vertex is clear. Now suppose the proposition holds for all connected Dynkin diagram with number of vertices $\le |\Lambda|-1$. Let $A,B$ be two sets of vertices of $\Lambda$. Consider a collection of vertices $x_1x_2x_3x_4$ in $\Delta'_\Lambda$ of type $\hat A,\hat B,\hat A,\hat B$ respectively such that $x_i\sim x_{i+1}$ for $i\in \mathbb Z/4\mathbb Z$. We aim to find $C\in \mc_\Lambda(A,B)$ and a vertex $y\in \Delta'_\Lambda$ of type $\hat C$ such that $y\sim x_i$ for each $i$.
	
We employ an inner layer of induction on the number of vertices in $A$ and $B$. If $|A|=1$ and $|B|=1$, let $\omega_1$ be a minimal length path connecting $A$ and $B$. By Definition~\ref{def:labaled}, there is a vertex $z_1\in \Delta_\Lambda$ adjacent to each of $\{x_i\}_{i=1}^4$ such that $z_1$ is of type $\hat s_1$ with $s_1\in \omega_1$. Now consider $\lk(z_1,\Lambda)$, which is a copy of $\Delta_{\Lambda\setminus\{s_1\}}$. If $A$ and $B$ are in different connected components of $\Lambda\setminus\{s_1\}$, then we are done. Otherwise we take a minimal length path $\omega_2$ connecting $A$ and $B$. As $\Delta_{\Lambda\setminus\{s_1\}}$ satisfies the labeled 4-cycle condition, there is a vertex $z_2\in\Delta_{\Lambda\setminus\{s\}}\cong\lk(z_1,\Lambda)$ adjacent to each of $\{x_i\}_{i=1}^4$ such that $z_2$ has type $s_2$ with $s_2\in \omega_2$. By repeating this finitely many times, we obtain $C=\{s_1,\ldots,s_k\}\in \mc_\Lambda(A,B)$ and vertices $\{z_1,\ldots z_k\}\in \Delta_{\Lambda}$ spanning a simplex such that $z_i$ has type $\hat s_i$ and $z_i$ is adjacent to each of $\{x_i\}_{i=1}^4$, as desired. 

Suppose at least one of $|A|$ and $|B|$ is $>1$. By Lemma~\ref{lem:graph}, there is an element of $A\cup B$, say $a\in A$ such that $\Lambda\setminus\{a\}$ contains $(A\cup B)\setminus\{a\}$ in a single connected component. The moreover part of Lemma~\ref{lem:graph} implies that we can assume $A\setminus\{a\}\neq\emptyset$. Let $\Lambda_a$ be the full subgraph of $\Lambda$ spanned by vertices in the component of $\Lambda\setminus\{a\}$ containing $(A\cup B)\setminus\{a\}$. 

We first consider the case $a\in A\cap B$. Then there is a vertex $y\in \Delta_\Lambda$ of type $\hat a$ such that $y\sim x_i$ for each $i$. Let $\bar x_1,\bar x_2,\bar x_3,\bar x_4$ be vertices of type $\hat A_1,\hat B_1,\hat A_1,\hat B_1$ with $A_1=A\setminus\{a\}$ and $B_1=B\setminus\{a\}$ such that $\bar x_i\sim x_i$ for each $i$. Note that $\{\bar x_i\}_{i=1}^4$ can be viewed as vertices in $\Delta'_{\Lambda_a}$. By our assumption, $\Delta_{\Lambda_a}$ also satisfies the labeled 4-cycle condition. As $\Lambda_a$ is connected, by induction we know there is $\bar y\in \Delta'_{\Lambda_a}$ of type $\hat C_1$ such that $\bar y\sim \bar x_i$ for each $i$ and $C_1\in\mc_{\Lambda_a}(A_1,B_1)$.  As $y\sim \{\bar y,\bar x_1,\bar x_2,\bar x_3,\bar x_4\}$, $y$ and $\bar y$ determine a vertex $y_0\in \Delta'_\Lambda$ of type $\hat C$ such that $y_0\sim x_i$ for $1\le i\le 4$ and $C=C_1\cup \{a\}$. Note that $C\in \mc_\Lambda(A,B)$.

Now we assume $a\notin B$. We refer to Figure~\ref{fig:strong} for the following discussion.
For $i=1,3$, let $x'_i\in \Delta'_\Lambda$ be the vertex of type $\hat a$ with $x'_i\sim x_i$, let $x''_i\in \Delta'_\Lambda$ be the vertex of type $\hat A_1$ with $x''_i\sim x_i$. 
By induction assumption, there is a vertex $z\in \Delta'_\Lambda$ of type $\hat T_z$ such that $z\sim x_i$ for $i=2,4$, $z\sim x'_i$ for $i=1,3$ and $T_z\in\mc_\Lambda(\{a\},B)$. If $a\in T_z$, then $T_z=\{a\}$ and $x'_1=x'_3=z$. As $x'_1\sim\{x''_1,x_2,x''_3,x_4\}$ and $a\notin B,A_1$, we can view $x''_1,x_2,x''_3,x_4$ as vertices in $\Delta'_{\Lambda_a}$. Note that $\Delta_{\Lambda_a}$ also satisfies the labeled 4-cycle condition. Similar as the previous paragraph, there exist $C_1\in \mc_{\Lambda_a}(A_1,B)$ and a vertex $\bar y\in \Delta'_{\Lambda_a}$ of type $\hat C_1$ such that $\bar y\sim\{x''_1,x_2,x''_3,x_4\}$. Thus $\bar y\sim x_i$ for all $i$. If $C_1\in \mc_\Lambda(A,B)$, then we are done. If not, then $C\in \mc_\Lambda(A,B)$ for $C=C_1\cup\{a\}$. Note that
$\bar y$ and $z$ determine a vertex $y\in \Delta'_\Lambda$ of type $\hat C$ such that $y\sim x_i$ for all $i$, as desired.

\begin{figure}
	\centering
	\includegraphics[scale=1]{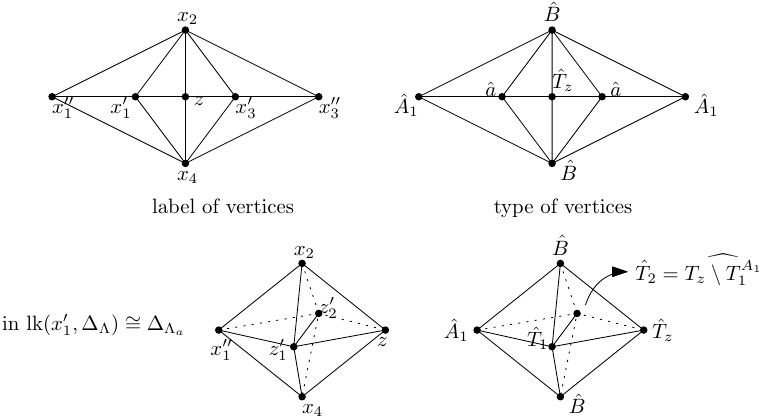}
	\caption{Two vertices are joined by an edge if they satisfy the relation $\sim$.}
	\label{fig:strong}
\end{figure}

We now assume $a\notin T_z$. It suffices to find $C\in \mc_\Lambda(A,B)$ and a vertex $y\in \Delta'_\Lambda$ of type $\hat C$ such that $y\sim \{x'_1,x''_1,x_2,x'_3,x''_3,x_4\}$. We will do it in two steps - first we produce such $y$ such that $y\sim\{x'_1,x''_1,x_2,x'_3,x_4\}$.

Note that $x'_1\sim x''_1,x_2,z,x_4$. As $a\notin A_1,B,T_z$, we can view $x''_1,x_2,z,x_4$ as vertices in $\Delta'_{\Lambda_a}$. By our assumption $\Delta_{\Lambda_a}$ also satisfies the labeled 4-cycle condition. By induction, $\Delta_{\Lambda_a}$ satisfies the strong labeled 4-cycle condition. Let $\mathcal P_1$ be the collection vertices in $\Delta'_{\Lambda_a}$ whose types belong to $\mc_{\Lambda_a}(A_1,B)$. We endow $\mathcal P_1$ with a poset structure as in Corollary~\ref{cor:strong} such that $x_2<x''_1$.
Let $z'_1$ be the join of $x_2$ and $x_4$ in $\mathcal P_1$, which exists by Corollary~\ref{cor:strong}. We denote the type of $z'_1$ by $\hat T_1$. In particular, $T_1\in \mc_{\Lambda_a}(A_1,B)$. 

Now we show $z'_1\sim z$. Indeed, by Lemma~\ref{lem:4-cycle} applied to the generalized 4-cycle $z'_1x_2zx_4$ in $\Delta'_{\Lambda_a}$, there is a vertex $z''_1\in \Delta'_{\Lambda_1}$ of type $\hat T_1$ such that $z''_1\sim\{x_2,z,x_4\}$. By the choice of $z'_1$, we know $z'_1=z''_1$. Thus $z'_1\sim z$.

\begin{claim*}
In $\Lambda$ we have $T_z$ separates $T_1$ from $\{a\}$. 
\end{claim*}

\begin{proof}[Proof of the claim]
 By contradiction we assume $a$ is connected to $c\in T_1\setminus T_z$ by a path $\omega$ outside $T_z$. Let $T^B_1$ be the union of components of $\Lambda_a\setminus T_1$ that contains a vertex in $B\setminus T_1$, and let $c^+$ be the collection of points in the $\epsilon$-sphere of $c$ that are contained in $T^B_1$. Next we show there is a path $\omega'$ in $T^B_1$ from a point in $t^+$ to $b\in B$ avoiding $T_z$. 
If this is not true, we can find $S_z\subset T_z$ such that $S_z\in \mc_{ T^B_1}(c^+,B)$. 
Let $$T=S_z\cup \{t\in T_1 \mid \mathrm{a\ point\ in\ }t^+\ \mathrm{is\ connected\ to\ a\ vertex\ in}\ B\ \mathrm{in}\  T^B_1\setminus S_z\}$$
Then $T\in \mc_{\Lambda_a}(B,A_1)$ and $T\le T_1$ in $\mc_{\Lambda_a}(B,A_1)$ - this is a consequence of the following three observations and Lemma~\ref{lem:comparable}.
\begin{enumerate}
	\item Any path in $\Lambda_a$ from a point in $A_1$ to a point in $B$ must meet $T$. Indeed, this path must meet $T_1$, and its subpath from the last time it meets $T_1$ to its endpoint in $B$ is contained in $T^B_1$. So if this subpath is outside $T\setminus S_z$, then one of its endpoint is in $T$.
	\item For any $t\in T$, there is $a_t\in A_1$ and $b_t\in B$ such that $a_t$ is connected to $b_t$ via a path in $\Lambda_a$ outside $T\setminus\{t\}$. Indeed, if $t\in S_z$, By Remark~\ref{rmk:visit}, there is path $\omega_1\subset T^B_1$ from a point in $c^+$ to $b_t\in B$ outside $S_z\setminus \{t\}$. As $c\in T_1$ and $T_1\in \mc_{\Lambda_a}(A_1,B)$, by Remark~\ref{rmk:visit}, there is an $(A_1,T_1)$-tight path $\omega_2\subset \Lambda_a$ from $a_t\in A_1$ to $c\in T_1$. By filling the small gap between $\omega_2$ and $\omega_1$, we obtain a path outside $T\setminus\{t\}$. If $t\in T\setminus S_z$, then there is a path $\omega'_1\subset T^B_1$ from a point in $t^+$ to $b_t\in B$ outside $S_z$. We define $\omega_2$ as before and obtain a path outside $T\setminus\{t\}$ by filling the gap between $\omega_2$ and $\omega'_1$.
	\item  Any element in $T_1\setminus T$ and any element in $B\setminus T$ are contained in different connected components of $\Lambda_a\setminus T$. Indeed, for any path in $\Lambda_a$ from a point in $T_1\setminus T$ to a point in $B\setminus T$, consider the subpath from the last time this path meets $T_1$ to its endpoint in $B$. Then either this subpath meets $S_z$, or it starts at a point in $T$.
\end{enumerate}
Note that each of $\{z'_1,z\}$ and each of $\{x_2,x_4\}$ are contained in a common simplex of $\Delta_{\Lambda_a}$, and $z'_1\sim z$. As $T\subset T_z\cup T_1$, we know $z'_1$ and $z$ determine a vertex of type $\hat T$ in $\Delta'_{\Lambda_a}$ which is $\sim$ to each of $x_2$ and $x_4$. This contradicts the choice of $z'_1$ (and $T_1$) and justifies the existence of $\omega'$. By filling the small gap between $\omega$ and $\omega'$, we obtain a path from $a$ to $b\in B$ avoiding $T_z$, contradicting the choice of $T_z$. Hence the claim is proved.
\end{proof}

Let $T_1^{A_1}$ be the union of components of $\Lambda_a\setminus T_1$ that contains a vertex in $A_1\setminus T_1$, and let $T_2=T_z\setminus T_{1}^{A_1}$. If $T_2\neq\emptyset$, then let $z'_2$ be the vertex in $\Delta'_{\Lambda}$ (also viewed as a vertex in $\Delta'_{\Lambda_a}$) of type $\hat T_2$ such that $z\sim z'_2$. Next we show
\begin{enumerate}
	\item any path in $\Lambda$ from a point in $A$ to a point in $B$ intersects $T_1\cup T_2$;
	\item $z'_1\sim\{x''_1,z,z'_2,x'_3\}$, and $z'_2\sim\{x''_1,x'_3\}$ (if $z'_2$ exists).
\end{enumerate}
For (1), given $x\in A$ and $y\in B$, and a path $\omega$ in $\Lambda$ from $x$ to $y$. If $a\notin \omega$, then $\omega\subset \Lambda_a$ and $\omega\cap T_1\neq\emptyset$ by the definition of $T_1$. Now suppose $a\in \omega$. Up to passing to a subpath, we assume $\omega$ starts with $a$, and never returns to $a$. Then $\omega\cap T_z\neq\emptyset$ by the definition of $T_z$. Take $w\in T_z\cap \omega$, if $w\in T_2$, then we are done. Now suppose $w\notin T_2$. Then $w\in T^{A_1}_1$. Note that the subpath of $\omega$ from $w$ to $y\in B$ is contained in $\Lambda_a$ by our choice of $\omega$. Thus this subpath must meet  $T_1$, which proves (1).

For (2), the proof of $z'_1\sim \{x''_1,z'_2\}$ is similar to $z'_1\sim z$ before. As $z'_1\sim z$ and $z\sim x'_3$, the above claim and Lemma~\ref{lem:transitive} imply $z'_1\sim x'_3$. Now we consider $z'_2$. Note that $z'_2\sim x'_3$ is a consequence of $z\sim x'_3$. As $z'_2\sim z'_1$, $z'_1\sim x''_1$ and any point in $T_2\setminus T_1$ and any point in $A_1\setminus T_1$ are in different components of $\Lambda_a\setminus T_1$, by Lemma~\ref{lem:transitive}, $z'_2\sim x''_1$. Thus (2) is proved.

(1) and (2) imply that each of $\{z'_1,z'_2\}$ is $\sim$ to each of $\{x_2,x_4,x''_1,x'_1,x'_3\}$, and $T_1\cup T_2$ contain a subset $T_y\subset \mc_\Lambda(A,B)$. Thus $z'_1,z'_2$ determine vertex $y\in \Delta'_\Lambda$ such that $y$ has type $\hat T_y$. This finishes step 1. By switching the role of $x_1$ and $x_3$, we can assume $y\sim \{x_2,x_4,x_3,x'_1\}$. Now we repeat the previous argument, with the role of $z,T_z$ replaced by $y,T_y$, to produce the desired vertex of $\Delta'_\Lambda$ (now the correct version of the above Claim in this new setting should be that $T_y$ separates $T_1$ from $A$ in $\Lambda$, however, this can be proved in the same way as the above Claim using the stronger assumption that $T_y\subset \mc_\Lambda(A,B)$).
\end{proof}

\section{Garside categories, $\widetilde A_n$-like complexes and Bestvina non-positive curvature}
\label{sec:garside}
\subsection{Background in Garside category}
Let $\mathcal C$ be a small category. One may think of $\mathcal C$ as of an oriented graph, whose vertices are objects in $\mathcal C$ and oriented edges are morphisms of $\mathcal C$. Arrows in $\mathcal C$ compose like paths: $x\stackrel{f}{\to} y\stackrel{g}{\to} z$ is composed into $x\stackrel{fg}{\to} z$. For objects $x,y\in\mathcal C$, let $\mathcal C_{x\to}$ denote the collection of morphisms whose source object is $x$. Similarly we define $\mathcal C_{\to y}$ and $\mathcal C_{x\to y}$. 

For two morphisms $f$ and $g$, we define $f\preccurlyeq g$ if there exists a morphism $h$ such that $g=fh$. Define $g\succcurlyeq f$ if there exists a morphism $h$ such that $g=hf$. Then $(\mathcal C_{x\to},\preccurlyeq)$ and $(\mathcal C_{\to y},\succcurlyeq)$ are posets. A nontrivial morphism $f$ which cannot be factorized into two nontrivial factors is an \emph{atom}.

The category $\mathcal C$ is \emph{cancellative} if, whenever a relation $afb=agb$ holds between composed morphisms, it implies $f=g$. $\mathcal C$ is \emph{homogeneous} if there exists a length function $l$ from the set
of $\mathcal C$-morphisms to $\mathbb Z_{\ge 0}$ such that $l(fg) = l(f) + l(g)$ and $(l(f) = 0)\Leftrightarrow$ ($f$ is a unit).

We consider the triple $(\mathcal C,\mathcal C\stackrel{\phi}{\to}\mathcal C,1_\mathcal C \stackrel{\Delta}{\Rightarrow}\phi)$ where $\phi$ is an automorphism of $\mathcal C$ and $\Delta$ is a natural transformation from the identity function to $\phi$. For an object $x\in \mathcal C$, $\Delta$ gives morphisms $x\stackrel{\Delta(x)}{\longrightarrow} \phi(x)$ and $\phi^{-1}(x)\stackrel{\Delta(\phi^{-1}(x))}{\longrightarrow} x$. We denote the first morphism by $\Delta_x$ and the second morphism by $\Delta^x$. A morphism $x\stackrel{f}{\to} y$ is \emph{simple} if there exists a morphism $y\stackrel{f^\ast}{\to} \phi(x)$ such that $f f^\ast=\Delta_x$. When $\mathcal C$ is cancellative, such $f^\ast$ is unique. 

\begin{definition}[\cite{bessis2006garside}]
	\label{def:Garside}
	A \emph{homogeneous categorical Garside structure} is a triple $(\mathcal C,\mathcal C\stackrel{\phi}{\to}\mathcal C,1_\mathcal C \stackrel{\Delta}{\Rightarrow}\phi)$ such that:
	\begin{enumerate}
			\item $\phi$ is an automorphism of $\mathcal C$ and $\Delta$ is a natural transformation from the identity function to $\phi$;
			\item $\mathcal C$ is homogeneous and cancellative;
			\item all atoms of $\mathcal C$ are simple;
			\item for any object $x$, $\mathcal C_{x\to}$ and $\mathcal C_{\to x}$ are lattices.
		\end{enumerate}
\end{definition}

\begin{definition}
	A \emph{Garside category} is a category $\mathcal C$ that can be equipped with $\phi$ and $\Delta$ to obtain a homogeneous categorical Garside structure. A \emph{Garside groupoid} is the enveloping groupoid of a Garside category. Informally speaking, it is a groupoid obtained by adding formal inverses to all morphisms in a Garside category.
\end{definition}

A fundamental property of $\mathcal C$ is that the natural map $\mathcal C\to\mathcal G$ is an embedding, where $\mathcal G$ denotes the enveloping groupoid, as follows from the discussion in \cite[Section 2]{bessis2006garside}.

\begin{thm}(\cite[Section 2]{bessis2006garside})
	\label{thm:normal form0}
	Consider homogeneous categorical Garside structure $(\mathcal C,\mathcal C\stackrel{\phi}{\to}\mathcal C,1_\mathcal C \stackrel{\Delta}{\Rightarrow}\phi)$, and let $\mathcal G$ be the associated Garside groupoid. For each morphism $f\in \mathcal G$, there is a unique way of writing $f$ as $f=s_1s_2\cdots s_l\Delta^k$ where $s_1,s_2,\ldots, s_l$ are simple elements in $\mathcal C$ with sources $x_1,x_2,\ldots,x_l$ and $k\in \mathbb Z$ such that
	\begin{itemize}
			\item $s_i\prec \Delta_{x_i}$ for $1\le i\le l$;
			\item $s_i=s_is_{i+1}\wedge \Delta_{x_i}$ in $(\mathcal C_{x_i\to},\preccurlyeq)$ (recall that $\wedge$ means meet);
			\item $\Delta^k$ means a concatenation of $|k|$ copies of $\Delta$ or $\Delta^{-1}$ with appropriate sources. 
		\end{itemize}
\end{thm}
Following \cite{charney2004bestvina}, the decomposition in the above theorem is called the \emph{left greedy Deligne normal form} of $f$.

\subsection{$\widetilde A_n$-like complex}

\label{subsec:A_nlike}
\begin{definition}
	\label{def:A_n}
	Let $X$ be a simplicial complex. We say $X$ is $\widetilde A_n$-like if
	\begin{enumerate}
		\item $X$ is flag and simply-connected;
		\item the vertex set of each simplex of $X$ has a cyclic order, and these cyclic orders are consistent with respect to inclusion of simplices;
		\item for each vertex $x\in X$ and any simplex $\sigma$ containing $x$, the cyclic order on $\sigma^0$ induces a linear order in $\sigma^0\setminus\{x\}$ (Section~\ref{subsec:pco}), and the previous item implies these linear orders fit together to form a relation $<_x$ on $\lk^0(x,X)$; we require this relation to be a partial order, and any pair of upper bounded elements have the join, any pair of lower bounded elements have the meet;
		\item there is a function $f:X^0\to \mathbb Z$ with finite image such that for each simplex $\sigma$ of $X$, $f|_{\sigma^0}$ is an injective morphism between cyclically ordered sets (we endow $\mathbb Z$ with the canonical cyclic order).
	\end{enumerate}
\end{definition}

This notion appeared in \cite[\S 4.2]{haettel2022link}. These complexes are Bestvina complexes of certain  Garside groupoids \cite{bessis2006garside}, as we will see below.
Item (4) is not required in \cite{haettel2022link}, however, it would simplify the discussion below so we include it (although it might not be essential for the discussion below). Here we do not require $X$ to be locally finite as in \cite[\S 4.2]{haettel2022link}.
Examples of $\widetilde A_n$-like simplicial complexes include the Coxeter complex of type $\widetilde A_n$, Euclidean buildings of type $\widetilde A_n$ \cite[\S 3.2]{hirai2020uniform}, and the Artin complex of type $\widetilde A_n$ \cite[Thm 4.3]{haettel2021lattices}. 

Up to translation, we assume $f(X^0)\subset[0,n-1]$. Following \cite[\S 3.2]{hirai2020uniform} and \cite[\S 4]{haettel2021lattices}, we define a simplicial complex structure on $\widehat X=X\times \mathbb R$ as follows.

The vertex set~$\widehat X^0$ of~$\widehat X$ is 
$$\{(x, i) \in X^0 \times \mathbb Z \mid f(x)\equiv i\ \mathrm{mod}\ n\},$$ 
The vertices $(x,i)$ and $(x',j)$ are neighbours if $x$ and $x'$ are equal or neighbours in~$X$, and $|i-j|  \leq  n$. Let $\widehat X$ be the flag simplicial complex with that $1$-skeleton. Note that any maximal simplex of $\widehat X$ has vertices $$(x_i,kn+f(x_i)),(x_{i+1},kn+f(x_{i+1})),\ldots,(x_n,kn+f(x_n)),(x_1,kn+n+f(x_1)),
\ldots,(x_{i},kn+n+f(x_i)),$$ where $k\in \mathbb Z$, $1\le i\le n,$ and $x_1,x_2,\ldots,x_n$ are vertices of a maximal simplex of $X$ with $f(x_1)<f(x_2)<\cdots<f(x_n)$.

We define a relation $\le$ on $\widehat X^0$, by requiring $(x,i)<(y,j)$ if there is an edge from $(x,i)$ to $(y,j)$ and $i<j$. Then there is an automorphism $\varphi$ of $(\widehat X^0,<)$ with $\varphi((x,i))=(x,i+n)$.
Note that the transitive closure $\le_t$ of $\le$ is a partial order on $\widehat X^0$. Then $(\widehat X^0,\le_t)$ is a weakly graded poset with rank function $r((x,i))=i$. 

\subsection{From $\widetilde A_n$-like complexes to Garside categories}
\label{subsec:widehatX}
\begin{definition}[{\cite[Def~4.6]{haettel2024lattices}}]
	\label{def:garsideflag}
	Let $\widehat X$ be a simply connected flag simplicial complex. Suppose that we have a binary relation $<$ on $\widehat X^0$ (not necessarily a partial order) such that vertices $x,y$ are neighbours exactly when $x<y$ or $y<x$. Furthermore, suppose that the transitive closure of $<$ is a partial order that is weakly graded with rank function $r$. We write $x\leq y$ when $x<y$ or $x=y$.
	
	Assume that we have an automorphism $\varphi$ of $(\widehat X^0,<)$ such that 
	\begin{itemize}
		\item 
		$r\circ \varphi=t\circ r$, for a translation $t\colon \mathbb Z\to \mathbb Z$, and
		\item
		$x\leq y$ if and only if $y\leq \varphi(x)$, for all $x,y\in \widehat X^0,$ and 
		\item 
		the interval $[x,\varphi(x)]=\{z\in \widehat X^0\mid x\le z\ \mathrm{and}\ z\le \varphi(x)\}$ is a lattice for all $x\in \widehat X^0$ (in particular, the relation $<$ restricted to $[x,\varphi(x)]$ is transitive).
	\end{itemize}
	We then call $\widehat X$ a \emph{Garside flag complex.}
\end{definition} 

\begin{lem}
	\label{lem:garsideA_n}
	Suppose $X$ is an $\widetilde A_n$-like complex. Then
	\begin{enumerate}
		\item 	the complex $\widehat X$ with the automorphism $\varphi$ and rank function $r$ on $(\widehat X^0,<)$ is a Garside flag complex;
		\item  for any $x\in \widehat X^0$, the posets $\{w\in \widehat X^0\mid w\ge_t x\}$  and $\{w\in \widehat X^0\mid w\le_t x\}$ are lattices;
		\item $\widehat X$ is contractible, hence $X$ is contractible.
	\end{enumerate}
\end{lem}

\begin{proof}
	If $(x,f(x)+kn)<(y,i)<(x,f(x)+kn+n)$, then $y\in \lk(x,X)$, and 
	\begin{enumerate}
		\item $i=f(y)+kn$ if $f(x)<f(y)$ (such $y$ is of \emph{type I});
		\item $i=f(y)+kn+n$ if $f(x)>f(y)$ (such $y$ is of \emph{type II}). 
	\end{enumerate}
	Given two such $(y_1,i_1)$ and $(y_2,i_2)$ that are comparable in $[(x,f(x)+kn),(x,f(x)+kn+n)]$. If they have different types, then the type II one is bigger. If they have the same type, then the one has $y_i$ with bigger $f(y_i)$ value is bigger. 
	Definition~\ref{def:A_n} (4) implies that this gives an isomorphism between $([(x,f(x)+kn),(x,f(x)+kn+n)],<)$ and $(\lk^0(x,X),<_x)$. Now Definition~\ref{def:A_n} (3) implies that $[(x,f(x)+kn),(x,f(x)+kn+n)]$ is a lattice. Other requirements of Definition~\ref{def:garsideflag} are clear. The last two statements follow from \cite[Thm 1.3]{haettel2024lattices}.
\end{proof}


\begin{cor}
	\label{cor:bestvina}
	Suppose $X$ is an $\widetilde A_n$-like complex. Then  $\widehat X$ gives a homogeneous categorical Garside structure $\mathcal C$. Objects of this category corresponds to elements in $\widehat X^0$, $x\le_t y$ corresponds to a morphism from $x$ to $y$, and $\varphi:\widehat X^0\to\widehat X^0$ plays the role of $\phi$ in Definition~\ref{def:Garside}. Moreover, $X$ is the Bestvina complex of $\mathcal C$ in the sense of \cite[\S 8]{bessis2006garside}.
\end{cor}
\subsection{Normal forms on $\widehat X$}
\label{subsec:normal form}
Theorem~\ref{thm:normal form0} translates to a left greedy Deligne normal form on $\widehat X$.



\begin{thm} 
	\label{thm:normal form}
	For each $x,y\in \widehat X^0$, there is a unique edge-path $x_1\cdots x_l\cdots x_n$ from $x_1=x$ to $x_n=y$ such that
	\begin{itemize}
		\item $x_i<x_{i+1} \neq \varphi(x_i)$ for $1\leq i<l$, and
		\item $x_{i}=x_{i+1}\wedge \varphi(x_{i-1})$ for $1< i<l$ in $[x_i,\varphi(x_i)]$, and
		\item $x_i=\varphi^{\pm (i-l)} (x_l)$ for $l\leq i\leq n$, with all signs positive or all signs negative.
	\end{itemize}
If $x<_t y$, then all signs in the last item are positive.
\end{thm}


Given $x\le_t y\in \widehat X^0$, let $\alpha(xy)$ be the maximal element (with respect to $\le_t$) in $[x,\varphi(x)]$ which is $\le_t y$. Note that $\alpha(xy)$ is well-defined by Lemma~\ref{lem:garsideA_n}. We can produce an edge path from $x$ to $y$ as follows: we go from $x=x_1$ to  $x_2=\alpha(xy)$, then from $x_2$ to $x_3=\alpha(x_2y)$, and repeat this procedure until we reach $x_n=y$. This gives a path $x_1\cdots x_n$ with $1\le l\le n$ such that $x_{i+1}=\varphi(x_i)$ for $1\le i<l$, $x_{i+1}<\varphi(x_i)$ for $l\le i\le n-1$ and  $x_i=x_{i+1}\wedge\varphi(x_{i-1})$ for $1<i<n$. We will call such path a \emph{left greedy quasi-normal path} from $x$ to $y$. If $l=1$, then such path is already a left greedy normal form path. In general, we consider an alternative path from $x=x_1$ to $y=x_n$ which is a concatenation of $\varphi^{-(l-1)}(x_lx_{l+1}\cdots x_n)$ and $\varphi^{-(l-1)}(x_n)\varphi^{-(l-2)}(x_n)\cdots x_n$, which gives the left greedy normal form path from $x$ to $y$ as in Theorem~\ref{thm:normal form}.

The following is proved as \cite[Prop 2.4]{holt2010garside} for Garside group. However, the same proof works in our setting.
\begin{prop}
	\label{prop:shortest}
Take $x,y\in \widehat X^0$ with $x<_t y$. Then the edge path in Theorem~\ref{thm:normal form} is a shortest edge path in $\widehat X^1$ from $x$ to $y$.
\end{prop}

We will also be using the right greedy normal form as follows. Let $x\le_t y\in \widehat X^0$. Let $\beta(xy)$ be the smallest element (with respect to $\le_t$) in $[\varphi^{-1}(y),y]$ which is $\ge_t x$. We go from $y=y_1$ to $y_2=\beta(xy)$, then from $y_2$ to $y_3=\beta(xy_2)$. We repeat this procedure until we reach $y_n=x$, which gives the right greedy normal form path from $x$ to $y$. If $x\le_t y$ is not true, then let $m$ be the smallest positive integer such that $x\le_t\varphi^m(y)$. Then the right greedy normal form is a concatenation of the right greedy normal form from $x$ to $\varphi^m(y)$, and the path $\varphi^m(y),\varphi^{m-1}(y),\ldots,y$. Right greedy normal form admits a similar characterization as Theorem~\ref{thm:normal form}.

\begin{thm} 
	\label{thm:right normal form}
	For each $x,y\in \widehat X^0$, there is a unique edge-path $x_1\cdots x_l\cdots x_n$ from $x_1=x$ to $x_n=y$ such that
	\begin{itemize}
		\item $x_i<x_{i+1} \neq \varphi(x_i)$ for $1\leq i<l$, and
		\item $x_{i}=x_{i-1}\vee \varphi^{-1}(x_{i+1})$ for $1< i<l$ in $[\varphi^{-1}(x_i),x_i]$, and
		\item $x_i=\varphi^{\pm (i-l)} (x_l)$ for $l\leq i\leq n$, with all signs positive or all signs negative.
	\end{itemize}
	If $x<_t y$, then all signs in the last item are positive,	
\end{thm}

 
The following was proved for certain Garside groups in \cite{deligne} and \cite[Lem 2.4]{charney1992artin}, however, the same proof works for Garside categories.
\begin{thm}
	\label{thm:intial}
Take $x\le_t y\le_t z\in\widehat X^0$. Let $z'$ be the endpoint of $\alpha(yz)$. Then $\alpha(xz)=\alpha(xz')$.
\end{thm}

The following was proved for certain Garside groups in \cite[Prop 3.3]{charney1992artin} (see also \cite[Prop 2.6]{gebhardt2010cyclic}), however, the same proof works for Garside categories.
\begin{thm}
	\label{thm:replace}
Let $x_1\cdots x_n$ be the left quasi-normal path from $x_1$ to $x_n$ in $\widehat X$. Take $y\in \widehat X^0$ with $x_n<y$. Then we can produce the left quasi-normal path from $x_1$ to $y$ by modifying $x_1\cdots x_ny$ in the following way. First we replace $x_{n-1}x_ny$ by the left quasi-normal path $x_{n-1}x'_ny$ from $x_{n-1}$ to $y$ (it is possible that $x'_n=y$). Second we replace $x_{n-2}x_{n-1}x'_n$ by the left quasi-normal path $x_{n-2}x'_{n-1}x'_n$ from $x_{n-2}$ to $x'_n$. Repeating this procedure until the last step, when we replace $x_1x_2x'_3$ by the left quasi-normal path $x_1x'_2x'_3$. The resulting path is the left quasi-normal path from $x_1$ to $y$.
\end{thm}

See Figure~\ref{fig:strip} for an illustration of Theorem~\ref{thm:replace}. The configuration in Figure~\ref{fig:strip} is a \emph{strip}, whose vertices are mapped to $\widehat X^0$. It is possible that $x_i=x'_i$ for some $i$.

\begin{figure}
	\centering
	\includegraphics[scale=1]{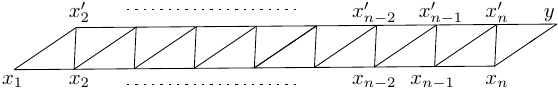}
	\caption{A strip.}
	\label{fig:strip}
\end{figure}

We record the following consequence of Proposition~\ref{prop:shortest} and Theorem~\ref{thm:replace}.

\begin{lem}
	\label{lem:smaller dist}
	Take $x,y,z\in \widehat X^0$ with $x<_t y<_t z$. Then the length of the left quasi-normal path from $x$ to $y$ is upped bounded by the length of the left quasi-normal path from $x$ to $z$. In particular,
	$d(x,y)\le d(x,z)$, where $d$ denotes the path metric on $\widehat X^1$ with edge length 1.
\end{lem}
\subsection{Bestvina convexity}
\begin{definition}
	Suppose $X$ is $\widetilde A_n$-like complex.
	Given an edge-path $P=x_1\cdots x_n$ in $X$, 
	an \emph{admissible lift} of $P$ is an edge-path $\widehat P=\hat x_1\cdots \hat x_n$ in $\widehat X$ 
	such that $\pi(\hat x_i)=x_i$, for $1\le i\le n$, and  $\hat x_i<\hat x_{i+1}$, for $1\le i\le n-1$.
	Note that for each edge-path $P$ in $X$, once a lift $\hat x_1$ of $x_1$ has been chosen, there is a unique admissible lift of $P$ starting at $\hat x_1$. Different admissible lifts of $P$ differ by the translation by $\varphi^k$ for some $k\in \mathbb Z$.
	
	Let $a,b\in X^0$. Following \cite{Bestvina1999,charney2004bestvina}, we say that an edge-path $P$ from $a$ to $b$ is a \emph{left geodesic}, (or \emph{left B-geodesic}) if some (hence all) admissible lift of $P$ to $\widehat X$ has left Deligne normal form with $n=l$. Similarly we define right geodesic from $a$ to $b$. By \cite[Lemma 2.1 and Proposition 2.2]{Bestvina1999}, these are indeed geodesics in the sense that their length equal to $d(a,b)$, where $d$ denotes the path metric on $X^1$ with edge length 1.
\end{definition}

The following lemma was proved for left $B$-geodesics in \cite[Lem 11.5 and 11.6]{huang2025353}. The statements for right $B$-geodesics can be proved similarly.

\begin{lem}
	\label{lem:B-geodesics}
	Suppose $X$ is an $\widetilde A_n$-like complex. 
	\begin{enumerate}
		\item For $a,b\in X^0$, there is a unique left $B$-geodesic from $a$ to $b$, and a unique right $B$-geodesic from $a$ to $b$.
		\item An edge path $z_1\cdots z_n$ in $X$ is left $B$-geodesic if and only if for each $2\le i\le n-2$, $z_{i-1}$ and $z_{i+1}$ do not have a common lower bound in $(\lk^0(z_i,X),<_{z_i})$. An edge path $z_1\cdots z_n$ in $X$ is right $B$-geodesic if and only if for each $2\le i\le n-2$, $z_{i-1}$ and $z_{i+1}$ do not have a common upper bound in $(\lk^0(z_i,X),<_{z_i})$. 
	\end{enumerate}
\end{lem}

By Lemma~\ref{lem:B-geodesics}, the left $B$-geodesic from $a$ to $b$ coincidence with the left $B$-geodesic from $b$ to $a$. A similar statement holds for right $B$-geodesics.

\begin{definition}
	\label{def:Bconvex}
	Let $X$ be an $\widetilde A_n$-like complex as before.
	Let $Y\subset X$ be a full subcomplex. 
	We say that $Y$ is \emph{left locally B-convex} if for each vertex $y\in Y^0$ and any vertices $y_1,y_2$ of $\lk(y,Y)$, if the meet $y_1\wedge y_2$ in the poset $(\lk(y,X),\le_y)$ exists, then $y_1\wedge y_2\in \lk(y,Y)$. 	We say that $Y$ is \emph{right locally B-convex} if for each vertex $y\in Y^0$ and any vertices $y_1,y_2$ of $\lk(y,Y)$, if the join $y_1\vee y_2$ in the poset $(\lk(y,X),\le_y)$ exists, then $y_1\vee y_2\in \lk(y,Y)$.
\end{definition}

The following is proved for left locally $B$-convex subcomplexes in \cite[Prop 11.8]{huang2025353} (the notions of $\widetilde A_n$-like complexes and locally $B$-convex subcomplexes in \cite{huang2025353} are more restricted, however, the same proof works in our setting here). The statement for right locally $B$-convex subcomplexes can be proved in a similar way.
\begin{prop}
	\label{prop:convex}
	Let $X$ 
	be an $\widetilde A_n$-like complex, 
	and let $Y\subset X$ be a connected left locally B-convex subcomplex. 
	Then $Y$ is simply-connected, and $Y$ is itself an $\widetilde A_n$-like complex, with the function $f$ induced from $X$. Moreover, for any pair of vertices $y_1,y_2\in Y^0$, the left B-geodesic in $X$ from $y_1$ to $y_2$ is contained in $Y$.
	
	If $Y$ is right locally B-convex instead, then all previous conclusions still hold, except that for any pair of vertices $y_1,y_2\in Y^0$, the right B-geodesic in $X$ from $y_1$ to $y_2$ is contained in $Y$.
\end{prop}

\subsection{Bestvina non-positive curvature}
\label{subsec:bestvina asymmetric}

Given $a,b\in X^0$, the \emph{Bestvina asymmetric distance from $a$ to $b$}, denoted $\ad(a,b)$, is defined as follows. Let $x_1\cdots x_n$ be the left $B$-geodesic in $X$ from $a$ to $b$, with $\hat x_1\cdots \hat x_n$ be an admissible lift. Let $r:\widehat X^0\to \mathbb Z$ be the rank function in Lemma~\ref{lem:garsideA_n}. Then $\ad(a,b):=r(\hat x_n)-r(\hat x_1)$. Note that $\ad(a,b)$ does not depend on the choice of the lift. But in general $\ad(a,b)$ might not equal to $\ad(b,a)$. This asymmetric distance satisfies a non-positive curvature like feature (Proposition~\ref{prop:bnpc} below), which will play a key role in this article.

Given $a,x,y\in X^0$ such that $x$ and $y$ are adjacent. Let $P_x$ (resp. $P_y$) be the left $B$-geodesic from $a$ to $x$ (resp. $y$). Let $Q_y$ be the concatenation of $P_x$ and the edge $xy$, and let $Q_x$ be the concatenation of $P_y$ and $yx$. Let $\hat P_x,\hat P_y,\hat Q_x,\hat Q_y$ be admissible lifts of these four paths, starting at the same vertex $\hat a$, with endpoints $\hat x,\hat y,\hat x',\hat y'$. Then $\hat x'=\varphi^m(\hat x)$ and $\hat y'=\varphi^n(\hat y)$ for integers $m,n$.

\begin{lem}(\cite[Lem 3.4]{Bestvina1999})
	\label{lem:mn}
One of $m,n$ is $0$, the other is $1$. In particular, either $\ad(a,x)<\ad(a,y)$, or $\ad(a,y)<\ad(a,x)$.
\end{lem} 

\begin{proof}
Note that $\varphi^{-n}(\hat x)<\varphi^{-n}(\hat y')=\hat y$ and $\hat y<\hat x'$. The concatenation of these two edges go from $\varphi^{-n}(\hat x)$ to $\hat x'=\varphi^{n+m}(\varphi^{-n}(\hat x))$. Thus $n+m=1$. 

It remains to show $n\ge 0$ and $m\ge 0$. Assume $n<0$. Then $\hat a<_t\hat x<_t\hat y'<_t\hat y$. However, the admissible lift of the left $B$-geodesic from $\hat a$ to $\hat y$ satisfies $n=l$ in Theorem~\ref{thm:normal form}, contradiction. Similarly, $m\ge 0$.
\end{proof}

\begin{prop}(\cite[Prop 3.12]{Bestvina1999})
	\label{prop:bnpc}
Given $a,b,c\in X^0$ and let $x_1\cdots x_n$ be the left $B$-geodesic in $X$ from $b$ to $c$. Then $\ad(a,x_i)$ first strictly decreases and then strictly increases as $i$ goes from $1$ to $n$ (the decreasing part or the increasing part is allowed to be trivial).
\end{prop}

\begin{proof}
Suppose on the contrary $\ad(a,x_{i-1})<\ad(a,x_i)>\ad(a,x_{i+1})$ for some $i$. Let $\hat P_{i-1},\hat P_i,\hat P_{i+1}$ be admissible lifts of left $B$-geodesics from $a$ to $x_{i-1},x_i,x_{i+1}$, starting from the same point $\hat a$ and ending at $\hat x_{i-1},\hat x_i,\hat x_{i+1}$. By Lemma~\ref{lem:mn}, $\hat x_{i-1}<\hat x_i$, $\hat x_i<\hat x'_{i+1}$ and $\hat x'_{i+1}=\varphi(\hat x_{i+1})$, where $\hat x_i\hat x'_{i+1}$ is the admissible lift of $x_ix_{i+1}$. As $x_1\cdots x_n$ is a left $B$-geodesic, $\alpha(\hat x_{i-1}\hat x'_{i+1})=\hat x_{i-1}\hat x_{i}$ with $\alpha$ defined in Section~\ref{subsec:normal form}, so Theorem~\ref{thm:intial} implies that $\alpha(\hat a\hat x'_{i+1})=\alpha(\hat a\hat x_i)$. However, as $\hat a<_t\hat x_{i+1}<_t\hat x'_{i+1}$, we know $\alpha(\hat a\hat x'_{i+1})=\hat a\varphi(\hat a)$. On the other hand, $\hat P_i$ satisfies $n=l$ in Theorem~\ref{thm:normal form}. This is a contradiction.
\end{proof}

%

%
%

\section{The minimal cut complex}
\label{sec:minimal cut complex}
Let $\Lambda$ be a Coxeter diagram. 
Let $P\subset \Lambda$ be a nontrivial path from vertex $a\in \Lambda$ to vertex $b\in \Lambda$ such that $P$ is embedded except possibly $a=b$, and all interior vertices of $P$ has valence $2$ in $\Lambda$. We also assume that $a,b$ are not valence one vertices of $\Lambda$. We remove all the interior points of $P$ from $\Lambda$ to obtain a subgraph $\Lambda_P$ of $\Lambda$.

Let $\mathcal C_P$ be the collection of sets of vertices of $\Lambda$ such that elements in $\mathcal C_P$ are either single vertices in the interior of $P$, called \emph{type I} elements, or belong to $\mc_{\Lambda_P}(\{a\},\{b\})$, called \emph{type II} elements. Two elements $S,S'\in\mathcal C_P$ are \emph{comparable}, if either at least one of them is of type I, or they are both of type II and they are comparable in $\mc_{\Lambda_P}(\{a\},\{b\})$.

\begin{definition}
We define the \emph{minimal cut complex} $\Delta^P_\Lambda$ associated to $P\subset\Lambda$ to be a simplicial complex as follows. Vertices of $\Delta^P_\Lambda$ are in 1-1 correspondence with vertices of $\Delta'_{\Lambda}$ that are of type $\hat S$ with $S\in \mathcal C_P$. Two vertices of $\Delta^P_{\Lambda}$, one of type $\hat S_1$ and another of type $\hat S_2$, are adjacent if they are contained in a common simplex of $\Delta_\Lambda$ and $S_1$ and $S_2$ are comparable. Then $\Delta^P_\Lambda$ is the flag complex on its 1-skeleton. A vertex of $\Delta^P_\Lambda$ is of \emph{type I or II} if it is of type $\hat S$ with $S$ being of type I or II.

In other words, vertices of $\Delta^P_\Lambda$ are in 1-1 correspondence with left cosets of form $gA_{\hat S}$, with $S\in \mathcal C_P$. A collection of vertices span a simplex if and only if the common intersection of the associated cosets is non-empty, and their types are pairwise comparable. By Proposition~\ref{prop:intersection}, this description coincides with the description in the previous paragraph.
\end{definition}

The goal of this section is the following.
\begin{prop}
	\label{prop:mincutAn}
Let $\Lambda$ be a Coxeter diagram. 	
Suppose that for all proper induced subdiagrams $\Lambda'$ of $\Lambda$, $\Delta_{\Lambda'}$ satisfies the labeled 4-cycle condition. Let $P$ be an embedded path in $\Lambda$ as above such that $\Lambda_P$ is connected. Then the complex $\Delta^P_\Lambda$ is an $\widetilde A_n$-like complex in the sense of Definition~\ref{def:A_n}.
\end{prop}

\subsection{Simply-connectedness of $\Delta^P_\Lambda$}
\label{subsec:sc}


Let $\Sigma_P$ be the collection of sets of vertices of $\Lambda$ such that elements in $\Sigma_P$ are either a single vertex of $P$, or a subset of $\Lambda_P\setminus\{a,b\}$ that contains at least one type II element of $\mathcal C_P$. Then $\mathcal C_P\subset \Sigma_P$.


\begin{definition}
	\label{def:Delta1}
We define a simplicial complex $\Delta_1$ as follows. Vertices of $\Delta_1$ are in 1-1 correspondence of left cosets of form $gA_{\hat T}$ with $T\in \Sigma_P$. Such a vertex is \emph{special} if $T$ is a vertex of $P$, otherwise the vertex is \emph{non-special}. A special vertex is adjacent to another vertex is the two associated left cosets have non-empty intersection. Two non-special vertices are adjacent the left coset associated with one vertex is contained in the left coset associated with another vertex. Then $\Delta_1$ is the flag complex on its 1-skeleton.
\end{definition}

\begin{lem}
	\label{lem:Delta1sc}
The complex $\Delta_1$ is simply-connected.
\end{lem}

\begin{proof}
We define an auxiliary complex $\Delta_2$ whose vertex set is the same as $\Delta_1$. A collection of vertices of $\Delta_2$ span a simplex if the associated cosets have non-empty common intersection. By Proposition~\ref{prop:intersection} $\Delta_2$ is a flag complex. 

Note that $\Delta_1$ and $\Delta_2$ are homotopic equivalent. Indeed, let $V$ be the collection of vertices such that the associated coset contain the identity element. For $i=1,2$, let $K_i$ be the full subcomplex of $\Delta_i$ spanned by $V$. By Proposition~\ref{prop:intersection}, $\mathcal U_1=\{gK_1\}_{g\in A_\Gamma}$ form a covering of $\Delta_1$. Note that if finitely many elements in $\mathcal U_1$ have non-empty intersection, then the intersection is contractible. To see this, it suffices to show the full subcomplex $L$ of  $K_1$ spanned by vertices whose associated cosets containing a finite list of given elements $\{g_i\}_{i=1}^k$ of $A_\Gamma$ is contractible. This is clear if $L$ contains a special vertex, as any other vertex in $L$ is adjacent to this special vertex. If $L$ does not contain any special vertex, then let $x$ be the vertex of $L$ corresponding to the intersection of all cosets associated with vertices of $L$. Then $x$ is adjacent to any other vertex of $L$. Let $\mathcal N_1$ be the nerve of this covering. Then $\mathcal N_1$ is homotopic to $\Delta_1$ by \cite[Thm 6]{bjorner2003nerves}.
Let $\mathcal N_2$ be the nerve of $\{gK_2\}_{g\in A_\Gamma}$. Then $\mathcal N_1\cong \mathcal N_2$. Similarly, $\mathcal N_2$ is homotopic equivalent to $\Delta_2$, so $\Delta_1$ and $\Delta_2$ are homotopic equivalent. 
	
Let $Z$ be the Cayley complex of the Artin group $A_\Lambda$ (i.e. the universal cover of its presentation complex). Then $Z$ is simply-connected. Then the 0-skeleton of $Z$ can be identified with elements in $A_\Lambda$. For each left coset $gA_{\hat T}$, let $Z(gA_{\hat T})$ be the full subcomplex of $Z$ spanned by vertices in $gA_{\hat T}$. By \cite{lek}, $Z(gA_{\hat T})$ is a copy of the Cayley complex of $A_{\hat T}$, hence is simply-connected. We claim the collection $\{Z(gA_{\hat T})\}_{g\in A_\Lambda, T\in \Sigma_P}$ forms a covering of $Z$. As the nerve of this covering is $\Delta_2$,
by \cite[Thm 6]{bjorner2003nerves}, $\pi_1(\Delta_2)\cong \pi_1(Z)$, and the lemma follows. To prove the claim, it suffices to show that for each edge $e$ of $\Lambda$, there is $T\in \Sigma_P$ such that $e\cap T=\emptyset$, as this would imply each 2-cell of $Z$ is contained in one of the subcomplexes in $\{Z(gA_{\hat T})\}_{g\in A_\Lambda, T\in \Sigma_P}$. Suppose $P$ has length $\ge 2$. Let $c$ be an interior vertex of $P$. Then any edge of $\Lambda_P$ must be disjoint from $c$, and $\{c\}\in \Sigma_P$; and any edge contained in $P$ is disjoint from either $a$, or $b$. Suppose $P$ has length $1$. Then $a$ and $b$ have distance $\ge 2$ in $\Lambda_P$. Hence any edge of $\Lambda_P$ is either disjoint from $a$, or from $b$. Now we consider the unique edge $e$ of $P$. Let $T$ be the collection of vertices in $\Lambda_P$ that has distance $1$ from $a$. Then $T$ is disjoint from $a$ and $b$. As $T$ separates $\{a\}$ from $\{b\}$ in $\Lambda_P$, $T$ contains a subset which belongs to $\mc_{\Lambda_P}(\{a\},\{b\})$. Thus $T\in \Sigma_P$. By the construction of $T$, we know $e\cap T=\emptyset$.
\end{proof}



\begin{definition}
We define a subdivision $b\Delta^P_\Lambda$ of $\Delta^P_\Lambda$ as follows. A vertex of $\Delta^P_\Lambda$ is \emph{special}, if it is of type $\hat s$ for vertex $s\in P$, otherwise the vertex is \emph{non-special}. A special simplex of $\Delta^P_\Lambda$ is a simplex made of special vertices. Similarly, we define non-special simplices of $\Delta^P_\Lambda$. Then $b\Delta^P_\Lambda$ is obtained from $\Delta^P_\Lambda$ by only performing barycentric subdivision on non-special simplices of $\Delta^P_\Lambda$, i.e. each simplex of $b\Delta^P_\Lambda$ is the join of a special simplex of $\Delta^P_\Lambda$ and a simplex of the barycentric subdivision of some non-special simplex of $\Delta^P_\Lambda$. We define a \emph{special} vertex of $b\Delta^P_\Lambda$ to be a vertex corresponding to a special vertex of $\Delta^P_\Lambda$. Other vertices of $b\Delta^P_\Lambda$ are non-special.

Note that a vertex $x$ of $b\Delta^P_\Lambda$ corresponds to a left coset of form $gA_{\hat T}$, where $T$ is either a vertex of $P$ (when $x$ is a special vertex), or a union of a collection of pairwise comparable elements in $\mc_{\Lambda_P}(\{a\},\{b\})\setminus\{\{a\},\{b\}\}$ (when $x$ is non-special). Although two different non-special vertices of $b\Delta^P_\Lambda$ could possibly give the same left coset, as the union of two different collections of pairwise comparable elements in $\mc_{\Lambda_P}(\{a\},\{b\})\setminus\{\{a\},\{b\}\}$ could be the same set.
\end{definition}


\begin{definition}
We define a map $\rho:b\Delta^P_\Lambda\to \Delta_1$ as follows. As each vertex of $b\Delta^P_\Lambda$ corresponds a left coset of form $gA_{\hat T}$ with $T\in \Sigma_P$, we obtain $\rho: (b\Delta^P_\Lambda)^0\to \Delta^0_1$. Note that $\rho$ sends adjacent vertices of $b\Delta^P_\Lambda$ to adjacent or identical vertices of $\Delta_1$, so $\rho$ extends to a simplicial map.
\end{definition}

\begin{lem}
	\label{lem:homotopy injective}
The map $\rho:b\Delta^P_\Lambda\to \Delta_1$ induces a homotopy equivalence between these two complexes.
\end{lem}

\begin{proof}
	We will construct a continuous map $\lambda:\Delta_1\to b\Delta^P_\Lambda$ such that $\lambda\circ \rho$ and $\rho\circ\lambda$ are homotopic to identity map.
	
	
	Each special vertex of $\Delta_1$ corresponds to a special vertex of $b\Delta^P_\Lambda$, which gives the definition of $\lambda$ on special vertices. Take a non-special vertex $v$ of $\Delta_1$ corresponding to $gA_{\hat T}$ with $T\in \Sigma_P$. Consider the partial order on $\mc_\Lambda(\{a\},\{b\})$ in Lemma~\ref{lem:latticecut}, such that $\{a\}<\{b\}$. Let $T^-\in \mc_\Lambda(\{a\},\{b\})$ be the smallest possible element such that $T^-\subset T$, which is well-defined by Lemma~\ref{lem:latticecut}. Consider the unique left $A_{\hat T^-}$-coset containing $gA_{\hat T}$, which gives a vertex $v^-$ in $\Delta^P_\Lambda$. Then $\lambda(v)=v^-$.
	
	We claim $\lambda$ sends vertices of $\Delta_1$ in a simplex to points that are contained in a common simplex of $\Delta^P_\Lambda$. It suffices to prove this claim for a collection of non-special vertices $\{v_i\}_{i=1}^k$ of $\Delta_1$. Suppose $v_i$ is of type $\hat T_i$, and assume $T_1\subset T_2\subset\cdots\subset T_k$. We define $T^-_i,v^-_i$ as in the previous paragraph. Then $T^-_k\le T^-_{k-1}\le\cdots\le T^-_1$. As the coset associated with $v^-_i$ contains the coset associated with $v_i$, we obtain 
	that the cosets associated with $\{v^-_1,\ldots,v^-_k\}$ have non-empty common intersection. Thus $\{v^-_1,\ldots,v^-_k\}$ spans a simplex in $\Delta^P_\Lambda$, and the claim follows. This claim implies that we can extend $\lambda$ using the affine structure on each simplex of $\Delta^P_\Lambda$ to obtain a continuous map $\lambda:\Delta_1\to \Delta^P_\Lambda$, which can also be viewed as a map to $b\Delta^P_\Lambda$.
	
Next we show $\rho\circ\lambda$ is homotopic to the identity map. Given a simplex $\sigma$ of $\Delta_1$, we write $\sigma$ as a join $\sigma_1*\sigma_2$ where $\sigma_1$ is made of special vertices and $\sigma_2$ is made of non-special vertices. Note that $\rho\circ\lambda(\sigma_1)=\sigma_1$. Suppose $\sigma_2$ has vertices $\{v_i\}_{i=1}^k$ such that $v_i$ has type $\hat T_i$ with $T_1\subset\cdots\subset T_k$. We define $T^-_i$ and $v^-_i$ as before. Let $\sigma^-_2$ be the simplex of $\Delta^P_\Lambda$ spanned by $\{v^-_1,\ldots,v^-_k\}$.  
Define $K(\sigma_2)$ to the full subcomplex of $\Delta_1$ spanned by $\sigma_2$ and $\rho(\sigma^-_2)$. As $T^-_i\subset T_i$, we know any union of $\{T^-_1,\ldots,T^-_k\}$ is contained in $T_k$. Thus each vertex of $K(\sigma_2)$ is adjacent or equal to $v_k$, hence $K(\sigma_2)$ is contractible. As each vertex of $K(\sigma_2)$ is adjacent every vertex of $\sigma_1$, we define $K(\sigma)=K(\sigma_2)*\sigma_1$, which is also contractible. It follows from construction that for two simplices $\sigma\subset \tau$ of $\Delta_1$, we have $K(\sigma)\subset K(\tau)$. As for each simplex $\sigma$ of $\Delta_1$, $\sigma$ and $\rho\circ\lambda(\sigma)$ are contained in $K(\sigma)$, we can construct skeleton by skeleton a continuous map $F:\Delta_1\times [0,1]\to \Delta_1$ such that $F\mid_{\Delta_1\times\{0\}}$ is the identity map and $F\mid_{\Delta_1\times\{1\}}=\rho\circ\lambda$. Hence $\rho\circ\lambda$ is homotopic to identity.

	Given a simplex $\sigma\subset b\Delta^P_\Lambda$ with vertices $\{v_i\}_{i=1}^k$, we claim $$\{v_1,\ldots,v_k,\lambda\circ\rho(v_1),\ldots,\lambda\circ\rho(v_k)\}$$ are contained in he smallest simplex $\bar \sigma$ of $\Delta^P_\Lambda$ that contains $\sigma$. 
	Again, it suffices to consider the case that all $v_i$ are non-special.
	Suppose $v_i$ has type $\hat T_i$ and suppose $T_1\subset\cdots\subset T_k$. Suppose vertices of $\bar\sigma$ are $\{u_i\}_{i=1}^n$ such that $u_i$ is of type $\hat U_i$ and  $U_1\le U_2\le \cdots\le U_n$ in $\mc_{\Lambda_P}(\{a\},\{b\})$. Then $T_k=\cup_{i=1}^n U_n$, and each  $T_i$ is a union of some members of $\{U_1,\ldots,U_n\}$. We define $T^-_i$ and $v^-_i$ as before. Then $T^-_i\in \{U_1,\ldots,U_n\}$ by Lemma~\ref{lem:comparable}, hence $v^-_i$ is a vertex of $\bar\sigma$. Note that $\lambda\circ\rho (v_i)=v^-_i$. Hence the claim is proved.
	Note that $\lambda\circ\rho\mid_{\sigma}$ is the unique affine map extending $\lambda\circ\rho\mid_{\sigma^0}$. So we can use the affine structure on $\bar\sigma$ to build the homotopy from $\lambda\circ\rho$ to the identity map.
\end{proof}

The following is a consequence of Lemma~\ref{lem:Delta1sc} and Lemma~\ref{lem:homotopy injective}.
\begin{cor}
	\label{cor:Xsc}
The complex $\Delta^P_\Lambda$ is simply-connected.
\end{cor}

\subsection{A  graph theoretical observations}
\label{subsec:graph theoretic}
We put a partial order on the collection of type II elements in $\mathcal C_P$ as in Lemma~\ref{lem:latticecut}, such that $\{a\}<\{b\}$. We put a linear order on the collection of type I elements in $\mathcal C_P$ along $P$ such that $\{b\}<\{a\}$. As $\mathcal C_P$ is obtained by gluing this two posets together, we can invoke Section~\ref{subsec:pco} and put a partial cyclic order on $\mathcal C_P$. Let $T\in \mathcal C_P$. Let $\mathcal C_{P,T}$ be the collection of elements in $\mathcal C_{P}$ that are comparable but not equal to $T$. By Section~\ref{subsec:pco}, $\mathcal C_{P,T}$ inherit a partial order from the partial cyclic order on $\mathcal C_P$.

Given $T\in \mathcal C_P$. Let $\Lambda_T$ be the full subgraph of $\Lambda$ spanned by all the vertices of $\Lambda\setminus T$ that are in the same connected component of $\Lambda\setminus T$ as $P\setminus T$. We consider the intersection of each element of $\mathcal C_{P,T}$ with $\Lambda_T$, which gives a collection $\mathcal C'_{P,T}$ of sets of vertices of $\Lambda_T$, and a map  $\phi:\mathcal C_{P,T}\to \mathcal C'_{P,T}$. If $T$ is of type I, then $\mathcal C'_{P,T}=\mathcal C_{P,T}$.


Let $\hat\Lambda_T$	be the connected component of $\Lambda\setminus T$ that contains $P\setminus T$. When $T$ is of type II, we let $T^{\{a\}}$ be the connected component of $\Lambda_P\setminus T$ that contains $a$ (if $a\in T$, then $T^{\{a\}}=\emptyset$). Similarly we define $T^{\{b\}}$. Then 
\begin{equation}
	\label{eq:decomp}
	\hat\Lambda_T=T^{\{a\}}\cup T^{\{b\}}\cup(P\setminus T).
\end{equation}
By Lemma~\ref{lem:comparable} and Lemma~\ref{lem:comparable1}, given $R\in \mathcal C_{P,T}$, either $R\subset T^{\{a\}}\cup T$, or $R\subset T^{\{b\}}\cup T$, or $R$ is an interior vertex of $P$; and these three cases are mutually exclusive. Thus for $R'\in \mathcal C'_{P,T}$, either $R'\subset T^{\{a\}}$, or $R'\subset T^{\{b\}}$, or $R$ is an interior vertex of $P$.

\begin{remark}
	\label{rmk:ab}
	We will often use the following observation, which follows from the definition.
	Suppose $\Lambda_P$ is connected. Then $\{a\}\in \mc_{\Lambda_P}(\{a\},\{b\})$ and $\{b\}\in \mc_{\Lambda_P}(\{a\},\{b\})$. Hence, if $C\in \mc_{\Lambda_P}(\{a\},\{b\})$ satisfies that $C\neq \{a\}$ and $C\neq\{b\}$, then $a\notin C$ and $b\notin C$.
\end{remark}

\begin{lem}
	\label{lem:admissible example}
Suppose $\Lambda_P$ is connected. Then $\phi:\mathcal C_{P,T}\to \mathcal C'_{P,T}$ is a bijection, hence $\mathcal C'_{P,T}$ inherits a partial order from $\mathcal C_{P,T}$. This partial order on $\mathcal C'_{P,T}$ satisfies all the properties in Definition~\ref{def:admissible} with $\Lambda$ in the definition replaced by $\Lambda_T$.
\end{lem}

\begin{proof}
First assume $T$ is of type I. Then $T=\{s\}$ with $s$ being an interior vertex of $P$, and $\Lambda\setminus\{s\}$ is connected. Then $\mathcal C'_{P,T}$ can be naturally identified with $\mc_{\Lambda_T}(\{s_1\},\{s_2\})$ (as posets), where $s_1,s_2$ are vertices of $P$ that are adjacent to $s$ in $P$. Hence we are done by Lemma~\ref{lem:mincut admissible}.

	
Assume $T$ is of type II.	
Define $\varphi:\mathcal C'_{P,T}\to \mathcal C_{P,T}$ as follows. Take $R'\in \mathcal C'_{P,T}$. If $R'\subset T^{\{a\}}$ (resp. $T^{\{b\}}$), then $\varphi(R')$ is $R'$ together with all vertices in $T$ that can be connected to $a$ (resp. $b$) via a $T$-tight path in $(T^{\{a\}}\cup T)\setminus R'$ (resp. $(T^{\{b\}}\cup T)\setminus R'$). If $R'$ is an interior vertex of $P$, then $\varphi(R')=R'$. By Lemma~\ref{lem:comparable1}, $\varphi(R')\in \mathcal C_{P,T}$ and $\varphi$ and $\phi$ are inverse of each other.

Now we verify the conditions in Definition~\ref{def:admissible}. Condition 1 is clear. For condition 4, given $R_1,R_2\in \mathcal C_{P,T}$ with an upper bound $R$. Suppose $R_1\subset T^{\{a\}}\cup T$. It suffices to consider the case $R_2\subset T^{\{a\}}\cup T$ as otherwise the join of $R_1$ and $R_2$ is $R_1$. By Lemma~\ref{lem:latticecut}, $R_1$ and $R_2$ have the join which is contained in $R_1\cup R_2$. Thus the same holds for $\phi(R_1)$ and $\phi(R_2)$. Other possibilities of $R_1$ as well as the lower bound part of condition 4 can be handled similarly. 

For condition 3, take $R_1<R_2<R_3$ in $\mathcal C_{P,T}$. If all of them are in $T^{\{a\}}\cup T$, then $R_2$ separates $R_1$ from $R_3$ in $\Lambda_P$ by Lemma~\ref{lem:separate}. By \eqref{eq:decomp}, $R'_2$ separates $R'_1$ from $R'_3$ in $\Lambda_T$  where $R'_i=\phi(R_i)$. If only two of them are in $T^{\{a\}}\cup T$, then they are $R_2$ and $R_3$. 
Thus either $R'_2=\{a\}$, or $R'_2$ separates $\{a\}$ from $R'_3$ in $\Lambda_T$. In the latter case, by \eqref{eq:decomp} $a$ is connected to each vertex of $R'_1$ via a path outside $R'_2$. Thus in both cases, $R'_2$ separates $R'_1$ from $R'_3$. Other cases of $R_1,R_2,R_3$ can be handled similarly.

For condition 2, given $T'_1<T'_2$ in $\mathcal C'_{P,T}$, we first consider the case that $T'_i\subset T^{\{a\}}$ for $i=1,2$. Take $M'\in \mc_{\Lambda_T}(T'_1,T'_2)$. By \eqref{eq:decomp} we know $M'\subset T^{\{a\}}$. 
Let $(M')^{\{a\}}$ be the collection of points in $T^{\{a\}}$ that can be connected to $a$ by a path in $T^{\{a\}}\setminus M'$. Similarly we define $(T'_i)^{\{a\}}$. We claim
\begin{equation}
	\label{eq:contain}
(T'_1)^{\{a\}}\subset(M')^{\{a\}}\subset(T'_2)^{\{a\}}.
\end{equation}
Now prove this claim. Let $T_i=\varphi(T'_i)$.
By Lemma~\ref{lem:comparable1}, $(T'_i)^{\{a\}}=T^{\{a\}}_i$. Hence $(T'_1)^{\{a\}}\subset (T'_2)^{\{a\}}$. 
By Remark~\ref{rmk:visit}, for each $m\in M'$, there is a $(T'_1,M',T'_2)$-path $\omega_m$ in $\Lambda_T$ from a point in $T'_1$ to a point in $T'_2$ passing through $m$. By \eqref{eq:decomp}, we can assume $\omega_m\subset T^{\{a\}}$. Then $\omega_m\cap (T'_1)^{\{a\}}=\emptyset$ (if such a path enters $(T'_1)^{\{a\}}$, then it must exits and hit $T'_1$ the second time). Hence $m\notin (T'_1)^{\{a\}}$, and  $M'\cap (T'_1)^{\{a\}}=\emptyset$, which implies $(T'_1)^{\{a\}}\subset(M')^{\{a\}}$. If $(M')^{\{a\}}\subseteq (T_2')^{\{a\}}$ is not true, then there is
a path $\omega\subseteq T^{\{a\}}\setminus M'$ from $a$ to
$x\notin (T_2')^{\{a\}}$. Thus $\omega$ must meet $T_2'$.
By Lemma~\ref{lem:comparable}(3), the initial subpath of $\omega$ ending at its first
point in $T_2'$ must also meet $T_1$. As $\omega\subseteq T^{\{a\}}$,
this intersection lies in $T_1'$. Hence $\omega$ contains a subpath
outside $M'$ from a point in $T_1'$ to a point in $T_2'$, which
contradicts $M'\in\operatorname{Mincut}_{\Lambda_T}(T_1',T_2')$.


Let $M$ be $M'$ together with all vertices of $T$ that can be connected
to $a$ by a $T$-tight path in $(T^{\{a\}}\cup T)\setminus M'$.
We claim $M\in \mc_{\Lambda_P}(T_1,T_2)$. 
To verify Lemma~\ref{lem:reformulate} (1), let $\omega\subset \Lambda_P$ be a path from $t_1\in T_1$ to $t_2\in T_2$. Up to passing to a subpath of $\omega$, we can assume $\omega$ is $(T_1,T_2)$-tight. 
As $T_1\le T_2\le T\in \mc_{\Lambda_P}(\{a\},\{b\})$, by Lemma~\ref{lem:comparable} (3), we can assume either $\omega\cap T=\emptyset$ or $\omega$ is $T$-tight with $t_2\in T\cap T_2$. In the former case, $\omega\subset T^{\{a\}}$, then $\omega$ can be viewed as a path in $\Lambda_T$, hence it must visit $M'$. In the latter case, let $\omega_1$ be a $T_1$-tight path in $T^{\{a\}}\cup T$ from $a$ to $t_1$, which exists by Remark~\ref{rmk:visit}. Then $(\omega_1\setminus\{t_1\})\cap M'=\emptyset$ by \eqref{eq:contain}. If $\omega$ is outside $M'$, then by considering the concatenation of $\omega_1$ and $\omega$, we know $t_2\in M$ by the definition of $M$. Thus Lemma~\ref{lem:reformulate} (1) follows. Now we verify Lemma~\ref{lem:reformulate} (2). Take $x\in M$. If $x\in M'$, as $M'\in \mc_{\Lambda_T}(T'_1,T'_2)$, there is a path in $\Lambda_T$ from a point in $T'_1$ to a point in $T'_2$ avoiding $M'\setminus\{x\}$. We can assume this path is in $T^{\{a\}}$, hence it avoids $M\setminus\{x\}$. If $x\in M\cap T$, then there is a $T$-tight path $\omega\subset T^{\{a\}}\cup T$ from $x$ to $a$ avoiding $M'$. By \eqref{eq:contain}, $(M')^{\{a\}}\subset T_2^{\{a\}}\subset T^{\{a\}}$, then $x\in T_2$.  If $\omega\cap T'_1=\emptyset$, then $x\in T_1$ by Lemma~\ref{lem:comparable1} (2) and it suffices to consider the constant path at $x$. If $\omega\cap T'_1\neq\emptyset$, then we consider the subpath of $\omega$ from $x$ to when it hits $T'_1$ the first time.

By Lemma~\ref{lem:squeez} $M\in \mc_{\Lambda_P}(\{a\},\{b\})$, and Lemma~\ref{lem:comparable1} and \eqref{eq:contain} imply that $(T_1)^{\{a\}}\subset M^{\{a\}}\subset(T_2)^{\{a\}}\subset T^{\{a\}}$. Thus $M'\in \mathcal C'_{P,T}$ and $T'_1\le M'\le T'_2$.

Now we consider the case $T'_2\subset T^{\{a\}}$ and $T'_1\nsubseteq T^{\{a\}}$. Take $M'$ as before. We will only consider the nontrivial situation $M'$ is not a vertex of $P$. Then either $M'\subset T^{\{a\}}$ or $M'\subset T^{\{b\}}$. Up to symmetry we consider $M'\subset T^{\{a\}}$. Then $M'\in\mc_{\Lambda_T}(\{a\},T'_2)$. By the above discussion, $M'\in \mathcal C'_{P,T}$ and $T'_1\le \{a\}<M\le T'_2$, as desired. Thus Condition 2 holds when at least one of $\{T'_1,T'_2\}$ is contained in $T^{\{a\}}$. By symmetry, the condition also holds when at least one of $\{T'_1,T'_2\}$ is contained in $T^{\{b\}}$. The remaining case of both $T'_1$ and $T'_2$ are vertices of $P$ is trivial.
\end{proof}

\subsection{Vertex links of $\Delta^P_\Lambda$}
Given a vertex $x\in \Delta^P_\Lambda$ of type $\hat T$ with $T\in \mathcal C_P$. Then vertices in $\lk(x,\Delta^P_\Lambda)$ have types in $\mathcal C_{P,T}$. We define a relation $<$ on $\lk(x,\Delta^P_\Lambda)^0$ as follows. Given vertices $x_1,x_2\in \lk(x,\Delta^P_\Lambda)$ of type $\hat S_1,\hat S_2$, we define $x_1<x_2$ if $x_1$ and $x_2$ are adjacent and $S_1<S_2$ in $(\mathcal C_{P,T},<)$.

\begin{lem}
	\label{lem:link of mincut}
Suppose that for any proper induced subdiagram $\Lambda'\subsetneq \Lambda$, $\Delta_{\Lambda'}$ satisfies the labeled 4-cycle property. Suppose $\Lambda\setminus\{p\}$ is connected for any interior point $p\in P$.
Then $(\lk(x,\Delta^P_\Lambda),<)$ is a weakly graded poset such that each pair with a common upper bound have the join and each pair with a common lower bound have the meet.
\end{lem}

\begin{proof}
Let $A_{\Lambda\setminus T}$ be the subgroup of $A_\Lambda$ generated by vertices in $\Lambda\setminus T$.
Up to a left translation, we assume $x$ corresponds to the identity coset $A_{\Lambda\setminus T}$, which contains $A_{\Lambda_T}$ as a direct summand. Suppose $A_{\Lambda\setminus T}\cong A_{\Lambda_T}\oplus H$ where $H$ is another standard parabolic subgroups.
Vertices of $\lk(x,\Delta^P_\Lambda)$ correspond to left cosets of form $gA_{\Lambda\setminus R}$ with $g\in A_{\Lambda\setminus T}$ and $R\in \mathcal C_{P,T}$. Let $R'=\phi(R)=R\cap \Lambda_T$. By the discussion in the beginning of Section~\ref{subsec:graph theoretic},  $$A_{\Lambda\setminus T}\cap gA_{\Lambda\setminus R}=g(A_{\Lambda_T\setminus R'}\oplus H).$$
This implies the following alternative description of $\lk(x,\Delta^P_\Lambda)$: vertices are in 1-1 corresponds with left cosets of form $g A_{\Lambda_T\setminus R'}$ with $g\in A_{\Lambda_T}$ and $R'\in \mathcal C'_{P,T}$. Two vertices are adjacent if their types are comparable with respect to the partial order on $\mathcal C'_{P,T}$ defined in Lemma~\ref{lem:admissible example} and they are contained in the same simplex of $\Delta_{\Lambda_T}$. As $\Lambda_T$ satisfies the labeled 4-cycle property and it is connected, Proposition~\ref{prop:strong} implies that $\Delta_{\Lambda_T}$ satisfies the strong labeled 4-cycle property. Now we are done by Lemma~\ref{lem:admissible} and Lemma~\ref{lem:admissible example}. 
\end{proof}

\begin{proof}[Proof of Proposition~\ref{prop:mincutAn}]
By Lemma~\ref{lem:latticecut}, there is a rank function $$f:\mc_{\Lambda_P}(\{a\},\{b\})\to \mathbb Z.$$
We can extend $f$ to be $f:\mathcal C_P \to \mathbb Z$ such that $f(\{s\})=f(\{b\})+k$ for any interior vertex $s$ of $P$, where $k$ is the distance of $s$ and $b$ in $P$.
Then $f$ is a morphism of partial cyclically ordered sets (with respect to the canonical cyclic order on $\mathbb Z$), i.e. $[x,y,z]$ implies $[f(x),f(y),f(z)]$. Then $f$ induces  $(\Delta^P_\Lambda)^0\to\mathbb Z$ by considering types of vertices in $\Delta^P_\Lambda$, which we also denote by $f$. Given a simplex in $\Delta^P_\Lambda$, then the collection of types of its vertices are pairwise comparable in $\mathcal C_P$, hence forms a cyclically ordered subset of $\mathcal C_P$. Hence Definition~\ref{def:A_n} (2) and (4) follow. Definition~\ref{def:A_n} (3) follows from Lemma~\ref{lem:link of mincut}. It remains to prove the simply-connectedness assertion of Definition~\ref{def:A_n} (1), which follows from Corollary~\ref{cor:Xsc}.
\end{proof}

\section{Propagation of labeled four cycle condition}
\label{sec:propagation}
The goal of this section is the following.
\begin{prop}
	\label{prop:prop}
Suppose $\Lambda$ is a connected Coxeter diagram which is not a tree. Suppose all proper induced subdiagrams $\Lambda'$ of $\Lambda$, $\Delta_{\Lambda'}$ satisfies the labeled 4-cycle condition. Then $\Delta_{\Lambda}$ satisfies the labeled 4-cycle condition.
\end{prop}

In the rest of this section, $x_1x_2x_3x_4$ is a 4-cycle of type $\hat s\hat t\hat s\hat t$ in $\Delta_\Lambda$. It suffices to find $x\in \Delta^0_\Lambda$ which is adjacent to all $x_i$. Indeed, let $\Lambda'$ be a connected induced subdiagram of $\Lambda$ containing $s$ and $t$. If $x$ is of type $\hat r$ with $r\in \Lambda'$, then we are done. If $r\notin \Lambda'$, then $x_1x_2x_3x_4$ can be viewed as a 4-cycle in $\lk(x,\Delta_\Lambda)\cong \Delta_{\Lambda\setminus\{r\}}$. As $\Delta_{\Lambda\setminus\{r\}}$ satisfies the labeled 4-cycle condition by our assumption, there is a vertex $x'$ neighboring to all $x_i$ with type $\hat s'$ such that $s'\in \Lambda'$.

We assume $\Lambda$ is not a cycle, otherwise Proposition~\ref{prop:mincutAn} implies that $\Delta_\Lambda$ is an $\widetilde A_n$-like complex and the proposition follows from \cite[Lem 4.9]{huang2023labeled}.


\subsection{Choice of $P$}
\label{subsec:P}
We endow $\Lambda$ with the path metric $d$ such that each edge has length $1$.
Let $s$ be as above. Let $C\subset\Lambda$ be an induced cycle such that the value $d(s,C)$ is smallest possible.  
Let $\uls\in C$ be a vertex with $d(s,\uls)=d(s,C)$. Note that $\uls$ is unique in $C$, otherwise we can find an induced cycle which is closer to $s$.  Let $P$ be a maximal subpath of $C$ from $\uls$ to vertex $\underline{b}$ such that all interior vertices of $P$ have valence 2 in $\Lambda$ (it is possible that $P$ goes around $C$ once and $\uls=\underline{b}$). Let $X=\Delta^P_\Lambda$ for this specific choice of $P$. Let $\mathcal C_P$ and $\Lambda_P$ be as in Section~\ref{sec:minimal cut complex}. If $\uls\neq\underline{b}$, then we put a partial order on the collection of type II elements in $\mathcal C_P$ as in Lemma~\ref{lem:latticecut}, such that $\{\uls\}<\{\underline{b}\}$. We put a linear order on the collection of type I elements in $\mathcal C_P$ along $P$ such that $\{\underline{b}\}<\{\uls\}$. If $\uls=\underline{b}$, then $P=C$ and elements in $\mathcal C_P$ are vertices of $P$. So $\mathcal C_P$ has a cyclic order induced from a chosen cyclic order on $C$.
Given $T\in \mathcal C_P$, let $\mathcal C_{P,T}$ be as in Section~\ref{sec:minimal cut complex}, endowed with its partial order. 


By Proposition~\ref{prop:mincutAn}, $X$ is an $\widetilde A_n$-like complex. Let $\widehat X$ be the complex constructed from $X$ as in Section~\ref{subsec:A_nlike}.

\begin{lem}
	\label{lem:disjoint}
	Let $Q$ be a shortest edge path in $\Lambda$ from $s$ to $\uls$. Let $T\in \mathcal C_P$ with $T\neq\{\uls\}$. Then $T\cap Q=\emptyset$.
\end{lem}

\begin{proof}
	Suppose $T$ is of type II. We argue by contradiction and take $v\in T\cap Q$. Then $v\neq \uls$. By Remark~\ref{rmk:visit}, there is an edge path $\omega\subset \Lambda_P$ from $\uls$ to $\underline{b}$ with $v\in \omega$, and we can assume $\omega$ is embedded and induced. Then $\omega\cup P$ is an  induced cycle in $\Lambda$ which is closer to $s$ than $C$, contradiction. The case of $T$ being type I is similar.
\end{proof}

\begin{lem}
	\label{lem:separateuls}
	Suppose $s\neq \uls$. Take $T\in \mathcal C_P$ with $T\neq\{\uls\}$. Then $\uls$ separates $s$ from $T$ in $\Lambda$.
\end{lem}

\begin{proof}
First we consider the case $T$ is of type II. We argue by contradiction and let $\eta$ be an edge path from $s$ to $x\in T$ avoiding $\uls$. We can assume $\eta$ is embedded and induced in $\Lambda$.  Let $\Lambda_P$ be as in the beginning of Section~\ref{sec:minimal cut complex}. Then Remark~\ref{rmk:visit} implies that there is an edge path $\eta'\subset \Lambda_P$ from $\uls$ to $\underline b$ passing through $x$ such that $\eta'$ is $T$-tight. We can assume $\eta'$ is embedded and induced in $\Lambda_P$. Then $\eta'\cup P$ gives an embedded and induced cycle in $\Lambda$. Then $(\eta'\cup P)\cap Q=\{\uls\}$, otherwise we found a cycle closer to $s$ than $C$. Let $\eta''$ be the subpath of $\eta'$ from $\uls$ to $x$. Now we consider the closed path $Q\cup \eta\cup\eta''$, where $Q$ is in Lemma~\ref{lem:disjoint}. From $(\eta'\cup P)\cap Q=\{\uls\}$, we know the path $Q\cup \eta''$ is locally embedded at $\uls$. On the other hand, $\uls\notin \eta$. So we can produce an induced and embedded cycle from $Q\cup \eta\cup \eta''$ which is closer to $s$ than $C$, contradiction. The case of $T$ being of type I is similar.
\end{proof}

\subsection{Some generalized 4-cycles in $X$}

\begin{lem}
	\label{lem:generalized 4cycle1}
	Take a vertex $u\in X$ of type $\hat T_u$. Take $z_1,z_2,w\in (\lk^0(u,X),<_{u})$ such that $w$ is either the join or meet of $z_1$ and $z_2$. Let $x$ be a vertex of $\Delta_\Lambda$ of type $\hat s$ with $s\notin T_u$ and $x\sim u$. Suppose $xz_1wz_2$ forms a generalized 4-cycle in $\Delta'_\Lambda$. Then $x\sim w$.
\end{lem}

\begin{proof}
	We consider the case when $w$ is the join of $z_1$ and $z_2$.
	Let $\Lambda'$ be the full subgraph spanned by vertices in $\Lambda\setminus T_{u}$. 
Then $x$ corresponds to a vertex in $\Delta_{\Lambda'}$, and $w,z_1,z_2$ correspond to vertices $w',z'_1,z'_2$ of $\Delta'_{\Lambda'}$ as in the proof of Lemma~\ref{lem:link of mincut}. By Lemma~\ref{lem:4-cycle}, there is a vertex $w''\in \Delta'_{\Lambda'}$ (also viewed as an element in $(\lk^0(u,X),<_{u})$) with the same type as $w'$ such that $w''\sim\{x,z'_1,z'_2\}$. Then $w''$ is an upper bound of $z'_1$ and $z'_2$ in $(\lk^0(u,X),<_{u})$. As $w'$ is the join of $z'_1$ and $z'_2$, $w'\le_u w''$. As $w'$ and $w''$ have the same type, $w'=w''$. So $x\sim w'$, hence $x\sim w$. Then case of $w$ being the meet of $z_1$ and $z_2$ is similar.
\end{proof}

\begin{lem}
	\label{lem:generalized 4cycle2}
	Take a vertex $u\in X$ of type $\hat T_u$. Let $z,w$ be adjacent vertices in $(\lk^0(u,X),<_{u})$ with $z<_u w$. Let $x,y$ be adjacent vertices in $\Delta_\Lambda$ of type $\hat s,\hat t$ with $s,t\notin T_u$, and assume $\{x,y\}\sim u$. Suppose $xyzw$ forms a generalized 4-cycle in $\Delta'_\Lambda$. Then there is a vertex $\theta\in (\lk^0(u,X),<_{u})$ with $z\le \theta\le w$ such that $\theta\sim\{x,y,z,w\}$.
\end{lem}

\begin{proof}
	Let $\Lambda'$ be the full subgraph spanned by vertices in $\Lambda\setminus T_{u}$. 
	Then $x,y$ correspond to vertices in $\Delta_{\Lambda'}$, and $w,z$ correspond to vertices $w',z'$ of $\Delta'_{\Lambda'}$ as in the proof of Lemma~\ref{lem:link of mincut}. Suppose $w$ is of type $\hat T_{w}$. Then $T_{w'}=T_{w}\cap \Lambda'$ and $T_{z'}=T_{z}\cap \Lambda'$. 
	Lemma~\ref{lem:4-cycle} implies there is a vertex $z''$ with the same type as $z'$ such that $z''\sim\{w',x,y\}$. If $z''=z'$, then lemma is proved. Suppose $z''\neq z'$. Then $w'$ is an upper bound for $\{z',z''\}$, when we view these elements as in $(\lk^0(u,X),<_{u})$. By Lemma~\ref{lem:link of mincut}, $z''$ and $z'$ have the join in this poset, represented by $z'_{11}\in \Delta'_{\Lambda'}$. Note that $z'< z'_{11}\le w'$. By Lemma~\ref{lem:generalized 4cycle1}, $z'_{11}\sim y$. Now we repeat this argument with $z'_{11}$ replacing the role $z'$, either to find $x\sim z'_{11}$ and the claim is proved, or we produce $z'_{12}$ with $z'_{12}\sim\{y,w',z'_{11}\}$ and $z'_{11}<z'_{12}\le w'$ (this implies $z'\sim z'_{12}$). As the length of strictly increasing chain in $(\lk^0(u,X),<_{u})$ is bounded, this stops after finitely many steps and we find the desired element. 
\end{proof}

%
%

\begin{lem}
	\label{lem:convex}
Let $x\in \Delta_\Lambda$ be a vertex. Let $Y_x$ be the full subcomplex of $X$ spanned by vertices that are contained in the same simplex of $\Delta_\Lambda$ as $x$. Then $Y_x$ is connected. Moreover, $Y_x$ is a both left locally $B$-convex and right locally $B$-convex subcomplex of $X$ in the sense of Definition~\ref{def:Bconvex}.
\end{lem}

\begin{proof}
Suppose $x$ is of type $\hat s$. We assume $s\notin P$, otherwise $x\in \Delta^P_\Lambda$ and $Y_x$ is the close star of $x$ in $\Delta^P_\Lambda$ which is connected.
Without loss of generality, we assume $x$ corresponding to the identity coset $A_{\hat s}$. Then vertices in $Y_x$ corresponds to left cosets of form $gA_{\hat T}$ such that $T\in \mathcal C_P$ and $gA_{\hat T}\cap A_{\hat s}\neq\emptyset$. Take two distinct vertices $s_1$ and $s_2$ in $P$. By definition of $\Delta^P_\Lambda$, any vertex in $Y_x$ is adjacent to a vertex of type $\hat s_1$ in $Y_x$. However, the full subcomplex of $Y_x$ spanned by all type $\hat s_1$ or $\hat s_2$ vertices is a copy of $\Delta_{\Lambda\setminus\{s\},s_1s_2}$, which is connected by Lemma~\ref{lem:c}. Thus $Y_x$ is connected. 


We only treat the left locally $B$-convex case, as the other case is similar.
Given vertex $y\in Y^0_x$ and a pair of vertices $y_1,y_2$ of $\lk(y,Y_x)$. We assume they have a meet $z=y_1\wedge y_2$ in the poset $(\lk(y,X),\le_y)$. Suppose $x$ has type $\hat s$, and $y$ has type $\hat T$. If $s\in T$, then $x\sim z$ since $y\sim z$. If $s\notin T$, then $x\sim z$ by Lemma~\ref{lem:generalized 4cycle1}.
\end{proof}


\subsection{Reduction}
Let $x_1,x_2,x_3,x_4,s,t$ be as in the beginning of Section~\ref{sec:propagation}.

\begin{lem}
	\label{lem:reduction}
Suppose there is an embedded edge path $w_1\cdots w_n$ in $\Delta_\Lambda$ from $x_1$ to $x_3$ such that each $w_i$ is adjacent to both $x_2$ and $x_4$. Then there is an $x\in \Delta^0_\Lambda$ adjacent to all $x_i$, hence Proposition~\ref{prop:prop} holds.
\end{lem}

\begin{proof}
We induct on $n$. The base case $n=3$ is clear. Suppose $w_i$ has type $\hat s_i$ for vertex $s_i\in \Lambda$, and let $\Lambda_i$ be the connected component of $\Lambda\setminus\{s_i\}$ containing $t$.

Given two consecutive vertices $w_i,w_{i+1}$. 
We assume $s_{i+1}\in \Lambda_i$ (if this is not true, then as $\Lambda$ is connected, $t$ and $s_i$ are contained in the same component of $\Lambda\setminus\{s_{i+1}\}$ and we switch the role of $w_i$ and $w_{i+1}$ for the following discussion). Let $\mathcal P$ be the collection of vertices in $\Delta'_{\Lambda_{i}}$ with type $\hat C$ for some $C\in\mc_{\Lambda_{i}}(\{t\},\{s_{i+1}\})$. We endow $\mathcal P$ with the partial order as in Corollary~\ref{cor:strong} such that elements of type $\hat t$ are minimal. In particular, $x_2,x_4,w_{i+1}$ are elements of $\mathcal P$ and $w_{i+1}$ is a common upper bound of $x_2,x_4$. By Proposition~\ref{prop:strong} and Corollary~\ref{cor:strong}, $x_2$ and $x_4$ have the join in $\mathcal P$, denoted by $\theta$. Then $\theta\sim w_{i+1}$. By an argument similar to Lemma~\ref{lem:generalized 4cycle1}, $\theta\sim w_{i-1}$.

Assume $\theta$ has type $\hat B$ with $B\in \mc_{\Lambda_{i}}(\{t\},\{s_{i}\})$. 
If $B=\{s_{i+1}\}$, then $\theta=w_{i+1}$ and $w_{i-1}\sim w_{i+1}$, which decreases the length of the path from $x_1$ to $x_3$ and we are done by induction assumption. Now suppose $B\neq\{s_{i+1}\}$. Then $B\subset \Lambda_{i+1}$.

We first consider the case $s_i\in \Lambda_{i+1}$.
Let $\mathcal P'$ be the collection of vertices in $\Delta'_{\Lambda_{i+1}}$ with type $\hat C$ for some $C\in\mc_{\Lambda_{i+1}}(\{t\},B\cup\{s_i\})$, endowed with the partial order as in Corollary~\ref{cor:strong} such that elements of type $\hat t$ are minimal. Then $w_i$ and $\theta$ together determine a maximal element in $\mathcal P'$ which is an upper bound for $x_2,x_4$.
Then we argue as before to obtain $\theta'$ of type $\hat B'$ with $B'\in \mc_{\Lambda_{i+1}}(\{t\},B\cup\{s_i\})$ such that $\theta'$ is the join of $\{x_2,x_4\}$ in $\mathcal P'$, and $\theta'\sim \{w_i,\theta,w_{i+2}\}$. 

We claim $B'\setminus\{s_i\}$ separates $t$ from $B$ in $\Lambda\setminus\{s_i\}$. Indeed, take a path $\omega$ in $\Lambda\setminus\{s_i\}$ from $t$ to a point in $B$. As $B'\in \mc_{\Lambda_{i+1}}(\{t\},B\cup\{s_i\})$, we are done if $s_{i+1}\notin\omega$. Now assume $s_{i+1}\in\omega$. 
We take the subpath $\omega'$ of $\omega$ from $t$ until it first meets $s_{i+1}$. As $B\in \mc_{\Lambda_i}(\{t\},s_{i+1})$, $\omega'$ must meets $B$ and we consider the subpath $\omega''$ of $\omega'$ from $t$ to the first time when it meets $B$. Then $\omega''\in \Lambda_{i+1}$ by the choice of $\omega'$, hence $\omega''$ meets $B'$, as desired. 

Thus we find $B''\subset B'$ such that $B''\in \mc_{\Lambda_{i}}(\{t\},B)$. Let $\theta''$ be the vertex of type $\hat B''$ such that $\theta''\sim\theta'$. As $\theta'\sim \{x_2,x_4,w_i,\theta,w_{i+2}\}$, we obtain $\theta''\sim \{x_2,x_4,w_i,\theta,w_{i+2}\}$. Lemma~\ref{lem:squeez} implies that $B''\in \mc_{\Lambda_i}(\{t\},\{s_{i+1}\})$. Thus $\theta''$ is an element of $\mathcal P$. By Lemma~\ref{lem:squeez} again, $\theta''\le \theta$ in $\mathcal P$. As $\theta''$ is an upper bound for $\{x_2,x_4\}$, we know $\theta''=\theta$. Thus $\theta\sim w_{i-1}$ and $\theta\sim w_{i+2}$, which decreases the length of the path from $x_1$ to $x_3$.

The case $s_i\notin \Lambda_{i+1}$ is similar. We instead let $\mathcal P'$ to be the collection of vertices in $\Delta'_{\Lambda_{i+1}}$ with type $\hat C$ for some $C\in\mc_{\Lambda_{i+1}}(\{t\},B)$ and let $\theta'$ to be the join of $x_2$ and $x_4$ in $\mathcal P'$. Then $\theta'\sim\{\theta,w_{i+2}\}$. We still have $\theta'\sim w_i$ as $\Delta_{\Lambda_{i+1}}$ is a join factor of $\Delta_{\Lambda\setminus\{s_{i+1}\}}$. The rest of the argument is similar.
\end{proof}


\begin{lem}
	\label{lem:degen}
	Suppose one of $\{s\}$ and $\{t\}$ is an element in $\mathcal C_P$. Then the assumption of Lemma~\ref{lem:reduction} holds, hence Proposition~\ref{prop:prop} holds.
\end{lem}

\begin{proof}
Suppose $\{s\}$ is an element of $\mathcal C_P$. Then $x_1,x_3$ are vertices of $X$. Let $Y_{x_2},Y_{x_4}$ be as in Lemma~\ref{lem:convex}. Then $x_1,x_3\in Y_{x_2},Y_{x_4}$. Lemma~\ref{lem:convex} and Proposition~\ref{prop:convex} imply that $Y_{x_2}\cap Y_{x_4}$ is connected, hence contained an edge path $z_1\cdots z_n$ from $x_1$ to $x_3$. Suppose $z_i$ has type $\hat T_i$. For each $i$, let $y_i$ be vertex of $\Delta_\Lambda$ of type $\hat s_i$ with $s_i\in T_i$ and $y_i\sim z_i$. As $z_i\sim \{x_2,x_4\}$, we obtain $y_i\sim\{x_2,x_4\}$ and $y_i\sim y_{i+1}$. Thus the assumption of Lemma~\ref{lem:reduction} is satisfied.
\end{proof}

From now on, we assume $\{s\}$ and $\{t\}$ are not elements in $\mathcal C_P$.
In later subsections, we  show the assumption of Lemma~\ref{lem:reduction} still holds, hence finish the proof.

\subsection{Distance minimizing quadruples}
For $1\le i\le 4$, let $Y_i=Y_{x_i}$ be as defined in Lemma~\ref{lem:convex}, which is $B$-convex. Let $\Xi$ be the collection of all quadruple $(y_1,y_2,y_3,y_4)$ of vertices in $\Delta^P_\Lambda$ with $y_i\in Y_i\cap Y_{i+1}$ for all $i\in\mathbb Z/4\mathbb Z$. Let $d$ denotes the path metric on $X^1$ with edge length 1. 

Let $\varphi:\widehat X^0\to\widehat X^0$ be as in Section~\ref{subsec:widehatX}. Recall that there are two relations on $\widehat X^0$ (Section~\ref{subsec:widehatX}), $<$ and its transitive closure denoted by $<_t$.

\begin{lem}
	\label{lem:distance minimizing}
Let $(v_1,v_2,v_3,v_4)\in \Xi$ be an element with minimal possible value of 
\begin{equation}
	\label{eq:ds}
d(v_1,v_2)+d(v_2,v_3)+d(v_3,v_4)+d(v_4,v_1).
\end{equation}
 Then either $v_i= v_{i+1}$ for some $i$, or $d(v_1,v_2)=d(v_3,v_4)$ and $d(v_2,v_3)=d(v_4,v_1)$.
\end{lem}


\begin{proof}
Suppose $v_i\neq v_{i+1}$ for all $i$.
It suffices to prove $d(v_1,v_2)+d(v_2,v_3)=d(v_3,v_4)+d(v_4,v_1)$, and $d(v_2,v_3)+d(v_3,v_4)=d(v_4,v_1)+d(v_1,v_2)$. We only prove the second equality as the first will be similar. Assume by contradiction that $d(v_2,v_3)+d(v_3,v_4)<d(v_1,v_2)+d(v_2,v_3)$. Let $\Xi'\subset \Xi$ made of $(y_1,y_2,y_3,y_4)$ with $d(y_2,y_3)\le d(v_2,v_3)$ and $d(y_3,y_4)\le d(v_3,v_4)$. Let $\ad$ be the asymmetric distance discussed in Section~\ref{subsec:bestvina asymmetric}. In $\Xi'$ we select $(u_1,u_2,u_3,u_4)$ which minimizes the value of
\begin{equation}
	\label{eq:ads}
\ad(u_2,u_3)+\ad(u_1,u_3)+\ad(u_4,u_3).
\end{equation}
Let $\omega=z_1\cdots z_n$ be the left $B$-geodesic from $u_2$ to $u_3$, and let $\omega'=z'_1\cdots z'_m$ be the left $B$-geodesic from $u_4$ to $u_3$. We will show
\begin{enumerate}
	\item $\ad(u_1,z_1)<\ad(u_1,z_2)$ and $\ad(u_1,z'_1)<\ad(u_1,z'_2)$;
	\item the concatenation of the left $B$-geodesic from $u_2$ to $u_1$ and the left $B$-geodesic from $u_1$ to $u_4$ is the left $B$-geodesic from $u_2$ to $u_4$.
\end{enumerate}
Note that Assertion (2) leads to a contradiction as it implies $$d(u_2,u_1)+d(u_1,u_4)=d(u_2,u_4)\le d(u_2,u_3)+d(u_3,u_4)\le d(v_2,v_3)+d(v_3,v_4).$$ So $(u_1,u_2,u_3,u_4)$ gives an element with smaller value of \eqref{eq:ds}, contradiction. It remains to prove Assertions (1) and (2).

\smallskip

For Assertion (1), let $\omega_{12}$ be the right $B$-geodesic from $u_1$ to $u_2$ and let $w_2$ be the vertex in $\omega_{12}$ adjacent to $u_2$. By Lemma~\ref{lem:convex} and Proposition~\ref{prop:convex}, $\omega_{12}\subset Y_2$. We consider admissible lifts of $\omega_{12}$ and the left $B$-geodesic from $u_1$ to $z_2$, starting at the same vertex $\hat u_1\in \widehat X$ and ending at $\hat u_2,\hat z_2$ respectively. As $\ad(u_1,z_2)<\ad(u_1,u_2)$, by Lemma~\ref{lem:mn}, $\hat z_2<\hat u_2$. Let $\hat w_2$ be the induced lift of $w_2$. Then the right greedy property of $\omega_{12}$ implies that $\hat w_2\le \hat z_2$. Thus $z_2$ and $w_2$ are adjacent in $\Delta^P_\Lambda$ and $w_2\le_{u_2} z_2$ in $(\lk^0(u_2,\Delta^P_\Lambda),<_{u_2})$. If $z_2= w_2$, then $w_2\in Y_2\cap Y_3$ and replacing $u_2$ by $w_2$ decreases the value of \eqref{eq:ads}. Indeed, we consider admissible lifts of $u_2z_2\cdots z_n$ and $w_2z_2\cdots z_n$ with the same endpoint $\hat z_n$. The two lifts coincide on the part $z_2\cdots z_n$, however, as $u_2<_{z_2} w_2$ in $(\lk^0(z_2,\Delta^P_\Lambda),<_{z_2})$, we have $\hat u_2<\hat w_2$. Thus
$$
\ad(w_2,u_3)\le r(\hat z_n)-r(\hat w_2)<r(\hat z_n)-r(\hat u_2)=\ad(u_2,u_3).
$$
After replacing $u_2$ by $w_2$, we still obtain an element in $\Xi'$, contradiction.

Now suppose $z_2\neq w_2$.
We assume $t\notin T_{u_2}$ where $u_2$ is of type $\hat T_{u_2}$, otherwise $u_2\sim z_2$ implies $x_2\sim z_2$, hence $z_2\in Y_2\cap Y_3$ and we can replace $u_2$ by $z_2$ to reduce the value of \eqref{eq:ads} but still remain in $\Xi'$. Similarly, we assume $s\notin T_{u_2}$ (otherwise we replace $u_2$ by $w_2$).
As $x_2x_3z_2w_2$ gives a generalized 4-cycle in $\Delta'_\Lambda$, by Lemma~\ref{lem:4-cycle},  there is a vertex $\theta\in (\lk^0(u_2,\Delta^P_\Lambda),<_{u_2})$ with $z_2\le \theta\le w_2$ such that $\theta\sim\{x_2,x_3,z_2,w_2\}$.
Note that $\theta\in Y_2\cap Y_3$. We can assume $\theta\neq z_2$, then $u_2<_{z_2}\theta$ in $\lk^0(z_2,\Delta^P_\Lambda)$. By the argument in the previous paragraph, replacing $u_2$ by $\theta$ decreases the value of \eqref{eq:ads} and keeps the tuple in $\Xi'$, contradiction. Thus Assertion (1) is proved (the other part is similar).


\smallskip

Now we prove Assertion (2). Let $\omega'_{12}$ and $\omega'_{14}$ be the left B-geodesic from $u_1$ to $u_2,u_4$ respectively. Similarly we define $\omega'_{23}$ and $\omega'_{43}$. For $i=2,4$, let $\theta_i$ be the vertex on $\omega'_{1i}$ that is adjacent to $u_1$.  By Proposition~\ref{prop:bnpc}, Proposition~\ref{prop:mincutAn} and Assertion (1), we know $\ad(u_1,z_i)<\ad(u_1,z_{i+1})$ and $\ad(u_1,z'_i)<\ad(u_1,z'_{i+1})$ for any $i$. We consider admissible lifts of the paths $\omega'_{12}z_1\cdots z_n$ and $\omega'_{14}z'_1\cdots z'_m$ in $\widehat X$ starting from the same vertex $\hat u_1$. Then Lemma~\ref{lem:mn} implies that these two lifts end in the same vertex $\hat u_3=\hat z_n=\hat z'_m$, and $\ad(u_1,u_3)=r(\hat u_3)-r(\hat u_1)$. Thus $\hat u_1<_t\hat \theta_i<_t\hat u_3$ for $i=2,4$. In particular, $\hat \theta_2,\hat\theta_4$ have the join (with respect to $<_t$) which is $<\varphi(\hat u_1)$. Consequently, $\theta_2$ and $\theta_4$ have the join $\bar \theta$ in $(\lk^0(u_1,X),<_{u_1})$. 

Suppose Assertion (2) fails. Then by Lemma~\ref{lem:B-geodesics} (2), $\theta_2$ and $\theta_4$ also have the meet $\underline\theta$ in $(\lk^0(u_1,X),<_{u_1})$.
Assume $s\notin T_{u_1}$ and $t\notin T_{u_1}$ (where $u_1$ has type $\hat T_{u_1}$), otherwise we can replace $u_3$ by $\theta_2$ or $\theta_4$ to decrease \eqref{eq:ads}. Let $\Lambda_1$ be the full subgraph spanned by vertices in $\Lambda\setminus T_{u_1}$. Let  $\bar\theta',\underline\theta',\theta'_2,\theta'_4$ be vertices in $\Delta'_{\Lambda_1}$ corresponding to $\bar\theta,\underline\theta,\theta_2,\theta_4$ respectively (as in the proof of Lemma~\ref{lem:link of mincut}). As $$\underline\theta'\le_{u_1}\theta'_2\le_{u_1}\bar\theta',$$ Lemma~\ref{lem:admissible example} and Definition~\ref{def:admissible} (3) implies $T_{\theta'_2}$ separates $T_{\underline\theta'}$ from $T_{\bar\theta'}$ in $\Lambda_1$. Thus in $\Lambda_1$, $T_{\theta'_2}$ separates $t$ from one of $\{T_{\underline\theta'},T_{\bar\theta'}\}$, say $T_{\underline\theta'}$. Then Lemma~\ref{lem:transitive} implies that $x_2\sim \underline\theta'$. By considering the generalized 4-cycle $x_2\underline\theta\theta_4x_1$ and applying Lemma~\ref{lem:generalized 4cycle2}, we find $\theta\in \lk(u_1,X)$ such that $\theta\sim\{x_2,\underline\theta,\theta_4,x_1,u_1\}$ and $\underline\theta\le\theta\le\theta_4$ in $ (\lk^0(u_1,X),<_{u_1})$. In particular, $\theta\in Y_1\cap Y_2$.
Replacing $u_1$ by $\theta$ decreases the value of \eqref{eq:ads}. Indeed, the path $u_1\underline\theta\theta\theta_4$ has an admissible lift to a path from $\hat u_1$ to $\hat \theta_4$. Let $\hat\theta$ be the lift of $\theta$. Then $\hat u_2<_t\hat\theta<_t\hat u_3$. Note that the path in $X$ from $\theta$ to $\theta_4$, then follow $\omega_{14}$ and $z'_1\cdots z'_m$ to $u_3$ have an admissible lift starting at $\hat\theta$ and ending at $\hat u_3$. Thus
$$
\ad(\theta,u_3)\le r(\hat u_3)-r(\hat \theta)<r(\hat u_3)-r(\hat u_1)=\ad(u_1,u_3).
$$
The case $x_2\sim \bar\theta'$ can be handled in a similar way. This finishes the proof.
\end{proof}


Take $(u_1,u_2,u_3,u_4)$ as in Lemma~\ref{lem:distance minimizing} which minimizes \eqref{eq:ds} among $\Xi$. Let $\rho_{ij}$ be the left $B$-geodesic from $u_i$ to $u_j$.
We can assume in addition $\rho_{21}\cup\rho_{14}=\rho_{24}$ -- indeed, we consider an element minimizing \eqref{eq:ads} among $\Xi'$, which satisfies $\rho_{21}\cup\rho_{14}=\rho_{24}$ by the proof of Assertion (2). By the definition of $\Xi'$, such an element still minimizes \eqref{eq:ds} among $\Xi$. Our next goal, occupying Section~\ref{subsec:disk} to Section~\ref{subsec:case 2}, is the following.
\begin{lem}
	\label{lem:key}
There exists $i\in \mathbb Z/4\mathbb Z$ such that $u_i=u_{i+1}$.
\end{lem}

\begin{lem}
Assuming Lemma~\ref{lem:key}. Then Proposition~\ref{prop:prop} holds.
\end{lem}

\begin{proof}
We assume without loss of generality that $u_2=u_3$. Then $Y_2\cap Y_3\cap Y_4$ is non-empty. Now we consider all triples $(v_1,v_2,v_3)$ with $v_1\in Y_1\cap Y_2$, $v_2\in Y_2\cap Y_4$ and $v_3\in Y_4\cap Y_1$, and take one such triple which minimizes $\ad(v_2,v_1)+\ad(v_2,v_3)$. We claim $\{v_1,v_2,v_3\}$ are identical. Suppose by contradiction that $\{v_1,v_2,v_3\}$ are mutually distinct. Let $\omega=z_1\cdots z_n$ be the left $B$-geodesic from $v_1=z_1$ to $v_3=z_n$. Similar to the proof of Assertion (1) in Lemma~\ref{lem:distance minimizing} (and using that $x_2\sim x_1$, $x_1\sim x_4$), we have $\ad(v_2,z_1)<\ad(v_2,z_2)$ and $\ad(v_2,z_n)<\ad(v_2,z_{n-1})$. However, this contradicts Proposition~\ref{prop:bnpc}. Thus the claim follows.

By the claim, $Y_1\cap Y_2\cap Y_4\neq\emptyset$. By Lemma~\ref{lem:convex} and Proposition~\ref{prop:convex}, $Y_2\cap Y_4$ is connected. By considering an edge path in $Y_2\cap Y_4$ from a vertex in 
	$Y_2\cap Y_3\cap Y_4$ to a vertex in $Y_2\cap Y_3\cap Y_4$, we can produce (as in Lemma~\ref{lem:degen}) the path satisfying the assumption of Lemma~\ref{lem:reduction}, which finishes the proof.
\end{proof}

It remains to prove Lemma~\ref{lem:key}.
We argue by contradiction and assume $u_i\neq u_{i+1}$ for all $i\in \mathbb Z/4\mathbb Z$ in Section~\ref{subsec:disk} to Section~\ref{subsec:case 2}. Then Lemma~\ref{lem:distance minimizing} implies $d(u_1,u_2)=d(u_3,u_4)$ and $d(u_2,u_3)=d(u_1,u_4)$. In the end we will prove there is a different quadruple with a smaller value of \eqref{eq:ds}, leading to a contradiction. Up a symmetry of the cycle $x_1x_2x_3x_4$, we can assume:
\begin{equation}
	\label{eq:assumption}
d(u_4,u_1)\le d(u_4,u_3).
\end{equation}
This symmetry  brings us two cases to consider, Case 1 being $x_4$ has type $\hat t$ and $x_1$ has type $\hat s$; and Case 2 being $x_4$ has type $\hat s$ and $x_1$ has type $\hat t$.

\subsection{A disk diagram}
\label{subsec:disk}
Suppose $\rho_{34}=a_{n1}a_{n2}\cdots a_{nm}$ where $a_{n1}=u_3$ and $a_{nm}=u_4$. We consider admissible lifts of $\rho_{21}\cup\rho_{14}$ and $\rho_{23}\cup\rho_{34}$ to $\widehat X$, starting at the same vertex $\hat u_2$. 
As $\rho_{21}\cup\rho_{14}$ is distance minimizing, so is $\rho_{23}\cup\rho_{34}$ by Lemma~\ref{lem:distance minimizing}, hence $d(a_{ni},u_2)<d(a_{n,i+1},u_2)$ for $1\le i<m$. By Lemma~\ref{lem:mn}, $\hat a_{ni}<\hat a_{n,i+1}$ and these two lifts have the same ending vertex $\hat u_4=\hat a_{nm}$. By considering the procedure of obtaining the left normal from $\hat u_2$ to $\hat a_{n,i+1}$ using the left normal from $\hat u_2$ to $\hat a_{ni}$ as in Theorem~\ref{thm:replace}, we produce a triangulated disk $D$ as in Figure~\ref{fig:systolic} (I) which is a union of strips as in Figure~\ref{fig:strip} (we highlight the first strip by gray) and a map $f$ from vertices of $D$ to $\widehat X^0$. The thickened line in $D$ (Figure~\ref{fig:systolic} (I)) is mapped to the admissible lift of $\rho_{21}\cup\rho_{14}$ starting at $\hat u_2$. By \eqref{eq:assumption} and repeatedly applying Theorem~\ref{thm:intial}, we can prove by induction on $i$ that the left vertical edge of the $i$-th strip in Figure~\ref{fig:systolic} (I) coincides with the lift of the $i$-th edge of $\rho_{21}\cup\rho_{14}$.
We will slightly abuse notation and use the same symbol to denote a vertex in $D$ and its $f$-image (note that two different vertices of $D$ might be mapped to the same vertex of $\widehat X$).  We label the vertices in $D$ as follows: vertices on the left most vertical line are labeled by $\hat a_{11},\hat a_{12},\cdots,\hat a_{1m}$ from bottom to top; vertices on the left most but one vertical line are labeled by $\hat a_{21},\hat a_{22},\cdots,\hat a_{2m}$, etc. Assume there are $n$ vertical lines in $D$. Let $a_{ij}$ be the image of $\hat a_{ij}$ in $X$.

\begin{figure}
	\centering
	\includegraphics[scale=0.85]{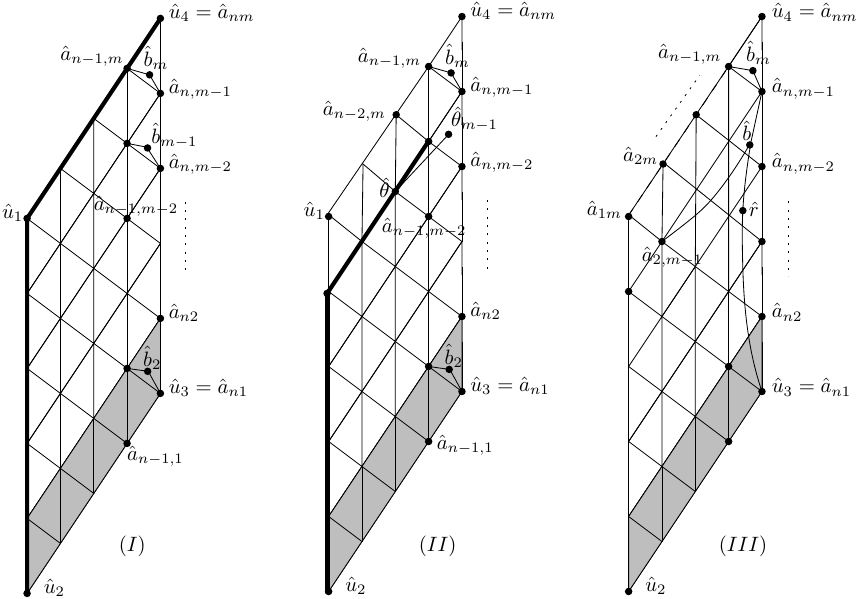}
	\caption{The disk $D$.}
	\label{fig:systolic}
\end{figure}

As $u_i\neq u_{i+1}$ for any $i$, we have $a_{n-1,m}\neq a_{nm}$ and $a_{n,m-1}\neq a_{nm}$.
Assume $a_{n-1,m}\neq a_{n,m-1}$, otherwise $a_{n-1,m}\in Y_1\cap Y_4$ and we can replace $u_4$ by $a_{n-1,m}$ to decrease \eqref{eq:ds}.
By considering the generalized 4-cycle $x_1a_{n-1,m}a_{n,m-1}x_4$ and applying Lemma~\ref{lem:generalized 4cycle2}, we know there is a vertex $b_m\in \lk(u_4,X)$ such that $a_{m-1,n}\le_{u_4} b_m\le_{u_4} a_{n-1,m}$ and $b_m\sim\{x_1,x_4\}$. So $b_m\in Y_1\cap Y_4$. 

We assume $b_m\neq a_{n-1,m}$ and $b_m\neq a_{m-1,n}$, otherwise we can replace $u_4$ by $b_m$ to decrease \eqref{eq:ds}. Let $\hat a_{m-1,n}\hat b_m\hat a_{n-1,m}$ be an admissible lift of $a_{m-1,n}b_ma_{n-1,m}$. Then $\hat a_{m-1,n}\le \hat b_m\le\hat a_{n-1,m}$.

We claim the admissible lift of the left $B$-geodesic from $a_{n1}$ to $b_m$ starting at $\hat a_{n1}$ must end at $\hat b_m$. Indeed, as $\varphi(\hat a_{n1})$ is not bounded above by $\hat a_{nm}$ (with respect to $<_t$), the same holds with $\hat a_{nm}$ replaced by $\hat b_m$. By considering admissible lift of $a_{n1}\cdots a_{n,m-1}b_m$ and apply Lemma~\ref{lem:mn}, the claim follows.


When $m\ge 3$, we assume $b_m$ is not adjacent to $a_{n,m-2}$, otherwise $d(u_3,b_m)<d(u_1,u_4)$ and replacing $u_4$ by $b_m$ decreases \eqref{eq:ds}. Let $\hat b_{m-1}=\beta(\hat a_{n,m-2}\hat b_m)$ where $\beta$ is defined in Section~\ref{subsec:normal form}. Then $b_{m-1}$ is on the right $B$-geodesic from $b_m$ to $a_{n,m-2}$. As $b_m,a_{n,m-2}\in Y_4$, by Lemma~\ref{lem:convex}, $b_{m-1}\in Y_4$. As $\hat a_{n,m-2}\le \hat a_{n-1,m-1}<\hat b_m$ and $d(b_m,a_{n,m-2})=2$, we know $\hat a_{n,m-2}<\hat b_{m-1}\le\hat a_{n-1,m-1}$. Define $\{\hat b_i\}_{i=2}^{m-1}$ inductively by $\hat b_i=\beta(\hat a_{n,i-1}\hat b_{i+1})$. Then $b_i\sim x_4$ and $\hat a_{n,i-1}\le \hat b_i\le \hat a_{n-1,i}$ for $2\le i\le m$. We assume $d(u_3,b_m)=d(u_1,u_4)$, otherwise we can replace $u_4$ by $b_m$. Thus $\hat a_{n,i-1}< \hat b_i\le \hat a_{n-1,i}$ for $2\le i\le m$.

Suppose $a_{ij}$ has type $\hat A_{ij}$, and $b_i$ has type $\hat B_i$.


\begin{lem}
	\label{lem:nonadj1}
We have $x_4\nsim a_{n-1,i}$ for $2\le i\le m$ (in particular $\hat a_{n-1,i}$ and $\hat a_{n,i-1}$ give different vertices in $\widehat X$). 
\end{lem}

\begin{proof}
We refer to Figure~\ref{fig:systolic} (II) for the following proof.	
We only prove treat the case of $x_4\nsim a_{n-1,2}$ as the other cases are similar. If $m=2$, then we can replace $u_4$ by $a_{n-1,2}$ to decrease \eqref{eq:ds}.
We assume $m\ge 3$.	Then $\hat a_{n-1,2}\le_t \hat b_m$. Let $\hat a_{n-1,2}\hat \theta_3\hat\theta_4\cdots\hat\theta_{m-1}\hat b_m$ be the left normal form from $\hat a_{n-1,2}$ to $\hat b_m$ (Figure~\ref{fig:systolic} (II)). As $x_4\sim\{b_m,a_{n-1,2}\}$, by Lemma~\ref{lem:convex}, $x_4\sim \theta_i$ for $3\le i\le m-1$.
As $\hat a_{n-1,m-1}\le \hat b_m$, we deduce that $\hat a_{n-1,i}\le_t \hat \theta_{i}$ for $3\le i\le m-1$ by repeatedly applying Theorem~\ref{thm:replace}.

Let $\eta$ be the thickened path in Figure~\ref{fig:systolic} (II) from $\hat u_2$ to $\hat a_{n-1,m-1}$. Let $\eta'$ be the edge path $$\hat a_{11}\hat a_{21}\cdots \hat a_{n-1,1}\hat a_{n-1,2}\hat \theta_3\hat\theta_4\cdots\hat\theta_{m-1}.$$ Then $\eta$ and $\eta'$ have the same length.
Let $\hat \theta=\hat a_{n-2,m-1}$. As $\eta$ is a left normal form in $\widehat X$ by the construction of $D$, by Theorem~\ref{thm:replace}, $\hat\theta$ is adjacent to $\hat\theta_{m-1}$. Thus $d(\hat \theta_{m-1},\hat u_1)\le d(\hat a_{n-1,m},\hat u_1)$, hence $d(\theta_{m-1},u_1)\le d(a_{n-1,m},u_1)$. In particular, $\hat\theta_{m-1}\neq\hat b_m$ in $\widehat X$, otherwise $d(u_1,b_m)<d(u_1,u_4)$ and replacing $u_4$ by $b_m$ decreases \eqref{eq:ds}. 

Next we show $\hat \theta_{m-1}\le\hat a_{n-2,m}$. It suffices to show $\hat a_{n-2,m}\le \varphi(\hat\theta_{m-1})$. Note that $\varphi(\hat u_1)\le_t \varphi(\hat \theta_{m-1})$ is not true, otherwise $\hat u_1\le_t \hat \theta_{m-1}\le \hat b_m\le \hat a_{n-1,m}$, and by Lemma~\ref{lem:smaller dist}, $d(\hat u_1,\hat b_m)\le d(\hat u_1,\hat a_{n-1,m})<d(\hat u_1, \hat u_4)$ and replacing $u_4$ by $b_m$ decreases \eqref{eq:ds}. Consider the path $\hat a_{1m}\hat a_{2m}\cdots\hat a_{n-1,m}\varphi(\theta_{m-1})$ (note that $\hat a_{n-1,m}\le \varphi(\hat\theta_{m-1})$ as $\hat a_{n-1,m-1}\le\hat\theta_{m-1}\le\hat a_{n-1,m}$). The previous discussion implies that $\alpha(\hat u_1\varphi(\hat\theta_{m-1}))<\varphi(\hat u_1)$ ($\alpha$ is defined in Section~\ref{subsec:normal form}). Then Lemma~\ref{lem:mn} implies that the admissible lift of the left $B$-geodesic from $u_1$ to $\theta_{m-1}$ starting at $\hat u_1$ must end at $\varphi(\hat\theta_{m-1})$. As $d(u_1,\theta_{m-1})\le d(u_1,a_{n-1,m})$, Theorem~\ref{thm:replace} implies that $\hat a_{n-2,m}\le \varphi(\hat\theta_{m-1})$, as desired. Actually $\hat \theta_{m-1}<\hat a_{n-2,m}$, otherwise we can replace $u_4$ by $b_m$ to decrease \eqref{eq:ds}.

Recall that $\theta_{m-1}\sim x_4$. Consider the generalized 4-cycle $\theta_{m-1}a_{n-1,m}x_1x_4$ around $b_m$. Lemma~\ref{lem:generalized 4cycle2} implies that there is $u\in \lk^0(b_m,X)$ such that $a_{n-1,m}\le_{b_m}u\le_{b_m} \theta_{m-1}$ in $(\lk^0(b_m,X),\le_{b_m})$ and $u\sim\{x_1,x_4\}$. Thus $u\in Y_1\cap Y_4$ and $a_{n-1,m}\le_{\theta_{m-1}} u$ in $(\lk^0(\theta_{m-1},X),\le_{\theta_{m-1}})$.
By previous paragraph, $a_{n-2,m},a_{n-1,m}\in \lk(\theta_{m-1},X)$ and $a_{n-2,m}\le_{\theta_{m-1}}a_{n-1,m}$. It follows that $a_{n-2,m}\le_{\theta_{m-1}}u$, and $a_{n-2,m}$ and $u$ are adjacent in $X$. Then $d(u_1,u)<d(u_1,u_4)$, and 
\begin{align*}
d(u_3,u)&\le d(u_3,a_{n-1,2})+d(a_{n-1,2},\theta_{m-1})+d(\theta_{m-1},u)\\
&=d(a_{n-1,2},\theta_{m-1})+2\le d(u_3,u_4).
\end{align*}
Thus replacing $u_4$ by $u$ decreases \eqref{eq:ds}. The lemma follows. 
\end{proof}


\begin{lem}
	\label{lem:nonadj2}
We have $x_1\nsim a_{i,m-1}$ for $2\le i\le n$. 
\end{lem}

\begin{proof}
We refer to Figure~\ref{fig:systolic} (III) for the following proof. We assume $d(b_m,u_3)=d(u_4,u_3)$, otherwise replacing $u_4$ by $b_m$ decreases \eqref{eq:ds}. 
	
We assume $n\ge 3$, otherwise the lemma is already clear. Then \eqref{eq:assumption} implies $m\ge 3$.
We only show $x_1\nsim a_{2,m-1}$, as other cases are similar. Assume by contradiction that $x_1\sim a_{2,m-1}$. Let $\hat b=\beta(\hat a_{2,m-1}\hat b_m)$. By Lemma~\ref{lem:convex}, $b\sim x_1$. As $\hat a_{2,m-1}\le_t\hat a_{n-1,m-1}\le \hat b_m$, we know $\hat b\le \hat a_{n-1,m-1}$.
By \eqref{eq:assumption}, $\hat a_{n1}\le_t \hat a_{2,m-1}$. So $\hat a_{n1}\le_t \hat b$. Let $\eta$ be the left normal form path from $\hat a_{n1}$ to $\hat b$. As $\hat a_{2,m-1}\le_t\hat a_{n,m-1}\le \hat b_m$, we know $\hat a_{n1}\le_t\hat b\le \hat a_{n,m-1}$.
Thus length$(\eta)\le d(\hat a_{n,m-1},\hat a_{n1})=d(u_4,u_3)-1$.
As $d(b_m,u_3)=d(u_4,u_3)$, we know $\eta$ has length $d(u_4,u_3)-1$. Let $\hat r$ be the vertex in $\eta$ that is adjacent to $\hat b$. Then Lemma~\ref{lem:mn} and Theorem~\ref{thm:replace} imply that $\hat r\le \hat a_{n,m-1}$ (in particular these two vertices are adjacent). 

Consider the generalized 4-cycle $x_1ba_{n,m-1}x_4$ around $b_m$. Lemma~\ref{lem:generalized 4cycle2} implies that there is $p\in (\lk^0(b_m,X),<_{b_m})$ with $b\le_{b_m}p\le_{b_m}a_{n,m-1}$ and $p\sim\{x_1,x_4\}$. Let $\hat b\hat p\hat a_{n,m-1}$ be the admissible lift of $bpa_{n,m-1}$ starting at $\hat b$. Then $\hat b\le\hat p\le \hat a_{n,m-1}$. By Lemma~\ref{lem:smaller dist}, $d(u_3,p)\le d(\hat u_3,\hat p)\le d(\hat u_3,\hat a_{n,m-1})<d(u_3,u_4)-1$. Thus replacing $u_4$ by $p$ decreases \eqref{eq:ds} -- it suffices to show $d(u_1,p)\le d(u_1,u_4)$.
Note that $d(u_1,p)\le d(u_1,a_{2,m-1})+d(a_{2,m-1},b)+d(b,p)=2+d(a_{2,m-1},b)\le 2+d(\hat a_{2,m-1},\hat b)$. However, $d(\hat a_{2,m-1},\hat b)=d(\hat a_{2,m-1},\hat b_m)-1\le d(\hat a_{2,m-1},\hat a_{n-1,m})-1\le d(u_1,u_4)-2$. Thus $d(u_1,p)\le d(u_1,u_4)$.
\end{proof}

\subsection{Case 1}
\label{subsec:case 1}
The goal of this section is to prove that under the following assumptions we can replace $(u_1,u_2,u_3,u_4)$ by another element in $\Xi$ that decreases \eqref{eq:ds}:
\begin{enumerate}
	\item $x_1,x_3$ are of type $\hat s$, and $x_2,x_4$ are of type $\hat t$;
	\item  $u_i\neq u_{i+1}$ for any $i\in \mathbb Z/4\mathbb Z$.
\end{enumerate}
This will prove Lemma~\ref{lem:key} in Case 1.


Let $A^c_{ij}$ be the induced subgraph of $\Lambda$ spanned by vertices in $\Lambda\setminus A_{ij}$. For $T,T_0\in \mathcal C_P$ that are comparable, let $\mathcal C'_{P,T}$ be as in the beginning of Section~\ref{subsec:graph theoretic}, and let $T'_0$ be the element in $\mathcal C'_{P,T}$ corresponding to $T_0$. For subsets $A,B,C$ of a graph, we write $A\mid BC$ if $A$ does not separate $B$ from $C$, and $B\mid A\mid C$ if $A$ separates $B$ from $C$.





\begin{definition}
	\label{def:uls+-}
Given $T\in \mathcal C_P$ with $T\neq \{\uls\}$, let $\hat\Lambda_T$ be as in the beginning of Section~\ref{subsec:graph theoretic}. Let $e$ be the edge in $P$ containing $\uls$ and the vertex of $P$ that is $<\uls$ with respect to the order on $P$ discussed in the beginning of Section~\ref{subsec:P}. 
We define $\uls^-_T$ to be the connected components of $\hat\Lambda_T\setminus\{\uls\}$ that contains $e\setminus\{\uls\}$, and $\uls^+_T$ to be the union of other connected components of $\hat\Lambda_T\setminus\{\uls\}$.
\end{definition}

\begin{lem}
	\label{lem:rolling}
Take vertex $r\in \Lambda$ with $r\in \uls^-_T$. Let $T_0\in \mathcal C_{P,T}$ with $T_0<\{\uls\}$ such that $r$ and $\uls$ are connected by a path $\eta\subset \hat\Lambda_T\setminus T_0$. Then $r\in \uls^-_{T_0}$.
\end{lem}

\begin{proof}
As $r\in \uls^-_T$, we can assume $\eta$ starts with the edge $e$ and never uses $e$ again. As $\eta\cap T_0=\emptyset$, we know $\eta\in \hat\Lambda_{T_0}$. As $\eta$ only uses $e$ once, $r\in \uls^-_{T_0}$.
\end{proof}

\begin{lem}
	\label{lem:+}
Let $T\in \mathcal C_P$ with $T\neq\{\uls\}$. Suppose there is an edge path $\eta\subset T^c$ from a vertex $r\in \Lambda$ to $\uls$ without using the edge $e$. Then $r\in \uls^+_T$.
\end{lem}

\begin{proof}
If $r\in \uls^-_T$, then for any interior point $x\in e$, $r$ and $\uls$ are in different connected components of $\hat\Lambda_T\setminus\{x\}$. Thus any edge path from $r$ to $\uls$ must use the edge $e$. Now the lemma follows.
\end{proof}

\begin{lem}
	\label{lem:am}
We have $t\notin A_{ni}$ for $1\le i\le m$, $\{\uls\}\neq A_{nm}$ and $t\in \uls^+_{A_{nm}}$ in $A^c_{nm}$. 
\end{lem}

\begin{proof} 
The first statement follows from Lemma~\ref{lem:nonadj1}.
As $\hat b_m\sim \{x_1,x_4\}$, Lemma~\ref{lem:nonadj1} and \ref{lem:nonadj2} implies that $\hat b_m\neq \hat a_{n-1,m}$ and $\hat b_m\neq \hat a_{n,m-1}$. Hence $B_m, A_{n,m-1},A_{n-1,m},A_{nm}$ are mutually distinct. Note that $(B_m\cup A_{nm})\mid sA_{n,m-1}$ in $\Lambda$, otherwise as $x_1\sim\{a_{nm},b_m\}$ and $a_{n,m-1}\sim\{a_{nm},b_m\}$, we deduce $x_1\sim a_{n,m-1}$ by Lemma~\ref{lem:transitive}, contradicting Lemma~\ref{lem:nonadj2}. Similarly, we deduce from Lemma~\ref{lem:transitive} and Lemma~\ref{lem:nonadj1} that $(B_m\cup A_{nm})\mid tA_{n-1,m}$ in $\Lambda$.
Note that $\{\uls\}\neq A_{nm}$ and $\{\uls\}\neq B_m$, otherwise Lemma~\ref{lem:separateuls} and Lemma~\ref{lem:transitive} imply that $x_1\sim a_{n,m-1}$, contradiction.
Lemma~\ref{lem:disjoint} implies that $(B_m\cup A_{nm})\mid \uls A_{n,m-1}$ in $\Lambda$. Hence $B_m\mid \uls A_{n,m-1}$ in $A^c_{nm}$. As $A'_{n,m-1}<B'_m$ in $\mathcal C_{P,A_{nm}}$, Lemma~\ref{lem:admissible example} and Definition~\ref{def:admissible} imply that $\{\uls\}<B'_m$ in $\mathcal C_{P,A_{nm}}$.
Then $\{\uls\}<B'_m<A'_{n-1,m}$ in $\mathcal C_{P,A_{nm}}$. Hence  $\uls\mid B'_m\mid A'_{n-1,m}$ in $A^c_{nm}$ by Lemma~\ref{lem:admissible example}.
As $B'_m\mid tA'_{n-1,m}$ in $A^c_{nm}$, we have $\{\uls\}\mid tA'_{n-1,m}$ in $A^c_{nm}$. Thus there is an edge path from $t$ to $c\in A'_{n-1,m}$ avoiding $\{\uls\}\cup A_{nm}$. As $A'_{n-1,m}\in \uls^+_{A_{nm}}$, the lemma follows.
\end{proof}

\begin{lem}
	\label{lem:inductiont1}
For $2\le i\le m-1$, suppose $A_{ni}\neq\{\uls\}$ and $t\in \uls^-_{A_{ni}}$ in $A^c_{ni}$. Then
	\begin{enumerate}
		\item $A'_{n,i-1}< \{\uls\}$ in $\mathcal C'_{P,A_{ni}}$ (in particular $A_{n,i-1}\neq\{
		\uls\}$).
		\item There is an edge path from $t$ to $\uls$ in $A^c_{ni}$ avoiding $A'_{n,i-1}$.
		\item We have $t\in \uls^-_{A_{n,i-1}}$.
	\end{enumerate}
\end{lem}

\begin{proof}
	For (1), we argue by contradiction and assume $A'_{n,i-1}\ge \{\uls\}$ in $\mathcal C'_{P,A_{ni}}$. Then $\{
	\uls\}\le A'_{n,i-1}\le A'_{n-1,i}$. Hence $t\in \uls^-_{A_{ni}}$ implies that $\uls$ separates $t$ from $A'_{n-1,i}$ in $A^c_{ni}$. By Lemma~\ref{lem:admissible example}, $\uls\mid A'_{n,i-1}\mid A'_{n-1,i}$ in $A^c_{ni}$ if $A'_{n,i-1}\neq\{\uls\}$. Thus $t\mid A'_{n,i-1}\mid A'_{n-1,i}$ in $A^c_{ni}$. On the other hand, Lemma~\ref{lem:nonadj1} and Lemma~\ref{lem:transitive} imply that $A'_{n,i-1}\mid tA'_{n-1,i}$ in $A^c_{ni}$, contradiction. Thus (1) follows. 
	For (2), by $A'_{n,i-1}\mid tA'_{n-1,i}$ in $A^c_{ni}$ there is an edge path $\eta'$ from $t$ to a point $c\in A'_{n-1,i}$ avoiding $A_{ni}\cup A_{n,i-1}$. When $A_{ni}$ is of type II, by $\uls\in P$, $A'_{n,i-1}\le A'_{n-1,i}$ in $\mathcal C'_{P,A_{ni}}$, Lemma~\ref{lem:comparable1} and \eqref{eq:decomp}, we deduce that there is an edge path $\eta''$ from $c$ to $\uls$ avoiding $A'_{n,i-1}$ in $A^c_{ni}$. Thus (2) follows.
	(3) follows from (1), (2) and Lemma~\ref{lem:rolling}.
\end{proof}

\begin{lem}
	\label{lem:inductiont2}
Suppose $2\le i\le m$. Suppose $A_{ni}\neq \{\uls\}$ and $t\in \uls^+_{A_{ni}}$ in $A^c_{ni}$.
	\begin{enumerate}
		\item If $A'_{n,i-1}<\{\uls\}$ in $\mathcal C_{P,A_{ni}}$, then $t\in \uls^+_{A_{n,i-1}}$.
		\item If $A'_{n,i-1}>\{\uls\}$ in $\mathcal C_{P,A_{ni}}$, then $t\in \uls^-_{A_{n,i-1}}$.
	\end{enumerate}
\end{lem}

\begin{proof}
For (1), let $e$ be as in Definition~\ref{def:uls+-}. Then $e\subset \hat \Lambda_{A_{ni}}$. As $t\in \uls^+_{A_{ni}}$, we have an edge path $\eta$ in $\hat \Lambda_{A_{ni}}$ connecting $t$ and $\uls$ without using the edge $e$. As $A'_{n,i-1}<\{\uls\}$, we deduce $A'_{n,i-1}\subset \uls^-_{A_{ni}}$. Thus we can assume $\eta$ is disjoint from $A'_{n,i-1}$, hence can viewed as a path in $\hat \Lambda_{A_{n,i-1}}$. As this path does not use $e$,  $t\in \uls^+_{A_{n,i-1}}$.

For (2), by the partial cyclic order structure on $\mathcal C_P$, $A'_{n,i-1}>\{\uls\}$ in $\mathcal C_{P,A_{ni}}$ implies that $A'_{ni}<\{\uls\}$ in $\mathcal C_{P,A_{n,i-1}}$. As $\hat a_{n,i-1}<\hat a_{n-1,i}<\hat a_{n,i}$, we obtain $A'_{n-1,i}<A'_{ni}$ in $\mathcal C_{P,A_{n,i-1}}$. Lemma~\ref{lem:nonadj1} implies that $A'_{ni}\mid tA'_{n-1,i}$ in $A^c_{n,i-1}$. By Lemma~\ref{lem:admissible example}, $A'_{n-1,i}\mid A'_{n,i}\mid \{\uls\}$ in $A^c_{n,i-1}$. Thus there is an edge path $\eta$ from $t$ to a point $c\in A'_{n-1,i}$ avoiding $\uls$. As $A'_{n-1,i}<\{\uls\}$ in $\mathcal C_{P,A_{n,i-1}}$, $c\in \uls^-_{A_{n,i-1}}$. Thus $t\in \uls^-_{A_{n,i-1}}$. 
\end{proof}

Note that the proof of Lemma~\ref{lem:inductiont2} (2) does not use the assumption $t\in \uls^+_{A_{ni}}$.
The proof of the following is similar to that of Lemma~\ref{lem:inductiont2} (2).

\begin{lem}
	\label{lem:induction3}
Let $2\le i\le m-1$. Suppose $A_{in}=\{\uls\}$. Then $t\in s^-_{A_{n,i-1}}$.
\end{lem}


\begin{lem}
	\label{lem:t}
The following holds:
\begin{enumerate}
	\item If $A_{n1}\neq\{\uls\}$ and $t\in \uls^+_{A_{n1}}$, then $A_{ni}\neq \{\uls\}$ and $t\in \uls^+_{A_{ni}}$ for each $1\le i\le m$. Moreover, $A'_{n,i-1}<\{\uls\}$ in $\mathcal C_{P,A_{ni}}$ for $2\le i\le m$.
	\item If $A_{n1}=\{\uls\}$ or $t\in \uls^-_{A_{n1}}$, then we can replace $(u_1,u_2,u_3,u_4)$ by another element in $\Xi$ with smaller value of \eqref{eq:ds}.
\end{enumerate}
\end{lem}

\begin{proof}
(1) follows from Lemma~\ref{lem:am}, Lemma~\ref{lem:inductiont1}, Lemma~\ref{lem:induction3} and Lemma~\ref{lem:inductiont2}.

For (2), first we consider the case $A_{n1}=\{\uls\}$. If $s\neq\uls$, then in $\Lambda$ we know $\uls$ separates $s$ from any elements of $\mathcal C_P$ that is distinct from $\{\uls\}$ by Lemma~\ref{lem:separateuls}. Then $x_3\sim b_2$ by Lemma~\ref{lem:transitive}. We still have $x_3\sim b_2$ if $s=\uls$. So $b_2\in Y_3\cap Y_4$, and replacing $u_4$ by $b_m$ and replacing $u_3$ by $b_2$ decreases \eqref{eq:ds} (as $d(u_2,b_2)=d(u_2,u_3)$, $d(b_2,b_m)<d(u_3,u_4)$ and $d(u_1,b_m)\le d(u_1,u_4)$).

Now assume $A_{n1}\neq\{\uls\}$ and $t\in \uls^-_{A_{n1}}$. As $x_4\sim\{a_{n1},b_2\}$, we have $(B_2\cup A_{n1})\mid tA_{n-1,2}$ in $\Lambda$ (otherwise $a_{n-1,2}\sim x_4$). Hence $B'_2\mid tA_{n-1,2}$ in $A^c_{n1}$. Similarly to Lemma~\ref{lem:inductiont1} (1), we deduce from $t\in \uls^-_{A_{n1}}$ that $B'_2<\{\uls\}$ in $\mathcal C'_{P,A_{n1}}$.  

Consider $A'_{n-1,1}$ in $\mathcal C'_{P,A_{n1}}$. As $\hat a_{n-1,1}<\hat a_{n1}<\hat b_2<\hat a_{n-1,2}$ and $\hat a_{n-1,1}<\hat a_{n-1,2}$, we know $a_{n-1,1}\sim b_2$, and $b_2<_{u_3}a_{n-1,1}\in (\lk^0(u_3,X),<_{u_3})$. Thus $B'_2<A'_{n-1,1}$ in $\mathcal C'_{P,A_{n1}}$.
Assume $A'_{n-1,1}\neq\{\uls\}$, otherwise we deduce from $a_{n-1,1}\sim b_2$ that 
$x_3\sim b_2$, and we replace $u_3$ by $b_2$ to decrease \eqref{eq:ds}.

If $B'_2<A'_{n-1,1}<\{\uls\}$ in $\mathcal C'_{P,A_{n1}}$, then Lemma~\ref{lem:admissible example} implies $\uls\mid(A_{n1}\cup A_{n-1,1})\mid B_2$ in $\Lambda$. By Lemma~\ref{lem:disjoint}, $s\mid (A_{n1}\cup A_{n-1,1})\mid B_2$ in $\Lambda$. As $x_3\sim\{a_{n1},a_{n-1,1}\}$ and $b_2\sim\{a_{n1},a_{n-1,1}\}$, by Lemma~\ref{lem:transitive}, $b_2\sim x_3$, and we conclude as before.

Suppose $A'_{n-1,1}>\{\uls\}$. 
We consider the generalized 4-cycle $b_2x_4x_3a_{n-1,1}$ around $a_{n1}$. Lemma~\ref{lem:generalized 4cycle2} implies that there is a vertex $a\in (\lk^0(a_{n1},X),<_{a_{n1}})$ such that with $b_2\le_{a_{n1}} a\le_{a_{n1}} a_{n-1,1}$ and $a\sim \{x_3,x_4\}$. Suppose $a$ has type $\hat A$. Then $B'_2\le A'\le A'_{n-1,1}$ in $\mathcal C_{P,A_{n1}}$. If $A'<\{\uls\}$, then Lemma~\ref{lem:admissible example} implies that $B'_2\mid A'\mid \{\uls\}$ in $A^c_{n1}$. By Lemma~\ref{lem:disjoint}, $B'_2\mid A'\mid s$ in $A^c_{n1}$. As $x_3\sim \{a_{n1},a\}$ and $b_2\sim\{a_{n1},a\}$, by Lemma~\ref{lem:transitive}, $b_2\sim x_3$ and we conclude as before. Now suppose $\{\uls\}\le A'< A'_{n-1,1}$. As $t\in \uls^-_{A_{n1}}$, $A'_{n-1,1}\mid\uls\mid t$ in $A^c_{n1}$. By Lemma~\ref{lem:admissible example}, $\uls\mid A'\mid A'_{n-1,1}$ in $A^c_{n1}$ if $A'\neq\{\uls\}$. Thus $A'_{n-1,1}\mid A'\mid t$ in $A^c_{n1}$, hence $A_{n-1,1}\mid (A\cup A_{n1})\mid t$ in $\Lambda$. As $x_4\sim\{a,a_{n1}\}$ and $a_{n-1,1}\sim \{a,a_{n1}\}$, by Lemma~\ref{lem:transitive}, $x_4\sim a_{n-1,1}$. Thus $a_{n-1,1}\in Y_4\cap Y_3$, and replacing $u_3$ by $a_{n-1,1}$ and $u_4$ by $b_m$ decreases \eqref{eq:ds}, as $d(u_2,a_{n-1,1})<d(u_2,u_3)$ and $d(a_{n-1,1},b_m)=1+d(b_2,b_m)\le d(u_3,u_4)$. It remains to consider $A'=A'_{n-1,1}$, in which case $a=a_{n-1,1}$ and $x_4\sim a_{n-1,1}$. Then $a_{n-1,1}\in Y_4\cap Y_3$. As before, replacing $u_4$ by $b_m$ and $u_3$ by $a_{n-1,1}$ decreases \eqref{eq:ds}.
\end{proof}

In the rest of this section, we assume Lemma~\ref{lem:t} (1) holds. In particular, $A'_{n,m-1}<\{\uls\}<A'_{n-1,m}$ in $\mathcal C_{P,A_{nm}}$.

\begin{lem}
\label{lem:s}
We have $A_{im}\neq\{\uls\}$ for $1\le i\le n-1$, and $A_{i,m-1}\neq\{\uls\}$ for $2\le i\le n$. Moreover, in $\mathcal C_{P,\{\uls\}}$, $A_{1m}<A_{2m}<\cdots<A_{n-1,m}<A_{2,m-1}<A_{3,m-1}<\cdots<A_{n,m-1}$.
\end{lem}

\begin{proof}
Lemma~\ref{lem:nonadj2} and Lemma~\ref{lem:disjoint} imply that $A_{im}\neq\{\uls\}$ for $1\le i\le n-1$.
The discussion in Lemma~\ref{lem:am} implies that in $\mathcal C_{P,\{\uls\}}$, $A'_{n-1,m}<A'_{n,m-1}$. As $\hat a_{n-1,m-1}<\hat a_{n,m-1}<\hat a_{n-1,m}$, we have $A'_{n-1,m-1}<A'_{n,m-1}$ in $\mathcal C_{P,A_{n-1,m}}$. Lemma~\ref{lem:nonadj2} implies $A'_{n-2,m}\mid sA'_{n-1,m-1}$ in $A^c_{n-1,m}$. Lemma~\ref{lem:disjoint} implies that $A'_{n-2,m}\mid \uls A'_{n-1,m-1}$ in $A^c_{n-1,m}$. As $A'_{n-1,m-1}<A'_{n,m-1}<\{\uls\}$ in $\mathcal C_{P,A_{n-1,m}}$ and $A'_{n-1,m-1}<A'_{n-2,m}$ (this follows from $\hat a_{n-1,m-1}<\hat a_{n-2,m}<\hat a_{n-1,m}$), we deduce from 
Lemma~\ref{lem:admissible example} that $A'_{n-2,m}>\{\uls\}$ in $\mathcal C_{P,A_{n-1,m}}$. So $A'_{n-1,m-1}<A'_{n,m-1}<\{\uls\}<A'_{n-2,m}$ in $\mathcal C_{P,A_{n-1,m}}$. It follows that $A'_{n-2,m}<A'_{n-1,m}<A'_{n-1,m-1}<A'_{n,m-1}$ in $\mathcal C_{P,\{\uls\}}$. Repeating this argument, we obtain $A'_{i-1,m}<A'_{i,m}<A'_{i,m-1}<A'_{i+1,m-1}$ in $\mathcal C_{P,\{\uls\}}$ for $2\le i\le n-1$. Thus the lemma follows.
\end{proof}

\begin{lem}
Suppose $t\in \uls^-_{A_{1m}}$. Then we can replace $(u_1,u_2,u_3,u_4)$ by another element in $\Xi$ with smaller value of \eqref{eq:ds}.
\end{lem}

\begin{proof}
Let $\hat c_{n-1}=\hat b_m$. We define $\hat c_i$ to be the meet of $\hat c_{i+1}$ and $\hat a_{im}$ in $(\widehat X^0,\le_t)$ for $1\le i<n-1$. Note that $\hat a_{i+1,m-1}\le\hat c_i\le \hat a_{i,m}$ for $1\le i\le n-1$. Note that $\hat c_i\neq \hat a_{i,m}$ for $1\le i\le n-1$, otherwise $d(u_1,c_{n-1})<d(u_1,u_4)$ and replacing $u_4$ by $b_m=c_{n-1}$ decreases \eqref{eq:ds}.
By Lemma~\ref{lem:B-geodesics} (2), $\hat a_{im}\hat c_i\hat c_{i+1}$ gives a left $B$-geodesic from $a_{im}$ to $c_{i+1}$ for $1\le i<n-1$. 
As $c_{n-1}\in Y_1$, using Lemma~\ref{lem:convex} we deduce inductively that $c_i\in Y_1$ for $1\le i\le n-1$. Then by Lemma~\ref{lem:nonadj2}, $\hat a_{i+1,m-1}\neq \hat c_i$ for $1\le i\le n-1$. Thus $\hat a_{i+1,m-1}<\hat c_i<\hat a_{i,m}$. Suppose $c_i$ has type $\hat C_i$.

Note that $C'_1\mid sA'_{2,m-1}$ in $A^c_{1m}$, otherwise by $x_1\sim \{a_{1m},c_1\}$, $a_{2,m-1}\sim\{a_{1m},c_1\}$ and Lemma~\ref{lem:transitive}, we know $x_1\sim a_{2,m-1}$, contradicting Lemma~\ref{lem:nonadj2}. By Lemma~\ref{lem:disjoint},  $C'_1\mid \uls A'_{2,m-1}$ in $A^c_{1m}$. 
By Lemma~\ref{lem:s}, $A'_{2,m-1}<\{\uls\}$ in $\mathcal C_{P,A_{1m}}$. Moreover, $\hat a_{2,m-1}<\hat c_1<\hat a_{1m}$ implies that $A'_{2,m-1}<C'_1$ in $\mathcal C_{P,A_{1m}}$. Thus $C'_1>\{\uls\}$ in $A^c_{1m}$. As $a_{1,m-1}<\hat a_{2,m-1}<\hat a_{1m}$, thus $A'_{1,m-1}\le A'_{2,m-1}$ in $A^c_{1m}$. Thus $A'_{1,m-1}\le A'_{2,m-1}<\{\uls\}<C'_1$ in $A^c_{1m}$. Now we consider the generalized 4-cycle $a_{1,m-1}c_1x_1x_2$ around $a_{1m}$. Lemma~\ref{lem:generalized 4cycle2} implies there exists $a\in (\lk^0(a_{1m},X),<_{a_{1m}})$ with $a_{1,m-1}\le_{a_{1m}} a\le_{a_{1m}} c_1$ such that $a\sim \{x_1,x_2\}$. Suppose $a$ is of type $\hat A$. Then $A'_{1,m-1}\le A'\le C'_1$ in $\mathcal C_{P,A_{1m}}$.

We assume $A'$ is not one of $\{A'_{1,m-1},C'_1,\{\uls\}\}$, otherwise either $x_1\sim a_{1,m-1}$ or $x_2\sim c_1$ (we use Lemma~\ref{lem:separateuls} when $A'=\{\uls\}$). We can decrease \eqref{eq:ds} by replacing $u_1$ by $c_1$ and $u_4$ by $b_m$ when $x_2\sim c_1$; or by replacing $u_1$ by $a_{1,m-1}$ and $u_4$ by $b_m$ when $x_1\sim a_{1,m-1}$. If $A'_{1,m-1}<A'<\{\uls\}$ in $\mathcal C_{P,A_{1m}}$, then $\uls\mid A'\mid A'_{1,m-1}$ in $A^c_{1m}$ by Lemma~\ref{lem:admissible example}. Hence $s\mid A'\mid A'_{1,m-1}$ in $A^c_{1m}$ by Lemma~\ref{lem:disjoint}. By $x_1\sim \{a_{1m},a\}$, $a_{1,m-1}\sim\{a_{1m},a\}$ and Lemma~\ref{lem:transitive}, we obtain $x_1\sim a_{1,m-1}$ and finish as before. It remains to consider $\{\uls\}<A'<C'_1$. Using $t\in \uls^-_{A_{1m}}$, we deduce $x_1\sim c_1$ from $c_1\sim\{a,a_{1m}\}$ and $x_1\sim\{a,a_{1m}\}$ in the same way as the last paragraph of the proof of Lemma~\ref{lem:t}, and conclude as before.
\end{proof}

Note that $t\notin A_{1m}$, otherwise $x_2\sim c_1$ and we decrease \eqref{eq:ds} as before. So it remains to consider $t\in \uls^+_{A_{1m}}$.
\begin{lem}
	\label{lem:chain}
	Suppose $t\in \uls^+_{A_{1m}}$. Then we can replace $(u_1,u_2,u_3,u_4)$ by another element in $\Xi$ with smaller value of \eqref{eq:ds}. 
\end{lem}

\begin{proof}
By Lemma~\ref{lem:t} (1), $A'_{n,i-1}<\{\uls\}$ in $\mathcal C_{P,A_{ni}}$. Thus $A'_{n,i}<A'_{n,i-1}$ in $\mathcal C_{P,\{\uls\}}$ for $1\le i\le m$. Hence $A'_{n,m-1}<A'_{n,m-2}<\dots<A'_{n1}$ in $\mathcal C_{P,\{\uls\}}$. This together with Lemma~\ref{lem:s} implies that in $\mathcal C_{P,\{\uls\}}$ we have:
\begin{align}
	\label{eq:chain}
 A'_{1m}<A'_{2m}<\dots<A'_{n-1,m}&<A'_{2,m-1}<A'_{3,m-1}<\cdots\\
 &
 <A'_{n,m-1}<A'_{n,m-2}<\cdots<A'_{n1}. \nonumber
\end{align}

Let $D$ be the disk in Figure~\ref{fig:systolic}. By construction, edges of $D$ that are in the same direction as $\hat a_{n-1,m}\hat a_{nm}$ or $\hat a_{n,m-1}\hat a_{nm}$ are mapped to non-degenerated edges in $\widehat X$.

We claim $A'_{1m}<A'_{12}<A'_{11}<A'_{21}<A'_{n1}$ in $\mathcal C_{P,\{\uls\}}$. To prove this, we start with an observation: if there is a triangle $\hat a_{ij}\le \hat a_{i+1,j}\le \hat a_{i,j+1}$ in $D$ such that
\begin{enumerate}
	\item the three vertices in $\widehat X$ are mutually distinct,
	\item  $A'_{i,j+1}<A'_{i+1,j}$ in $\mathcal C_{P,\{\uls\}}$;
\end{enumerate}
Then $A'_{i,j+1}<A'_{ij}<A'_{i+1,j}$ in $\mathcal C_{P,\{\uls\}}$. Indeed, this follows from $A'_{ij}<A'_{i+1,j}<\{\uls\}$ in $\mathcal C_{P,A_{i,j+1}}$. Now look at the top strip of $D$. Then $\hat a_{1,m-1}\neq \{\hat a_{1m},\hat a_{2,m-1}\}$ in $\widehat X$. Moreover, $\hat a_{1m}\neq\hat a_{2,m-1}$ as $A'_{1m}<A'_{2,m-1}$ in $\mathcal C_{P,\{\uls\}}$. Thus
So the above observation implies that $A'_{1m}<A'_{1,m-1}<A'_{2,m-1}$ in $\mathcal C_{P,\{\uls\}}$. Thus
$$
A'_{1m}<A'_{1,m-1}<A'_{2,m-1}<\cdots<A'_{n,m-1}.
$$
Next we look at the second top strip of $D$, from right to left, first at the triangle with vertices $\hat a_{n-1,m-2},\hat a_{n,m-2},\hat a_{n-1,m-1}$. Then $\hat a_{n-1,m-1}\neq \hat a_{n,m-2}$ in $\widehat X$ by \eqref{eq:chain}, and $\hat a_{n-1,m-2}\neq \hat a_{n,m-2}$ and $\hat a_{n-1,m-2}\neq\hat a_{n-1,m-1}$ in $\widehat X$ by the construction of $D$. Thus the above observation implies that $A'_{n-1,m-1}<A'_{n-1,m-2}<A'_{n,m-2}$ in $\mathcal C_{P,\{\uls\}}$. This together with \eqref{eq:chain} implies $A'_{n-2,m-1}<A'_{n-1,m-1}<A'_{n-1,m-2}$ in $\mathcal C_{P,\{\uls\}}$. In particular, $\hat a_{n-2,m-1}\neq\hat a_{n-1,m-2}$ in $\widehat X$. We argue as before for the triangle with vertices $\hat a_{n-2,m-2},\hat a_{n-1,m-2},\hat a_{n-2,m-1}$, and deduce $A'_{n-2,m-1}<A'_{n-2,m-2}<A'_{n-1,m-2}$ in $\mathcal C_{P,\{\uls\}}$. We process triangles in this strip from right to left, and obtain 
$$A'_{1,m-1}<A'_{1,m-2}<A'_{2,m-2}<\cdots<A'_{n-1,m-2}<A'_{n,m-2}.$$ 
We successively look at lower strips of $D$, and repeat such argument to deduce $A'_{12}<A'_{11}<A'_{21}<\cdots<A'_{n-1,1}<A'_{n1}$ in $\mathcal C_{P,\{\uls\}}$. Thus the claim follows.

As $t\in \uls^+_{A_{1m}}$, by the previous claim, $t\in \uls^+_{A_{11}}$. Considering the generalized 4-cycle $x_2a_{12}a_{21}x_3$ around $a_{11}$. By Lemma~\ref{lem:generalized 4cycle2}, there is $a\in (\lk^0(a_{11},X),<_{a_{11}})$ such that $a_{21}<_{a_{11}} a<_{a_{11}} a_{12}$ and $a\sim\{x_2,x_3\}$. Suppose $a$ has type $\hat A$. Then $A'_{21}\le A'\le A'_{12}$ in $\mathcal C_{P,A_{11}}$. By the claim, $A'_{21}<\{\uls\}<A'_{12}$ in $\mathcal C_{P,A_{11}}$.
We will show either $x_2\sim a_{21}$ or $x_3\sim a_{12}$, in which case replacing $u_2$ by $a_{21}$ or $a_{12}$ decreases \eqref{eq:ds}. This is clear when $A'\in \{A'_{21},A'_{12},\{\uls\}\}$. If $A'_{21}<A'<\{\uls\}$, then $A'_{21}\mid A'\mid t$ in $A^c_{11}$ as $t\in \uls^+_{A_{11}}$, hence we deduce $x_2\sim a_{21}$ from $x_2\sim\{a_{11},a\}$ and $a_{21}\sim\{a_{11},a\}$. If $\{\uls\}<A'<A'_{12}$, then $\uls\mid A'\mid A'_{12}$, hence $s\mid A'\mid A'_{12}$ in $A^c_{11}$ by Lemma~\ref{lem:disjoint}. It follows that $x_3\sim a_{12}$. 
\end{proof}

\subsection{Case 2}
\label{subsec:case 2}
The goal of this section is to prove that under the following assumptions we can replace $(u_1,u_2,u_3,u_4)$ by another element in $\Xi$ that decreases \eqref{eq:ds}:
\begin{enumerate}
	\item $x_1,x_3$ are of type $\hat t$, and $x_2,x_4$ are of type $\hat s$;
	\item $u_i\neq u_{i+1}$ for any $i\in \mathbb Z/4\mathbb Z$.
\end{enumerate}
This will prove Lemma~\ref{lem:key} in Case 2.

\begin{lem}
	\label{lem:s2}
In $\mathcal C_{P,\{\uls\}}$, we have $A'_{nm}<A'_{n,m-1}<\cdots<A'_{n1}<B'_2$.
\end{lem}

The proof of this lemma is similar to that of Lemma~\ref{lem:s}, using Lemma~\ref{lem:nonadj1}.

\begin{lem}
Suppose $t\in \uls^+_{A_{n1}}$. Then we can replace $(u_1,u_2,u_3,u_4)$ by another element in $\Xi$ with smaller value of \eqref{eq:ds}. 
\end{lem}

The proof of this lemma is similar to that of Lemma~\ref{lem:t} (2), using $A'_{n1}<B'_2$ in $\mathcal C_{P,\{\uls\}}$ (Lemma~\ref{lem:s2}).

Note that $t\notin A_{n1}$, otherwise $x_3\sim b_2$ and we can replace as in Lemma~\ref{lem:t} (2) to decrease \eqref{eq:ds}.
In the rest of this section, we assume $t\in \uls^-_{A_{n1}}$. 

\begin{lem}
	\label{lem:t2}
Suppose $t\in \uls^+_{A_{im}}$ for some $i$ with $2\le i\le n$. Then $\{\uls\}<A'_{i-1,m}$ in $\mathcal C_{P,A_{im}}$ and $t\in \uls^+_{A_{i-1,m}}$.
\end{lem}

\begin{proof}
First we show $\{\uls\}<A'_{i-1,m}$ in $\mathcal C_{P,A_{im}}$. Suppose $A'_{i-1,m}\le\{\uls\}$. Then $A'_{i,m-1}<A'_{i-1,m}\le \{\uls\}$ in $\mathcal C_{P,A_{im}}$. As $t\in \uls^+_{A_{im}}$, $t\mid A'_{i-1,m}\mid A'_{i,m-1}$ in $A^c_{im}$. Then $t\mid (A_{i-1,m}\cup A_{im})\mid A_{i,m-1}$, and by Lemma~\ref{lem:transitive} we obtain $t\sim a_{i,m-1}$, contradiction Lemma~\ref{lem:nonadj2}. 

Lemma~\ref{lem:nonadj2} implies that $A'_{i-1,m}\mid tA'_{i,m-1}$ in $A^c_{im}$. Then there is an edge path $\eta\subset A^c_{im}$ from $t$ to $x\in A'_{i,m-1}$ avoiding $A'_{i-1,m}$. Let $e$ be the edge in Definition~\ref{def:uls+-}. If $A'_{i,m-1}\le \{\uls\}$ in  $\mathcal C_{P,A_{im}}$, as $t\in \uls^+_{A_{im}}$, then $\eta$ has a subpath from $t$ to $\uls$ without using $e$. If $A'_{i,m-1}>\{\uls\}$ in  $\mathcal C_{P,A_{im}}$, then we can assume $\eta$ does not use $e$, and there is an edge path $\eta'\subset A^c_{im}$ from $x$ to $\uls$ without using $e$ and avoiding $A'_{i-1,m}$ (the existence of $\eta'$ follows from $A'_{i-1,m}>A'_{i,m-1}>\{\uls\}$). In either cases, we have an path outside $A_{im}\cup A_{i-1,m}$ from $t$ to $\uls$ without using $e$. Thus $t\in \uls^+_{A_{i-1,m}}$ by Lemma~\ref{lem:+}.
\end{proof}

\begin{lem}
	\label{lem:t2'}
Let $2\le i\le n$. Suppose $A'_{im}<A'_{i-1,m}$ in $\mathcal C_{P,\{\uls\}}$. Then $t\in \uls^+_{A_{i-1,m}}$.
\end{lem}

\begin{proof}
Our assumption implies that $A'_{i-1,m}<\{\uls\}$ in $\mathcal C_{P,A_{im}}$. Thus $A'_{i,m-1}<A'_{i-1,m}<\{\uls\}$. Lemma~\ref{lem:nonadj2} implies that $A'_{i-1,m}\mid tA'_{i,m-1}$ in $A^c_{im}$. So there is an edge path $\eta\subset A^c_{im}$ from $t$ to $x\in A'_{i,m-1}$ avoiding $A'_{i-1,m}$. As $A'_{i,m-1}\mid A'_{i-1,m}\mid \uls$ in $A^c_{im}$ by Lemma~\ref{lem:admissible example}, $\uls\notin \eta$. 
Let $e$ by the edge in Definition~\ref{def:uls+-}. Then $\eta$ does not use $e$. 

We view $\eta$ as an edge path in $A^c_{i-1,m}$. As $\{\uls\}<A'_{i,m-1}$ in $\mathcal C_{P,A_{i-1,m}}$, there is an edge path $\eta'\subset A^c_{i-1,m}$ from $x$ to $\uls$ without using $e$. The concatenation of $\eta$ and $\eta'$ gives a path in $A^c_{i-1,m}$ from $t$ to $\uls$ without using $e$. Hence $t\in \uls^+_{A_{i-1,m}}$ by Lemma~\ref{lem:+}.
\end{proof}

\begin{lem}
Suppose $t\in \uls^+_{A_{1m}}$. Then we can replace $(u_1,u_2,u_3,u_4)$ by another element in $\Xi$ with smaller value of \eqref{eq:ds}. 
\end{lem}

\begin{proof}
Similar to Lemma~\ref{lem:t2}, we have $\{\uls\}<C'_1$ in $\mathcal C_{P,A_{1m}}$. We claim $x_1\sim a_{1,m-1}$ or $x_2\sim c_1$. The lemma follows from this claim as in the former case we replace $u_1$ by $a_{1,m-1}$ and $u_4$ by $b_m$ to decrease \eqref{eq:ds}, in the latter case we replace $u_1$ by $c_1$ and $u_4$ by $b_m$ to decrease \eqref{eq:ds}. It suffices to prove the claim.
Note that $A'_{1,m-1}<C'_1$ in $\mathcal C_{P,A_{1m}}$. If $\{\uls\}\le A'_{1,m-1}<C'_1$, then Lemma~\ref{lem:admissible example} and Lemma~\ref{lem:disjoint} imply that $C'_1\mid A'_{1,m-1}\mid s$ in $A^c_{1m}$. As $x_2\sim\{a_{1m},a_{1,m-1}\}$ and $c_1\sim\{a_{1m},a_{1,m-1}\}$, Lemma~\ref{lem:transitive} implies that $x_2\sim c_1$. Now we assume $A'_{1,m-1}<\{\uls\}$ in $\mathcal C_{P,A_{1m}}$. Consider the generalized 4-cycle $x_1c_1a_{1,m-1}x_2$ around $a_{1m}$. By Lemma~\ref{lem:generalized 4cycle2}, there is $a\in (\lk^0(a_{1m},X),<_{a_{1m}})$ with $a_{1,m-1}\le_{a_{1m}} a\le_{a_{1m}} c_1$ and $a\sim\{x_1,x_2\}$. Suppose $a$ has type $\hat A$. If $A\in \{A'_{1,m-1},\{\uls\},C'_1\}$, then the claim follows. If $\{\uls\}<A'<C'_1$ in $\mathcal C_{P,A_{1m}}$, then $C_1\mid(A_{1m}\cup A)\mid \uls$ by Lemma~\ref{lem:admissible example} and $C_1\mid(A_{1m}\cup A)\mid s$ by Lemma~\ref{lem:disjoint}. As $x_2\sim\{a_{1m},a\}$ and $c_1\sim\{a_{1m},a\}$, by Lemma~\ref{lem:transitive}, $x_1\sim c_1$. It remains to consider $A'_{1,m-1}<A'<\{\uls\}$. As $t\in \uls^+_{A_{1m}}$, $A'_{1,m-1}\mid A'\mid t$ in $A^c_{1m}$. Thus $t\mid (A_{1m}\cup A)\mid A_{1,m-1}$. As $x_1\sim \{a_{1m},a\}$ and $a_{1,m-1}\sim\{a_{1m},a\}$, we have $x_1\sim a_{1,m-1}$ and the claim follows.
\end{proof}

Note that $t\notin A_{1m}$, otherwise $x_1\sim a_{1,m-1}$. So it remains to consider  $t\in \uls^-_{A_{1m}}$.

\begin{lem}
Suppose $t\in \uls^-_{A_{1m}}$. Then we can replace $(u_1,u_2,u_3,u_4)$ by another element in $\Xi$ with smaller value of \eqref{eq:ds}. 
\end{lem}

\begin{proof}
By Lemma~\ref{lem:t2} and Lemma~\ref{lem:t2'}, $A'_{im}>A'_{i-1,m}$ in $\mathcal C_{P,\{s\}}$ for $2\le i\le n-1$. This and Lemma~\ref{lem:s2} imply that $$A'_{1m}<A'_{2m}<\cdots A'_{n-1,m}<A'_{n,m-1}<A'_{n,m-2}<\cdots<A'_{n1}.$$
Similar to the proof of Lemma~\ref{lem:chain}, we deduce that $A'_{1m}<A'_{12}<A'_{11}<A'_{21}<A'_{n1}$ in $\mathcal C_{P,\{\uls\}}$. As $t\in \uls^-_{A_{n1}}$ (we assumed so before Lemma~\ref{lem:t2}) and $\{\uls\}<A'_{11}$ in $\mathcal C_{P,A_{n1}}$, we have $t\in \uls^-_{A_{11}}$. 


We now show either $x_2\sim a_{21}$ or $x_3\sim a_{12}$. Considering the generalized 4-cycle $x_2a_{12}a_{21}x_3$ around $a_{11}$. Let $a$ and $A$ be the same as in the last paragraph of the proof of Lemma~\ref{lem:chain}. If $\{\uls\}<A'<A'_{12}$ in $\mathcal C_{P,A_{11}}$, then $A'_{12}\mid A'\mid t$ in $A^c_{11}$ as $t\in \uls^-_{A_{11}}$, hence $a_{12}\sim x_3$. If $A'_{21}<A'<\{\uls\}$, then $\uls\mid A'\mid A'_{21}$ in $A^c_{11}$, hence $s\mid A'\mid A'_{21}$ by Lemma~\ref{lem:disjoint} and $x_2\sim a_{21}$. The rest of the proof is identical to Lemma~\ref{lem:chain}.
\end{proof}

\section{Contractibility of Artin complexes}
\label{sec:contractibility}
\subsection{Cycle reduction}
\begin{thm}
	\label{thm:cycle reduction}
Let $\mathcal C$ be a class of Coxeter diagrams which is closed under taking induced subdiagrams. Suppose that we have $\mathcal C_1\subset\mathcal C$ such that each element in $\mathcal C\setminus \mathcal C_1$ is not a forest. 
\begin{enumerate}
	\item Suppose for each element $\Lambda$ in $\mathcal C_1$, $\Delta_\Lambda$ satisfies the labeled 4-cycle condition. Then $\Delta_\Lambda$ satisfies the labeled 4-cycle condition for any $\Lambda\in \mathcal C$.
	\item Suppose in addition that for each element $\Lambda$ in $\mathcal C_1$ which is not spherical, $\Delta_\Lambda$ is contractible. Then $\Delta_\Lambda$ is contractible for each non-spherical $\Lambda\in \mathcal C$. In particular, $A_\Lambda$ satisfies the $K(\pi,1)$-conjecture for each $\Lambda\in \mathcal C$.
\end{enumerate}
\end{thm}

\begin{proof}
For (1), we induct on number of vertices in $\Lambda\in \mathcal C$. It suffices to consider the case $\Lambda$ is connected, as otherwise by induction each connected component of $\Lambda$ satisfies the labeled 4-cycle condition. Also we can assume $\Lambda\notin \mathcal C_1$, then $\Lambda$ is not a tree. Then Proposition~\ref{prop:prop} implies that $\Delta_\Lambda$ satisfies the labeled 4-cycle condition.

For (2), we induct on number of vertices in $\Lambda\in \mathcal C$. It suffices to consider the case $\Lambda$ is connected, as otherwise there is a non-spherical connected component $\Lambda'\subset \Lambda$ with $\Delta_{\Lambda'}$ contractible. As $\Delta_{\Lambda'}$ is a join factor of $\Delta_\Lambda$, $\Delta_\Lambda$ is contractible. We can also assume $\Lambda\notin \mathcal C_1$, so $\Lambda$ is not a tree. We choose a path $P\subset \Lambda$ and define $\Lambda_P$ as in the beginning of Section~\ref{sec:minimal cut complex}.

\begin{claim}
The complex $\Delta_{\Lambda}$ deformation retracts onto $\Delta_{\Lambda,\Lambda'}$ for any connected induced non-spherical subdiagrams $\Lambda'$ of $\Lambda$.
\end{claim}

\begin{proof}
This follows from Lemma~\ref{lem:dr} and our induction assumption.	
\end{proof}

Let $\mathcal C_P,\Sigma_P,\Delta_1$ be as in Section~\ref{subsec:sc}.
We define a subdivision $\Delta_0$ of $\Delta_{\Lambda}$ as follows. A vertex $v$ of $\Delta_0$ correspondences a left coset of form $gA_{\hat T}$ with $T$ being either a single vertex of $P$ (in which case $v$ is called a \emph{special vertex}), or $T$ being a subset of $\Lambda_P\setminus\{a,b\}$ (in which case $v$ is a \emph{non-special vertex}). A special vertex is adjacent to another vertex if the associated two cosets have non-empty intersection. Two non-special vertices are adjacent, if the coset associated to one vertex is contained in the coset associated to the other vertex. Then $\Delta_0$ is the flag subcomplex on its 1-skeleton. Note that $\Delta_1$ embeds as a subcomplex of $\Delta_0$.

For a set of vertex $V\subset \Lambda_P$, let $\Delta^V_0$ to be the full subcomplex of $\Delta_0$ spanned by all vertices of type $\hat T$ with $T\cap (V\cup P)\neq\emptyset$. Note that given two sets of vertices $V_1,V_2$ of $\Lambda_0\cap \Lambda$, $\Delta^U_0=\Delta^{V_1}_0\cup\Delta^{V_2}_0$ where $U=V_1\cup V_2$. Clearly, $\Delta^{V_1}_0\cup\Delta^{V_2}_0\subset\Delta^U_0$. Given a simplex $\sigma$ of $\Delta^U_0$, we write a join decomposition $\sigma=\sigma_1*\sigma_2$ where $\sigma_1$ is spanned by special vertices and $\sigma_2$ is spanned by non-special vertices. Suppose $\sigma_2$ has vertices $\{v_i\}_{i=1}^k$ such that $v_i$ has type $\hat T_i$. Suppose $T_1\subset\cdots\subset T_k$. As $T_1\cap U\neq \emptyset$, we know either $T_1\cap V_1\neq\emptyset$ or $T_2\cap V_2\neq\emptyset$. In the former case, $T_i\cap V_1\neq\emptyset$ for $1\le i\le k$, hence $\sigma\subset \Delta^{V_1}_0$. In the latter case $\sigma\subset\Delta^{V_2}_0$.

Consider the set of all embedded induced edge paths in $\Lambda_P$ joining $a$ to $b$. Let $\Theta$ denote the collection of vertex sets consisting of the interior vertices of such paths, and let $\Theta'$ be the family of (non-empty) unions of elements of $\Theta$. If $V\in \Theta'$, then the induced subdiagram of $\Lambda$ spanned by $V\cup P$ is connected, hence $\Delta_{\Lambda}$ deformation retracts onto $\Delta_{\Lambda,V\cup P}$ by Claim 1. Thus $H_i(\Delta_0,\Delta_{\Lambda,V\cup P})=0$ for all $i$. Note that $\Delta^V_0$ is the union of the collection of simplices of $\Delta_0$ that have non-empty intersection with  $\Delta_{\Lambda,V\cup P}$. Hence there is a deformation retraction $r:\Delta^V_0\to \Delta_{\Lambda,V\cup P}$ sending vertices of type $\hat T$ to vertices of type $\hat T'$ with $T'=T\cap(V\cup P)$. Hence $H_i(\Delta^V_0,\Delta_{\Lambda,V\cup P})=0$ for $i\ge 0$. Then  $H_i(\Delta_0,\Delta^V_0)=0$ for $i\ge 0$.

Next we prove $\Delta_1=\cap_{V\in\Theta'}\Delta^V_0$. It suffices to show $\Delta_1=\cap_{V\in\Theta}\Delta^V_0$. We only need to show these two subcomplexes have the same set of non-special vertices.
Given a non-special vertex $v\in \cap_{V\in\Theta}\Delta^V_0$ of type $\hat T$, then $T\cap V\neq\emptyset$ for any $V\in \Theta$. Hence $T$ has non-empty intersection with any
embedded and induced edge path in $\Lambda_P$ from $a$ to $b$. Hence $T$ separates $a$ from $b$ in $\Lambda_P$.  It follows that $T$ contains an element in $\mc_{\Lambda_P}(\{a\},\{b\})$ and $v\in \Delta_1$. Conversely, if $v\in \Delta_1$ is a non-special vertex of type $\hat T$, then $T$ contains an element from $\mc_{\Lambda_P}(\{a\},\{b\})$. Hence $T\cap V\neq\emptyset$ for any $V\in \Theta$.

\begin{claim}
For	any $V_1,V_2,\ldots V_k\in \Theta'$, $H_i(\Delta_0,\cap_{i=1}^k\Delta^{V_i}_0)=0$ for $i\ge 0$.
\end{claim} 

\begin{proof}
 We induct on $k$ and the case $k=1$ is already proved. Suppose the claim holds true for $k-1$ sets. We consider the Mayer–Vietoris  sequence:
\begin{align*}
\cdots &\to H_n(\Delta_0,(\cap_{i=1}^{k-1}\Delta^{V_i}_0)\cap\Delta^{V_k}_0)\to H_n(\Delta_0,\cap_{i=1}^{k-1}\Delta^{V_i}_0)\oplus H_n(\Delta_0,\Delta^{V_k}_0)\\
&\to H_n(\Delta_0,(\cap_{i=1}^{k-1}\Delta^{V_i}_0)\cup\Delta^{V_k}_0)\to H_{n-1}(\Delta_0,(\cap_{i=1}^{k-1}\Delta^{V_i}_0)\cap\Delta^{V_k}_0)\to\cdots
\end{align*}
By previous discussion,
\begin{align*}
(\cap_{i=1}^{k-1}\Delta^{V_i}_0)\cup\Delta^{V_k}_0=\cap_{i=1}^{k-1} (\Delta^{V_i}_0\cup\Delta^{V_k}_0)=\cap_{i=1}^{k-1}\Delta^{V_i\cup V_k}_0.
\end{align*}
As $V_i\cup V_k\in \Theta'$, by induction, we know $$ H_n(\Delta_0,(\cap_{i=1}^{k-1}\Delta^{V_i}_0)\cup\Delta^{V_k}_0)=0\ \textrm{and}\ H_n(\Delta_0,\cap_{i=1}^{k-1}\Delta^{V_i}_0)=0.$$
Thus the claim follows. 
\end{proof}

In particular, $H_i(\Delta_0,\Delta_1)=0$ for all $i$. As $\Delta_0$ is simply-connected (\cite[Lem 4]{cumplido2020parabolic}) and $\Delta_1$ is connected,  $\pi_1(\Delta_0,\Delta_1)=0$. As $\Delta_1$ is simply-connected by Lemma~\ref{lem:Delta1sc}, by the relative version of Hurewicz Theorem \cite[Thm 4.32]{hatcher}, $\pi_i(\Delta_0,\Delta_1)=0$ for $i\ge 1$. By the long exact sequence of relative homotopy groups, we obtain that the inclusion $\Delta_1\to \Delta_0$ induces isomorphism on homotopy groups in all dimensions. Hence the Whitehead Theorem implies that $\Delta_1$ and $\Delta_0$ are homotopic equivalent. However, $\Delta_1$ and $\Delta^P_\Lambda$ are homotopic equivalent by Lemma~\ref{lem:homotopy injective}, and $\Delta^P_\Lambda$ is contractible by Assertion (1), Proposition~\ref{prop:mincutAn} and Lemma~\ref{lem:garsideA_n}. Thus $\Delta_\Lambda$ is contractible. 

To see $A_\Lambda$ satisfies the $K(\pi,1)$-conjecture, we induct on the number of vertices in $\Lambda$, and use Theorem~\ref{thm:kpi1} and the fact that the $K(\pi,1)$-conjecture is known for spherical Artin groups \cite{deligne} to conclude the proof.
\end{proof}

\begin{cor}
	\label{cor:cycle reduction all}
Suppose that for each non-spherical tree Coxeter diagram $\Lambda$, $\Delta_\Lambda$ satisfies the labeled 4-cycle condition and $A_\Lambda$ satisfies the $K(\pi,1)$-conjecture. Then any Artin group satisfies the $K(\pi,1)$-conjecture.
\end{cor}

\begin{proof}
We apply Theorem~\ref{thm:cycle reduction} with $\mathcal C$ being the class of all Coxeter diagrams, and $\mathcal C_1$ being the class of all Coxeter diagrams that are forests. For $\Lambda\in \mathcal C_1$, $\Delta_{\Lambda'}$ satisfies the labeled 4-cycle condition for each connected component $\Lambda'$ of $\Lambda$. Hence $\Delta_\Lambda$ is contractible. Now take $\Lambda\in \mathcal C_1$ non-spherical, then it has at least one non-spherical connected component $\Lambda'$. By \cite[Cor 2.4]{godelle2012k}, $A_{\Lambda'}$ and $A_{\Lambda'\setminus\{s\}}$ satisfy the $K(\pi,1)$-conjecture for each vertex $s\in\Lambda'$. \cite[Thm 3.1]{godelle2012k} implies that $\Delta_{\Lambda'}$ is contractible. Hence $\Delta_\Lambda$ is contractible. Then the corollary follows from Theorem~\ref{thm:cycle reduction}.
\end{proof}

\begin{cor}
	\label{cor:cycle reduction single}
	Let $\Lambda$ be a Coxeter diagram such that for any non-spherical induced subdiagram $\Lambda'$ which is a tree, $\Delta_{\Lambda'}$ satisfies the labeled 4-cycle condition and $A_{\Lambda'}$ satisfies the $K(\pi,1)$-conjecture. Then $A_\Lambda$ satisfies the $K(\pi,1)$-conjecture.
\end{cor}

\begin{proof}
We apply Theorem~\ref{thm:cycle reduction} such that $\mathcal C$ is the collection of all induced subdiagrams of $\Lambda$, and $\mathcal C_1$ is the collection of all induced subdiagrams of $\Lambda$ that is a forest. The rest of the proof is similar to Corollary~\ref{cor:cycle reduction all}.	
\end{proof}

\subsection{$\widetilde G_2$-reduction}
\label{subsec:G2}
\begin{definition}
An edge $e$ is \emph{$n$-solid} in $\Lambda$ if the graph $\Delta_{\Lambda,e}$ has girth $\ge 2n$.
\end{definition}
\begin{prop}(\cite[Prop 9.11]{huang2023labeled})
	\label{prop:tree}
	Suppose $\Lambda$ is tree Coxeter diagram. Suppose there exists a collection $E$ of open edges with label $\ge 6$ such that for each component $\Lambda'$ of $\Lambda\setminus E$ the Artin complex $\Delta_{\Lambda'}$ satisfies the labeled 4-cycle condition.  Then  $\Delta_{\Lambda,\Lambda'}$ satisfies the labeled 4-cycle condition for each component $\Lambda'$ of $\Lambda\setminus E$; and each edge $e\in E$ is $6$-solid in $\Lambda$.
\end{prop}

The following is a consequence of \cite[Prop 6.20, Lem 6.9]{huang2023labeled}. Recall that the notion of bowtie free is defined in Definition~\ref{def:bowtie free}.
\begin{lem}
	\label{lem:subdiagrams}
Given tree Coxeter diagrams $\Lambda',\Lambda,\Lambda'_1,\Lambda_1$ such that $\Lambda',\Lambda_1,\Lambda'_1$ are connected induced subdiagrams of $\Lambda$, $\Lambda'_1\subset \Lambda_1$, $\Lambda'_1\subset \Lambda'$ and $\Lambda'$ is linear. Suppose $\Delta_{\Lambda,\Lambda'}$ is bowtie free. Then $\Delta_{\Lambda_1,\Lambda'_1}$ is bowtie free.
\end{lem}

\begin{lem}
	\label{lem:enlarge0}
Let $\Lambda$ be a tree Coxeter diagram satisfying the assumption of Proposition~\ref{prop:tree}. Let $\Lambda'$ be a linear subdiagram of $\Lambda$ with its consecutive nodes being $\{s_i\}_{i=1}^n$, and let $\{b_1,b_2,b_3\}$ be three consecutive vertices of $\Lambda'$ such that $b_1$ is the closest to $s_1$ among them. We assume that
	\begin{enumerate}
		\item the edge $b_2b_3$ has label $\ge 6$;
		\item if $\Lambda'_1$ is the connected component of $\Lambda'\setminus\{b_1\}$ containing $s_n$ and $C_1$ is the component of $\Lambda\setminus\{b_1\}$ containing $\Lambda'_1$, then the  $(C_1,\Lambda'_1)$-relative Artin complex is bowtie free;
		\item if $\Lambda'_3$ is the connected component of $\Lambda'\setminus\{b_3\}$ containing $s_1$ and $C_3$ is the component of $\Lambda\setminus\{b_3\}$ that containing $\Lambda'_3$, then the $(C_3,\Lambda'_3)$-relative Artin complex is bowtie free.
	\end{enumerate}
	Then the $(\Lambda,\Lambda')$-relative Artin complex satisfies the bowtie free condition.
\end{lem}

\begin{proof}

Let $\Lambda''$ be the induced subdiagram spanned by $\{b_1,b_2,b_3\}$. Let $X=\Delta_{\Lambda,\Lambda''}$. We metric triangles in $X$ such that they are flat triangles in the Euclidean plane with angle $\pi/2$ at vertex of type $\hat b_2$, angle $\pi/6$ at vertex of type $\hat b_1$ and angle $\pi/3$ at vertex of type $\hat b_3$. As the link of each vertex of type $\hat b_2$ in $X$ is a bipartite graph, the link of each vertex of type $\hat b_1$ has girth $\ge 12$ by Lemma~\ref{lem:link}, Proposition~\ref{prop:tree} and assumption (1), and the link of each vertex of type $\hat b_3$ has girth $\ge 6$ by Lemma~\ref{lem:link} and assumption (3). Hence $X$ is locally CAT$(0)$ with such metric \cite[Thm II.5.5 and Lem II.5.6]{BridsonHaefliger1999}. As $X$ is simply-connected (\cite[Lem 6.1]{huang2023labeled}), it is CAT$(0)$ \cite[Thm II.4.1]{BridsonHaefliger1999}.
 
We will prove the lemma by induction on the number of vertices in $\Lambda'$. The base case is $\Lambda'$ has three vertices, i.e. $\Lambda'=\Lambda''$, in which case $\Delta_{\Lambda,\Lambda'}$ is bowtie free by Proposition~\ref{prop:tree} and \cite[Lem 9.3]{huang2023labeled}. Actually, we deduce that $\Delta_{\Lambda,\Lambda'}$ is bowtie free whenever $b_3=s_n$.

Assume the lemma holds for all pairs $\Lambda'\subset \Lambda$ satisfies the assumptions of the lemma such that $\Lambda'$ has $\le n-1$ vertices. Now take $\Lambda'\subset \Lambda$ such that $\Lambda'$ has $n$ vertices. It suffices to verify the two assumptions of \cite[Lem 6.10]{huang2023labeled}. 
Take $v\in \Delta_{\Lambda,\Lambda'}$ of type $\hat s_1$. If $s_1\neq b_1$, then by Lemma~\ref{lem:link}, $\lk(v,\Delta_{\Lambda,\Lambda'})\cong \Delta_{\Theta_1,\Theta'_1}$, where $\Theta'_1=\Lambda'\setminus\{s_1\}$ and $\Theta_1$ is the component of $\Lambda\setminus\{s_1\}$ that contains $\Theta'_1$. Then $\Theta'_1\subset\Theta_1$ satisfies the induction assumption as 
\begin{enumerate}
	\item by Lemma~\ref{lem:subdiagrams} any connected subdiagram of $\Lambda$ (in particular $\Theta$) satisfies the assumption of Proposition~\ref{prop:tree};
\item 	by Lemma~\ref{lem:link} and Lemma~\ref{lem:subdiagrams}, the pair $\Theta'_1\subset\Theta_1$ satisfies three assumptions of the lemma.
\end{enumerate}
Hence $\lk(v,\Delta_{\Lambda,\Lambda'})$ is bowtie free. If $s_1=b_1$, then $C_1=\Theta_1$ and $\Lambda'_1=\Theta'_1$ and $\Delta_{\Theta_1,\Theta'_1}$ is bowtie free by assumption (2). The case $v$ is of type $\hat s_n$ can be handled similarly.
	

It remains to verify assumption (2) of \cite[Lem 6.10]{huang2023labeled}. Take embedded 4-cycle $x_1y_1x_2y_2$ in $\Delta_{\Lambda,\Lambda'}$ of type $\hat s_1\hat s_n\hat s_1\hat s_n$.

\begin{claim*}
	If there is an edge path $w_1w_2\cdots w_k$ in $\Delta_{\Lambda,\Lambda'}$ from $x_1$ to $x_2$ such that for $1\le i\le k$, $w_i$ is adjacent to each of $y_1$ and $y_2$, then there is a vertex $z\in \Delta_{\Lambda,\Lambda'}$ such that $z$ is adjacent to each of $\{x_1,x_2,y_1,y_2\}$.  
	
	If there is an edge path $u_1u_2\cdots u_k$ in $\Delta_{\Lambda,\Lambda'}$ from $y_1$ to $y_2$ such that for $1\le i\le k$, $u_i$ is adjacent to each of $x_1$ and $x_2$, then there is a vertex $z\in \Delta_{\Lambda,\Lambda'}$ such that $z$ is adjacent to each of $\{x_1,x_2,y_1,y_2\}$.  
\end{claim*}

\begin{proof}
We only prove the first statement, as the second can be proved in a similar way.	
We induct on $k$. The base case of $k=3$ is clear. Now we consider $k>3$. Suppose $w_2,w_3,w_4$ have type $\hat s_{i_2},\hat s_{i_3},\hat s_{i_4}$ respectively. Then $i_2>1$. If $i_3>i_2$, then $w_1$ and $w_3$ are adjacent by Lemma~\ref{lem:transitive} and we finish by induction assumption. Assume $i_2>i_3$. Similarly, we can assume $i_4>i_3$. Let $\Theta$ (resp. $\Theta'$) be the connected component of $\Lambda\setminus\{s_{i_3}\}$ (resp. $\Lambda'\setminus\{s_{i_3}\}$) that contains $s_n$. Then $y_1w_2y_2w_4$ is a 4-cycle in $\lk(w_3,\Delta_{\Lambda,\Lambda'})\cong \Delta_{\Theta,\Theta'}$.

 However, $\Delta_{\Theta,\Theta'}$ is bowtie free. This is similar to the proof of $\Delta_{\Theta_1,\Theta'_1}$ being bowtie free as above, if $s_{i_3}$ is between $s_1$ and $b_1$ in $\Lambda'$ (including these two endpoint). If $s_{i_3}$ is between $b_2$ and $s_n$, then $\Theta\subset C_1$ and $\Theta'\subset\Lambda'_1$, hence $\Delta_{\Theta,\Theta'}$ is bowtie free by assumption (2) and Lemma~\ref{lem:subdiagrams}.
 
 It follows that there is $w'_3\in \Delta_{\Theta,\Theta'}$ that is adjacent to each of $\{y_1,w_2,y_2,w_4\}$. Thus $w'_3$ is of type $\hat s_{i'_3}$ with $i'_3>i_2$. Hence $w'_3$ is adjacent to $w_1$ by Lemma~\ref{lem:transitive}. As $w'_3$ is adjacent $w_4$, this decreases the length of the edge path.
\end{proof}

We view $\Delta_{\Lambda,\Lambda'}$ and $X$ as subcomplexes of $\Delta_\Lambda$. Let $X_i$ be the full subcomplex of $X$ spanned by vertices of $X$ that are adjacent or equal to $x_i$. Similarly we define $Y_i$. Now we prove for $i=1,2$, $X_i$ and $Y_i$ are convex subcomplexes of $X$. This claim is only interesting when $x_i,y_i\notin X$. So we can assume $s_1\neq b_1$ and $s_n\neq b_3$. By Lemma~\ref{lem:link}, $X_i\cong \Delta_{\Theta_1,\Lambda''}$, so \cite[Lem 6.2]{huang2023labeled} implies $X_i$ is connected. Similarly, $Y_i$ is connected. Convexity of $X_i$ follows from \cite[Lem 9.4]{huang2023labeled}. To see $Y_1$ is convex, it suffices to show $Y_1$ is locally convex in $X$ around each vertex of $Y_1$, which reduces further to show for each vertex $y\in Y_1$, $\lk(y,Y_1)$ is $\pi$-convex in $\lk(y,X)$, i.e. for any two points in $\lk(y,Y_1)$ of distance $<\pi$, the shortest path in $\lk(y,X)$ between these two points is contained in $\lk(y,Y_1)$. If $y$ is of type $\hat b_3$, as $b_3$ separates $\{b_1,b_2\}$ from $s_n$ in $\Lambda$,  Lemma~\ref{lem:transitive} implies that $\lk(y,Y_1)=\lk(y,X)$, so $\pi$-convexity is clear. If $y$ is of type $\hat b_2$, then both $\lk(y,Y_1)$ and $\lk(y,X)$ are complete bipartite graphs with edge length $\pi/2$, hence $\lk(y,Y_1)$ is $\pi$-convex in $\lk(y,X)$.

Suppose $y$ is of type $\hat b_1$. By Lemma~\ref{lem:link}, $\lk(y,X)\cong \Delta_{C_1,b_2b_3}$. Let $b_4$ be the vertex in $\Lambda'$ that is adjacent to $b_3$ such that $b_4\neq b_2$. Let $X_{234}=\Delta_{C_1,b_2b_3b_4}$, such that triangles in $X_{234}$ are flat triangles in Euclidean plane with angle $\pi/2$ at vertices of type $\hat b_3$, angle $\pi/3$ at vertices of type $\hat b_2$, and angle $\pi/6$ at vertices of type $\hat b_4$. By Proposition~\ref{prop:tree}, the link of each vertex of type $\hat b_4$ in $X_{234}$ have girth $\ge 12$. Let $\Theta_2$ (resp. $\Theta'_2$) be the connected component of $C_1\setminus\{b_2\}$ (resp. $\Lambda'\setminus\{b_2\}$) that contains $s_n$. Then Lemma~\ref{lem:subdiagrams} and assumption (2) imply that $\Delta_{\Theta_2,\Theta'_2}$ is bowtie free. As the link of a vertex of type $\hat b_2$ in $X_{234}$ is a copy of $\Delta_{\Theta_2,b_3b_4}$ by Lemma~\ref{lem:link}, the bowtie free condition implies such link has girth $\ge 6$. Thus $X_{234}$ is CAT$(0)$. Let $Y_{234}$ be the full subcomplex of $X_{234}$ spanned by vertices that are adjacent to $y_1$ (we view both $X_{234}$ and $y_1$ as sitting inside $\Delta_{C_1,\Lambda'_1}\cong \lk(y,\Delta_{\Lambda,\Lambda'})$). Then $Y_{234}$ is a convex subcomplex of $X_{234}$ by \cite[Lem 9.4]{huang2023labeled}. Note that for any vertex $v$ of type $\hat b_4$ in $Y_{234}$, we have $\lk(v,X_{234})=\lk(v,Y_{234})$ by Lemma~\ref{lem:transitive}. 
As $\lk(y,X)$ is the induced subcomplex of $X_{234}$ spanned by vertices of type $\hat b_2$ and $\hat b_3$, and $\lk(y,Y_1)$ is the induced subcomplex of $Y_{234}$ spanned by vertices of type $\hat b_2$ and $\hat b_3$, by \cite[Lem 9.8 (2)]{huang2023labeled}, for any two vertices of $\lk(y,Y_1)$ that are joined by an edge path in $\lk(y,X)$ with $<6$ edges, any shortest edge path of $\lk(y,X)$ connecting these two vertices are contained in $\lk(y,Y_1)$. Thus $\lk(y,Y_1)$ is $\pi$-convex in $\lk(y,X)$. This shows $Y_1$ is convex subcomplex of $X$. Similarly, $Y_2$ is a convex subcomplex of $X$.
	
\begin{claim*}
Either $Y_1\cap Y_2\cap X_i\neq\emptyset$ for $i=1,2$, or $X_1\cap X_2\cap Y_j\neq\emptyset$ for $j=1,2$.
\end{claim*}

\begin{proof}
A $4$-gon in $X$ is \emph{admissible} if it is made of fours geodesic segments $\{\bS_i\}_{i=1}^4$ such that $\bS_1$ goes from $p_1\in X_1\cap Y_1$ to $p_2\in Y_1\cap X_2$, $\bS_2$ goes from $p_2\in Y_1\cap X_2$ to $p_3\in Y_2\cap X_2$, $\bS_3$ goes from $p_3\in Y_2\cap X_2$ to $p_4\in Y_2\cap X_1$ and $\bS_4$ goes from $Y_2\cap X_1$ to $Y_1\cap X_1$. Let $\ell(\bS_i)$ denotes the length of $\bS_i$. The \emph{perimeter} of this 4-gon is defined to be $\sum_{i=1}^4\ell(\bS_i)$.
Let $P$ be a 4-gon with perimeter minimized among all admissible 4-gons. We say $P$ is non-degenerate at $p_1$, if $p_4\neq p_1$ and $p_2\neq p_1$.
The first three paragraphs of the proof of \cite[Lem 9.6]{huang2023labeled} imply that such $P$ exists, and if $P$ is non-degenerate at the corner $p_1$ with $\angle_{p_1}(p_2,p_4)<\pi$, then $p_1$ is a vertex of $X$, and the shortest arc in $\lk(p_1,X)$ from $\log_{p_1}(p_4)$ to $\log_{p_1}(p_2)$ does not contain any point in $\lk(p_1,X_1\cap Y_1)$, where $\log_{p_1}(p_4)$ is the point in $\lk(p_1,X)$ given by the geodesic segment from $p_1$ to $p_4$. Similar statements hold at other corners of $P$.

By Lemma~\ref{lem:transitive}, if  $p_1$ is of type $\hat b_1$, then $\lk(p_1,X_1)=\lk(p_1,X)$; if $p_1$ is of type $\hat b_3$, then $\lk(p_1,Y_1)=\lk(p_1,X)$. Thus the previous paragraph implies that 
if $P$ is non-degenerate at $p_1$ with $\angle_{p_1}(p_2,p_4)<\pi$, then $p_1$ is not type $\hat b_1$ or $\hat b_3$. Thus $p_1$ is of type $\hat b_2$, and $\lk(p_1,X)$ is a complete bipartite graph which is a join $B_1\circ B_3$ where $B_1$ (resp. $B_3$) is made of all type $\hat b_1$ (resp. $\hat b_3$) vertices in $\lk(p_1,X)$. Lemma~\ref{lem:transitive} implies that $B_1\subset \lk(p_1,Y_1)$ and $B_3\subset\lk(p_1,X_1)$. The complete bipartite structure implies that at $y_1$ we can have four triangles $\{\delta_i\}_{i=1}^4$ forming a piece of a flat plane as in Figure~\ref{fig:4triangles} left such that the segments $p_1p_2$ and $p_1p_4$ have their initial subsegments contained in one of $\{\delta_i\}_{i=1}^4$. Let $\{z_i\}_{i=1}^4$ be vertices in Figure~\ref{fig:4triangles} left such that $z_1,z_3$ having type $\hat b_3$ and $z_2,z_4$ having type $\hat b_1$. Note that the edges $y_1z_4,y_1z_2$ are contained in $Y_1$, and the edges $y_1z_1,y_1z_3$ are contained in $X_1$. Convexity of $X_1$ implies that if $\angle_{y_1}(z_4,p_4)<\pi/2$, then $\sigma_4,\sigma_3\subset X_1$; if $\angle_{y_1}(z_2,p_4)<\pi/2$, then $\sigma_1,\sigma_2\subset X_1$. An analogous statement can be deduced from convexity of $Y_1$. Thus the previous paragraph implies that $\angle_{p_1}(p_2,p_4)\ge\pi/2$. 

Suppose $P$ is non-degenerate at each of its 4 corners. Then by previous discussion the angle at each corner is $\ge \pi/2$. By \cite[Chapter II.2.12]{BridsonHaefliger1999}, the angle at each corner is $\pi/2$ and $P$ bounds a flat convex rectangle $R$ in $X$. As the angle at $p_1$ is $\pi/2$, the only way the condition at the end of the first paragraph is satisfied is that the segment $p_4p_1$ contains one of the edges $p_1z_1,p_1z_3$ and $p_1p_2$ contains one of the edges $p_1z_2,p_1z_4$. We have similar conclusions at other corners of $P$. Thus $P$ is contained in the 1-skeleton of $X$ and the flat rectangle bounded by $P$ is tessellated by flat triangles with angles $(\pi/6,\pi/3,\pi/2)$. Figure~\ref{fig:4triangles} right shows part of such tessellation around the corner $p_1$. Let $z_5,z_6,z_7,z_8$ be the vertices in Figure~\ref{fig:4triangles} right with $z_5,z_6\in X_1$. As $z_6$ has type $\hat b_1$ and $z_7$ has type $\hat b_2$, Lemma~\ref{lem:transitive} implies that $x_1\sim z_7$, hence $z_7\in X_1$. As $X_1$ is a convex subcomplex of $X$, we know $z_8\in X_1$. However, $z_8\in Y_1$, which contradicts the end of the first paragraph of the proof. Thus the possibility of $P$ being non-degenerate at each of its 4 corners is ruled out.

\begin{figure}
	\centering
	\includegraphics[scale=1.2]{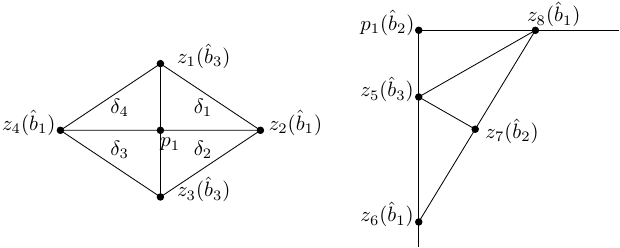}
	\caption{}
	\label{fig:4triangles}
\end{figure}

It remains to consider $P$ is degenerate at one of its corners, say $p_1$. Then we have either $p_1=p_4$ or $p_1=p_2$. If both equality holds, then we are done. Suppose $p_1=p_4$ and $p_1\neq p_2$. By the same reasoning we can assume in addition that $p_3\neq p_4$. Note that $p$ is degenerate at one of  $p_2,p_3$, as otherwise the angles at both corner are $\ge \pi/2$, forcing the triangle $p_1p_2p_3$ to be degenerate. If $P$ is degenerate at $p_2$, as $p_1\neq p_2$, we must have $p_2=p_3$, and the claim follows. If $P$ is degenerate at $p_3$, as $p_3\neq p_4$, we still have $p_2=p_3$. It remains to consider $p_1\neq p_4$ and $p_1=p_2$, which can be handled in a similar way.
\end{proof}
As $Y_1\cap Y_2\cap X_i$ is convex, hence connected, if non-empty, taking an edge path in $Y_1\cap Y_2$ from a vertex in $Y_1\cap Y_2\cap X_1$ to a vertex in $Y_1\cap Y_2\cap X_2$ gives an edge path $w_1w_2\cdots w_k$ in $\Delta_{\Lambda,\Lambda'}$ from $x_1$ to $x_2$ such that each $w_i$ is adjacent to both $\{y_1,y_2\}$, as desired. The case of $X_1\cap X_2\cap Y_i\neq\emptyset$ is similar.
\end{proof}

\begin{prop}
	\label{prop:G2bowtie free}
	Suppose $\Lambda$ is a tree Coxeter diagram. 	Let $\Lambda'$ be the diagram obtained from $\Lambda$ by removing all open edges of label $\ge 6$.	Suppose the labeled 4-cycle condition holds for $\Delta_{\Gamma}$ where $\Gamma$ is any connected component of $\Lambda'$. Then $\Delta_\Lambda$ satisfies the labeled 4-cycle condition.
\end{prop}

\begin{proof}
	By \cite[Prop 6.17]{huang2023labeled}, it suffices to show $\Delta_{\Lambda,\Theta}$ is bowtie free for any maximal linear subdiagram $\Theta$ of $\Lambda$. This follows immediately from Proposition~\ref{prop:tree} and Lemma~\ref{lem:subdiagrams} if $\Theta$ does not contain any edge of label $\ge 6$. Now suppose $\Theta$ contains at least one edge of label $\ge 6$. Let $\{e_i\}_{i=1}^{k-1}$ be all such edges and let $\{\Theta_i\}_{i=1}^{k}$ be all connected components of $\Theta\cap \Lambda'$ such that $e_i$ is between $\Theta_i$ and $\Theta_{i+1}$.
	Let $\Lambda_i$ be the connected component of $\Lambda'$ that contains $\Theta_i$. By Proposition~\ref{prop:tree}, $\Delta_{\Lambda,\Lambda_i}$ satisfies the labeled 4-cycle condition. Hence $\Delta_{\Lambda,\Theta_i}$ satisfies the labeled 4-cycle condition, in particular, it is bowtie free by \cite[Lem 6.14]{huang2023labeled}.
	
	Now we use Lemma~\ref{lem:enlarge0} to show $\Delta_{\Lambda,\Theta_1\cup e_1\cup\Theta_2}$ is bowtie free. We choose consecutive vertices $\{b_1,b_2,b_3\}$ such that $b_1\in \Theta_1$ and $b_2,b_3\in e_1$. Then assumption (1) of Lemma~\ref{lem:enlarge0} follows from Proposition~\ref{prop:tree}. By Proposition~\ref{prop:tree}, \cite[Lem 9.3 and Prop 6.20]{huang2023labeled}, $\Delta_{\Lambda,e_1\cup\Theta_2}$ is bowtie free, hence assumption (2) of Lemma~\ref{lem:enlarge0} follows. Assumption (3) of Lemma~\ref{lem:enlarge0} follows from the previous paragraph and Lemma~\ref{lem:subdiagrams}.
	
	By repeatedly applying Lemma~\ref{lem:enlarge0} in such a way, $\Delta_{\Lambda,\Theta}$ is bowtie free.
\end{proof}

\begin{cor}
	\label{cor:G2}
	Suppose that for each non-spherical tree Coxeter diagram $\Lambda$ such that all edge labels are $\le 5$, $\Delta_\Lambda$ satisfies the labeled 4-cycle condition and $A_\Lambda$ satisfies the $K(\pi,1)$-conjecture. Then any Artin group satisfies the $K(\pi,1)$-conjecture.
\end{cor}

\begin{proof}
By Proposition~\ref{prop:G2bowtie free} and \cite[Prop 9.12]{huang2023labeled}, for any tree Coxeter diagram $\Lambda$, $\Delta_\Lambda$ satisfies the labeled 4-cycle and $A_\Lambda$ satisfies the $K(\pi,1)$-conjecture. Now we are done by Corollary~\ref{cor:cycle reduction all}.
\end{proof}

\begin{cor}
	\label{cor:G2 single}
	Let $\Lambda$ be a Coxeter diagram such that for any non-spherical induced subdiagram $\Lambda'$ which is a tree with all edge labels $\le 5$, $\Delta_{\Lambda'}$ satisfies the labeled 4-cycle condition and $A_{\Lambda'}$ satisfies the $K(\pi,1)$-conjecture. Then $A_\Lambda$ satisfies the $K(\pi,1)$-conjecture.
\end{cor}

\begin{proof}
By Proposition~\ref{prop:G2bowtie free} and \cite[Prop 9.12]{huang2023labeled}, for any induced tree subdiagram $\Lambda'$ of $\Lambda$, $\Delta_{\Lambda'}$ satisfies the labeled 4-cycle and $A_{\Lambda'}$ satisfies the $K(\pi,1)$-conjecture. Now we are done by Corollary~\ref{cor:cycle reduction single}.
\end{proof}

\begin{cor}
	\label{cor:spherical combine}
Let $\Lambda$ be a Coxeter diagram. Suppose that, after removing all edges of $\Lambda$ with labels at least $6$, every connected component of the resulting diagram is either spherical, or locally reducible, or $3$-dimensional. Then the Artin group $A_\Lambda$ satisfies the $K(\pi,1)$-conjecture.
\end{cor}

\begin{proof}
For each Coxeter diagram $\Lambda$ which is connected and spherical, $\Delta_\Lambda$ satisfies the labeled 4-cycle condition by \cite[Cor 8.2]{huang2023labeled}, and $A_\Lambda$ satisfies the $K(\pi,1)$-conjecture by \cite{deligne}. For locally reducible tree diagram $\Lambda$, the $K(\pi,1)$-conjecture of $A_\Lambda$ follows from a combination of \cite{CharneyDavis} and \cite{charney2000tits}. Moreover, $\Delta_\Lambda$ satisfies the labeled 4-cycle by \cite[Cor 9.14]{huang2023labeled}. These two conditions in the case of $\Lambda$ being a tree and $A_\Lambda$ being 3-dimensional follow from \cite{huang2025353}. Thus the corollary follows from Corollary~\ref{cor:G2 single}.
\end{proof}

\section{Some generalizations}
\subsection{Link conditions for simplicial complexes of type $S$}

Let $S=\{s_1,\ldots,s_n\}$. A simplicial complex $\Delta_S$ is of \emph{type $S$} if it is flag, all the maximal simplices of $\Delta_S$ has dimension $n-1$ and there is a type function $\type$ from the vertex set of $\Delta_S$ to $\{\hat s_1,\hat s_2,\ldots,\hat s_n\}$ such that $\type(x)\neq\type(y)$ whenever $x$ and $y$ are adjacent vertices of $\Delta_S$. This labeling induces a bijection between $\{\hat s_1,\hat s_2,\ldots,\hat s_n\}$ and the vertex set of each maximal simplex of $\Delta_S$. For $S'\subset S$, we define the \emph{$(S,S')$ relative subcomplex} $\Delta_{S,S'}$ be the full subcomplex of $\Delta_S$ spanned by vertices of type $\hat s$ with $s\in S'$. We define the type of simplices in $\Delta_S$ in the same way as Section~\ref{subsec:rel Artin complex}. For a simplex $\sigma\subset \Delta_{S}$ of type $\hat T$, we define $\Delta^\sigma_{S\setminus T}$ to be $\lk(\sigma,\Delta_{S})$, which is a simplicial complex of type $S\setminus T$. For any $R\subset S\setminus T$, let $\Delta^\sigma_{S\setminus T,R}$ be the full subcomplex of $\Delta^\sigma_{S\setminus T}$ spanned by vertices of type $\hat s$ with $s\in R$. Then for any $T\subset S'\subset S$, $\Delta^\sigma_{S\setminus T,S'\setminus T}=\lk(\sigma,\Delta_{S,S'})$. 

\begin{remark}
	\label{rmk:notation}
If $R_1\subset R_2\subset S\setminus T$, then $\Delta^\sigma_{S\setminus T, R_1}\subset \Delta^\sigma_{S\setminus T, R_2}$. So $\Delta^\sigma_{S\setminus T,R_1}$ can also be viewed as a relative subcomplex inside $\Delta^\sigma_{S\setminus T, R_2}$ which is a simplicial complex of type $R_2$. Thus we will also use the notation $\Delta^\sigma_{R_2,R_1}$ to denote $\Delta^\sigma_{S\setminus T, R_1}$. These two notations give the same complex, but have a different emphasis on what is the ambient complex where this complex sits as a relative subcomplex.
\end{remark}

We say the complex $\Delta_S$ of type $S$ is \emph{regular} if
\begin{enumerate}
	\item for any simplex $\sigma$ of codimension $\ge 2$ in $\Delta_S$, $\lk(\sigma,\Delta_S)$ is connected;
	\item for any simplex $\sigma$ of codimension $\ge 3$ in $\Delta_S$, $\lk(\sigma,\Delta_S)$ is simply connected.
\end{enumerate}
In the definition of regular, we do not exclude the possibility that $\sigma=\emptyset$, in which case the link is $\Delta_S$. The regular condition is the appropriate analogue of Lemma~\ref{lem:c} and Lemma~\ref{lem:sc} in this more general context.

\begin{lem}
	\label{lem:sc general}
Suppose $\Delta_S$ is regular. Then for any $S'\subset S$, $\Delta_{S,S'}$ is regular as a simplicial complex of type $S'$. In particular, $\Delta_{S,S'}$ is connected if $|S'|\ge 2$ and $\Delta_{S,S'}$ is simply connected if $|S'|\ge 3$.
\end{lem}

\begin{proof}
It suffices to show that if $\Delta_{S,T}$ is regular, then $\Delta_{S,T\setminus\{t\}}$ is regular for any $t\in T$. Let $\sigma\subset \Delta_{S,T\setminus\{t\}}$ be a (possibly empty) simplex of type $\hat R$. Then $M:=\lk(\sigma,\Delta_{S,T\setminus\{t\}})$ is obtained from $L:=\lk(\sigma,\Delta_{S,T})$ by removing the open star of all vertices of type $\hat t$. If $\sigma$ has codimension $\ge 2$ in $\Delta_{S,T\setminus\{t\}}$, it has codimension $\ge 3$ in $\Delta_{S,T}$. As $\Delta_{S,T}$ is regular, $L$ is simply-connected and for any vertex $x\in L$ of type $\hat t$, 
$$\lk(x,L)=\lk(\{x\}\ast\sigma,\Delta_{S,T})$$
is connected. So $L$ is obtained from $M$ attaching a collection of cones (i.e. close starts of certain vertices in $L$) such that the base of the cones are connected. Hence $L$ is connected implies $M$ is connected. If $\sigma$ has codimension $\ge 3$ in $\Delta_{S,T\setminus\{t\}}$, then a similar discussion implies that $L$ is simply connected and $L$ is obtained from $M$ attaching a collection of cones such that the base of the cones are simply connected. By Van Kampen, $M$ is simply connected.
\end{proof}

Let $\Delta'_S$ be the barycentric subdivision of $\Delta$, and the meaning of $x\sim y$ for points $x,y\in \Delta_S$ is the same as Section~\ref{subsec:generalized cycles}. The complex $\Delta_S$ satisfies \emph{property C$_4$} if the first paragraph of Lemma~\ref{lem:4-cycle} holds true for $\Delta_S$. Note that if $\Delta_S$ satisfies property C$_4$, then for any simplex $\sigma$ of type $\hat T$, $\Delta^\sigma_{S\setminus T}$ also satisfies property C$_4$ as a complex of type $S\setminus T$, as each simplex in $\Delta^\sigma_{S\setminus T}$ can be enlarged to a simplex in $\Delta_S$ by joining with $\sigma$.

Let $\Lambda$ be a diagram with vertex set $S$ (here we do not put edge labeling on $\Lambda$). 
We will also use $\Delta_{\Lambda,\Lambda'}$ to denote $\Delta_{S,S'}$, where $\Lambda'$ is the subdiagram of $\Lambda$ induced on $S'$.
We say $\Delta_S$ is \emph{modeled on $\Lambda$} if
 for any simplex $\sigma$ of type $\hat T$ in $\Delta_S$, we have a join decomposition 
 \begin{equation}
 	\label{eq:join}
 \Delta^\sigma_{S\setminus T}=\Delta^\sigma_{S\setminus T,\Pi_1}*\Delta^\sigma_{S\setminus T,\Pi_2}*\cdots*\Delta^\sigma_{S\setminus T,\Pi_k}
 \end{equation}
 where $\Pi_1,\Pi_2,\ldots,\Pi_k$ are connected components of the subdiagram of $\Lambda$ induced on $S\setminus T$.
This condition implies that the analogue 
of Lemma~\ref{lem:transitive} and Lemma~\ref{lem:link} holds true, more precisely:
\begin{lem}
Let $\Delta_S$ be a simplicial complex of type $S$ modeled on a diagram $\Lambda$. Let $\Lambda'\subset \Lambda$ be an induced subdiagram. Then the following holds:
\begin{enumerate}
	\item Let $v\in \Delta_{\Lambda,\Lambda'}$ be a vertex of type $\hat s$ with $s\in \Lambda'$. Let $\Lambda_s$ and $\Lambda'_s$ be the induced subgraph of $\Lambda$ and $\Lambda'$ respectively spanned all the vertices which are not $s$. Let $I_s$ be the union of connected components of $\Lambda_s$ that contain at least one component of $\Lambda'_s$. Then $\Lambda'_s\subset I_s$ and  $\lk(v,\Delta_{\Lambda,\Lambda'})=\Delta^v_{I_s,\Lambda'_s}$ (see Remark~\ref{rmk:notation} for clarification of notations).
	\item 	Given vertices $x_1,x_2,x_3$ of types $\hat S_1,\hat S_2,\hat S_3$ in $\Delta'_{\Lambda}$. Suppose that for any $x\in S_1\setminus S_2$ and $y\in S_3\setminus S_2$, $x$ and $y$ are in different components of $\Lambda\setminus S_2$. Suppose $x_1\sim x_2$ and $x_2\sim x_3$. Then $x_1\sim x_3$.
\end{enumerate}
\end{lem}

The proof of the following is identical to Lemma~\ref{lem:dr}.

\begin{lem}
	\label{lem:drgeneral}
Let $\Delta_S$ be a simplicial complex of type $S$.
Let $T$ be a subset of $S$. Suppose that for any simplex $\sigma$ in $\Delta_{S}$ of type $\hat R$ with $R\subset T$ and $R\neq\emptyset$, $\Delta^\sigma_{S\setminus R}$ is contractible. Then $\Delta_S$ deformation retracts onto $\Delta_{S,S\setminus T}$.
\end{lem}

We define labeled 4-cycle condition and strong labeled 4-cycle condition for $\Delta_S$ and $\lk(\sigma,\Delta_S)$ for any simplex $\sigma\subset\Delta_S$ in exactly the same way as Definition~\ref{def:labaled}. For example, the definition of labeled 4-cycle condition is as follows.
\begin{definition}
	\label{def:labaled general}
	Given a (possibly empty) simplex $\sigma$ of type $\hat T$ in $\Delta_{\Lambda}$.	Let $\Lambda_{S\setminus T}$ be the induced subdiagram on $S\setminus T$.
	We say $\lk(\sigma,\Delta_\Lambda)$ satisfies \emph{labeled 4-cycle condition}, if for any embedded 4-cycle $x_1x_2x_3x_4$ in $\lk(\sigma,\Delta_\Lambda)$ of type $\hat s\hat t \hat s\hat t$, and any connected induced subgraph $\Lambda'\subset \Lambda_{S\setminus T}$ containing both $s$ and $t$, there is a vertex $z$ such that  $z$ is adjacent to each $x_i$ for $1\le i\le 4$, and  $z$ has type $\hat r$ with $r\in \Lambda'$.
\end{definition}

All the results in Section~\ref{sec:labeled 4-cycle} have natural analogues for simplicial complexes of type $S$ with identical proofs, if we assume $\Delta_S$ satisfies property C$_4$, as this replace of the dependence of results in Section~\ref{sec:labeled 4-cycle} on the first part of Lemma~\ref{lem:4-cycle}.

An edge path $P$ in $\Lambda$ from $a$ to $b$ is \emph{unbranched} if it is a nontrivial embedded path except possibly at endpoints, all interior vertices of $P$ has valence 2 in $\Lambda$, and both $a,b$ are not valence one vertices in $\Lambda$.
We define the subdiagram $\Lambda_P$ and the minimal cut complex $\Delta^P_\Lambda$ in the same way as Section~\ref{sec:minimal cut complex}. 

\begin{prop}
	\label{prop:mincutgeneral}
	Let $\Delta_S$ be a regular simplicial complex of type $S$ modeled on the diagram $\Lambda$. Suppose that
	\begin{enumerate}
		\item $\Delta_\Lambda$ satisfies condition C$_4$;
      \item for any simplex $\sigma\neq\emptyset$, $\lk(\sigma,\Delta_S)$ satisfies the labeled 4-cycle condition.
     \end{enumerate} 
 Let $P$ be an unbranched path in $\Lambda$  such that $\Lambda_P$ is connected. Then the complex $\Delta^P_\Lambda$ is an $\widetilde A_n$-like complex in the sense of Definition~\ref{def:A_n}.
\end{prop}

\begin{proof}
The proof follows the same argument as in Section~\ref{sec:minimal cut complex}, except the following adjustment in the proof of Lemma~\ref{lem:Delta1sc} and Lemma~\ref{lem:link of mincut}. 

First we adjust the definition of $\Delta_1$ and $\Delta_2$. Let $P$ and $\Sigma_P$ be as in Section~\ref{sec:minimal cut complex}. 
Vertices of $\Delta_1$ are in 1-1 correspondence faces in $\Delta_S$ of type $\hat T$ with $T\in \Sigma_P$. Such a vertex is \emph{special} if $T$ is a vertex of $P$, otherwise the vertex is \emph{non-special}. A special vertex is adjacent to another vertex if the two associated faces are contained in a common face of $\Sigma_S$. Two non-special vertices are adjacent one associated face is contained in another. Then $\Delta_1$ is the flag complex on its 1-skeleton. $\Delta_2$ has the same vertex set, and a collection of vertices of $\Delta_2$ span a simplex if the associated faces are contained in a common face of $\Delta_S$.

We define a \emph{Cayley complex} $Z$ for $\Delta_S$ as follows. A \emph{chamber} in $\Delta_S$ is a top-dimensional simplex. Two chambers are \emph{adjacent} if they intersect in a codimension 1 face. Vertices of $Z$ are in 1-1 correspondence with chambers in $\Delta_S$, and two vertices are joined by an edge if their corresponding chambers are adjacent. For each embedded edge path in $Z$ such that the common intersection of chambers associated with vertices on this path is a face in $\Delta_S$ of codimension $\le 2$, we fill in this path by a polygonal 2-disk. For each face $\sigma\subset \Delta_S$, let $Z(\sigma)$ be the full subcomplex of $Z$ spanned by vertices that corresponds to chambers that contain $\sigma$. Then $\Delta_S$ being regular implies that $Z(\sigma)$ is connected and simply-connected for each face $\sigma$ of $\Delta_S$, including the empty face in which case $Z(\sigma)=Z$. In the proof of Lemma~\ref{lem:Delta1sc}, the role of $Z$ is replaced the complex defined here, the role of $A_\Lambda$ and its left cosets of standard parabolic subgroups are replaced by vertices of $Z$ or $Z(\sigma)$, and the role of $Z(gA_{\hat T})$ is replaced by $Z(\sigma)$. 

Now we describe the adjustment for Lemma~\ref{lem:link of mincut}. Given vertex $x\in \Delta^P_\Lambda$ of type $\hat T$ with $T\in \mathcal C_P$, and let $\sigma$ be the simplex of $\Delta_S$ with $x$ in the interior. Then $\Delta^\sigma_{\Lambda\setminus T}=\Delta^\sigma_{\Lambda\setminus T,\Lambda_T}*\Delta^\sigma_{\Lambda\setminus T,U}$ where $R$ is the union of all connected component of $\Lambda\setminus T$ that is not $\Lambda_T$. Vertices of $\lk(x,\Delta^P_\Lambda)$ correspond to faces $\tau$ of $\Delta_{S}$ of type $\hat R$ such that $R\in \mathcal C_{P,T}$ and $\tau\cup \sigma$ is contained in a common simplex of $\Delta_S$. Let $R'=R\cap \Lambda_T$. Then $\tau$ uniquely determines a simplex $\tau'$ in the join factor $\Delta^\sigma_{\Lambda\setminus T,\Lambda_T}$ by intersecting, such that $\tau'$ is of type $\Lambda_T\setminus R'$. So vertices of $\lk(x,\Delta^P_\Lambda)$ are in 1-1 corresponds with simplices in $\Delta^\sigma_{\Lambda\setminus T,\Lambda_T}$ of type $\Lambda_T\setminus R'$ with $R'\in \mathcal C'_{P,T}$.  Two vertices are adjacent if their types are comparable with respect to the partial order on $\mathcal C'_{P,T}$ defined in Lemma~\ref{lem:admissible example} and their corresponding simplices are contained in the same simplex of $\Delta^\sigma_{\Lambda\setminus T,\Lambda_T}$. We view $\Delta^\sigma_{\Lambda\setminus T,\Lambda_T}$ as a simplicial complex of type $S'$ with $S'$ being the vertex set of $\Lambda_T$. Then Assumption (2) of the proposition implies that $\Delta^\sigma_{\Lambda\setminus T,\Lambda_T}$ satisfies the labeled 4-cycle condition (as a complex of type $S'$), hence satisfies the strong labeled 4-cycle condition by the same proof as Proposition~\ref{prop:strong}. Note that Assumption (1) implies $\Delta^\sigma_{\Lambda\setminus T}$ satisfies condition C$_4$, hence the same holds for $\Delta^\sigma_{\Lambda\setminus T,\Lambda_T}$. This replace the dependence of Proposition~\ref{prop:strong} on Lemma~\ref{lem:4-cycle}.
We finish the rest of the proof using Lemma~\ref{lem:admissible} and Lemma~\ref{lem:admissible example} as before.
\end{proof}

The following version of Proposition~\ref{prop:prop} for simplicial complexes of type $S$ can be proved in a similar way as before.
\begin{prop}
	\label{prop:propgeneral}
	Let $\Delta_S$ be a regular simplicial complex of type $S$ modeled on a connected diagram $\Lambda$ which is not a tree. Suppose $\Delta_S$ satisfies C$_4$, and for each nonempty simplex $\sigma$ of $\Delta_S$, $\lk(\sigma,\Delta_S)$ satisfies the labeled 4-cycle condition. Then $\Delta_{\Lambda}$ satisfies the labeled 4-cycle condition.
\end{prop}

The following is a consequence of Proposition~\ref{prop:propgeneral}, Proposition~\ref{prop:mincutgeneral}, Corollary~\ref{cor:bestvina} and Lemma~\ref{lem:garsideA_n}.

\begin{thm}
	\label{thm:cycle reduction Bestvina}
Let $\Delta_S$ be a regular simplicial complex of type $S$ modeled on diagram $\Lambda$. Suppose that
\begin{enumerate}
	\item $\Delta_S$ satisfies condition C$_4$;
	\item for any simplex $\sigma\subset \Delta_S$ of type $\hat T$ such that the induced subdiagram of $\Lambda$ on $S\setminus T$ is a tree,  $\lk(\sigma,\Delta_S)$ satisfies the labeled $4$-cycle condition.
\end{enumerate}
Then for any unbranched path $P$ in $\Lambda$ with $\Lambda_P$ being connected, the minimal cut complex $\Delta^P_S$ is an $\widetilde A_n$-like complex. Hence $\Delta^P_S$ is the Bestvina complex of a suitable homogeneous categorical Garside structure, in particular it is contractible.
\end{thm}

\begin{thm}
	\label{thm:cycle redution contractible}
	Let $\Delta_S$ be a regular simplicial complex of type $S$ modeled on a connected diagram $\Lambda$ which contains a cycle. 
	Suppose that
	\begin{enumerate}
		\item $\Delta_S$ satisfies condition C$_4$;
		\item for any simplex $\sigma\subset \Delta_S$ of type $\hat T$ such that the induced subdiagram on $S\setminus T$ is a tree, $\lk(\sigma,\Delta_S)$ satisfies the labeled $4$-cycle condition. 
	\end{enumerate}
	Then $\Delta_S$ is contractible.
\end{thm}

\begin{proof}
Let $P(n)$ be the statement that for any simplex $\sigma\subset \Delta_S$ of type $\hat T$ such that $|S\setminus T|\le n$ and the induce subdiagram $\Lambda_{S\setminus T}$ contains a cycle, $\Delta^\sigma_{S\setminus T}$ is contractible. Note that $P(1)$ is vacuously true. Now suppose $P(n-1)$ is true. Take $\sigma$ of type $\hat T$ such that $|S\setminus T|=n$ and $\Lambda_{S\setminus T}$ contains a cycle. Then by inductive hypothesis and Lemma~\ref{lem:dr}, $\Delta^\sigma_{S\setminus T}$ deformation retracts onto $\Delta^\sigma_{S\setminus T,\Lambda'}$ for any connected induced subdiagram $\Lambda'$ of $\Lambda_{S\setminus T}$ that contains a cycle. This would replace the role of Claim 1 in the proof of Theorem~\ref{thm:cycle reduction} (2). The rest of the induction is identical to Theorem~\ref{thm:cycle reduction} (2), except we adjust the definition of $\Delta_0$ in the same way as adjusting definition of $\Delta_1$ in the proof of Proposition~\ref{prop:mincutgeneral}, and use Lemma~\ref{lem:sc general} to conclude simply connectedness of $\Delta_0$ instead of \cite{cumplido2020parabolic}.
\end{proof}

Next we generalize the material in Section~\ref{subsec:G2} to simplicial complexes of type $S$. Lemma~\ref{lem:subdiagrams} needs reproving as its proof in \cite{huang2023labeled} relies on nearest point projections in Davis complexes, which is not available here.

Given a diagram $\Lambda$, we define the notion of \emph{admissible} subdiagram $\Lambda'$ of $\Lambda$ in the same way as Section~\ref{subsec:posets}. When $\Lambda'$ is a linear subdiagram, and $\Delta_\Lambda$ is modeled on $\Lambda$, we define bowtie freeness of $\Delta_{\Lambda,\Lambda'}$ in the same way as Definition~\ref{def:bowtie free}.

\begin{lem}
	\label{lem:subdiagramsgeneral}
	Let $\Delta_S$ be a simplicial complex of type $S$ modeled on a diagram $\Lambda$. Suppose $\Delta_S$ satisfies condition C$_4$. Let $\Lambda',\Lambda_1,\Lambda'$ be connected induced subdiagrams of $\Lambda$ with $\Lambda'_1\subset\Lambda_1$, $\Lambda'_1\subset \Lambda'$ and $\Lambda'$ being an admissible linear subdiagram of $\Lambda$. Suppose $\Delta_{\Lambda,\Lambda'}$ is bowtie free. Then $\Delta^\sigma_{\Lambda_1,\Lambda'_1}$ is bowtie free for any simplex $\sigma$ of type $\hat T$ with $T\cap \Lambda_1=\emptyset$.
\end{lem}
Note that $T\cap \Lambda_1=\emptyset$ guarantees that $\Delta^\sigma_{\Lambda_1,\Lambda'_1}$ is well-defined (Remark~\ref{rmk:notation}).
\begin{proof}
It suffices the treat the case $\Lambda'_1=\Lambda'$. Up to possibly enlarging $\sigma$ we can assume $\Lambda_1$ is exactly the induced subdiagram of $\Lambda$ on $S\setminus T$. Pick a linear order of vertices of $\Lambda'$, and let $x_1x_2x_3x_4$ be an embedded 4-cycle in $\Delta^\sigma_{\Lambda_1,\Lambda'}$ of type $\hat s\hat t\hat s\hat t$ with $x_1<x_2$. Let $z$ be the join of $x_1$ and $x_3$ in $\Delta^0_{\Lambda,\Lambda'}$. Then $z\le \{x_2,x_4\}$. It suffices to show $z\in \lk(\sigma,\Delta_{\Lambda})$. Suppose this it not the case. Consider the generalized 4-cycle $x_1zx_3b_{\sigma}$ with $b_\sigma$ being the barycenter of $\sigma$. By condition C$_4$, there is a vertex $z'$ with the same type as $z$ such that $z'\sim\{x_1,x_3,b_{\sigma}\}$. In particular $z'\in \lk(\sigma,\Delta_S)$ and $z'\neq z$. As $z$ is the join of $x_1$ and $x_3$ and $z'$ is a common upper bound for $\{x_1,x_3\}$ in $\Delta^0_{\Lambda,\Lambda'}$, we have $z\le z'$. Hence $z=z'$ as they have the same type, which is a contradiction. Hence the lemma is proved.
\end{proof}

A \emph{marked diagram} $\Lambda$ is a diagram with a set $E$ of \emph{marked edges}. We say $\Delta_S$ is \emph{modeled on the marked diagram $\Lambda$}, if \eqref{eq:join} holds, and for any codimension 2 simplex $\sigma\subset \Delta_S$ of type $\hat T$ such that $S\setminus T$ span a marked edge, $\lk(\sigma,\Delta_S)\ge 12$.

The following is a combination of Proposition~\ref{prop:G2bowtie free} and \cite[Prop 9.11 and Cor 9.12]{huang2023labeled}, whose proofs adjust naturally to the more general setting of simplicial complexes of type $S$ (we replace the role of Lemma~\ref{lem:subdiagrams} by Lemma~\ref{lem:subdiagramsgeneral}, and use Lemma~\ref{lem:sc general} to justify relevant simply connectedness statement).

\begin{thm}
	\label{thm:G2 general}
	Let $\Delta_S$ be a regular simplicial complex of type $S$ modeled on a marked tree diagram $\Lambda$ with marked edge set $E$. 
	Suppose that
	\begin{enumerate}
		\item $\Delta_S$ satisfies condition C$_4$;
		\item for any simplex $\sigma\subset \Delta_S$ of type $\hat T$ such that the induced subdiagram on $S\setminus T$ is a tree without marked edges, $\lk(\sigma,\Delta_S)$ satisfies the labeled $4$-cycle condition. 
	\end{enumerate}
	Then  $\Delta_S$ satisfies the labeled 4-cycle condition and $\Delta_{S,e}$ has girth $\ge 12$ for each edge $e\in E$. Moreover, 
	\begin{enumerate}
		\item for any linear subdiagram $\Lambda_0\subset\Lambda$ made of three consecutive vertices $s,t,r$ with $st\in E$, $\Delta_{\Lambda,\Lambda_0}$ admits a piecewise Euclidean CAT$(0)$ metric;
		\item if $|V\Lambda|\ge 3$ and $E\neq\emptyset$, then $\Delta_\Lambda$ is contractible.
	\end{enumerate}
\end{thm}
The following is combination of Theorem~\ref{thm:cycle reduction Bestvina}, Theorem~\ref{thm:cycle redution contractible} and Theorem~\ref{thm:G2 general}.
\begin{cor}
	\label{cor:reducion together}
	Let $\Delta_S$ be a regular simplicial complex of type $S$ modeled on a connected marked diagram $\Lambda$ with marked edge set $E$. 
Suppose that
\begin{enumerate}
	\item $\Delta_S$ satisfies condition C$_4$;
	\item for any simplex $\sigma\subset \Delta_S$ of type $\hat T$ such that the induced subdiagram on $S\setminus T$ is a tree without marked edges, $\lk(\sigma,\Delta_S)$ satisfies the labeled $4$-cycle condition. 
\end{enumerate}
Then if $\Lambda$ is not a tree, then for any unbranched path $P$ in $\Lambda$ with $\Lambda_P$ being connected, the minimal cut complex $\Delta^P_S$ is an $\widetilde A_n$-like complex; if $\Lambda$ is a tree, then for any linear subdiagram $\Lambda_0\subset\Lambda$ made of three consecutive vertices $s,t,r$ with $st\in E$, $\Delta_{\Lambda,\Lambda_0}$ admits a piecewise Euclidean CAT$(0)$ metric. 

Moreover, $\Delta_\Lambda$ is contractible if either $\Lambda$ is a tree with  $|V\Lambda|\ge 3$ and $E\neq\emptyset$, or $\Lambda$ contains a cycle.
\end{cor}

\subsection{A group theoretical setting}
Let $\Lambda$ be a finite diagram with vertex set $S$. Let $G$ be a group generated by a family of subgroups $\{H_s\}_{s\in S}$. For $T\subset S$, let $G_T$ be the subgroup generated by $\{H_s\}_{s\in T}$ (with $G_\emptyset$ being the trivial subgroup). For each pair $s,t$ with $s\neq t$, choose relations \(\mathcal R_{st}\) presenting \(G_{\{s,t\}}\) as a quotient of \(H_s*H_t\).

For each $s\in S$, let $G_{\hat s}=G_{S\setminus\{s\}}$.
We define the \emph{coset complex} $\Delta_\Lambda$ (or $\Delta_S$) associated with the admissible parabolic system as follows. Its vertices is in one to one correspondence with left cosets of $\{G_{\hat s}\}_{s\in S}$. A collection of vertices span a simplex if the corresponding cosets have non-empty common intersection. In other words, $\Delta_S$ is the nerve of the covering of $G$ by left cosets of $\{G_{\hat s}\}_{s\in S}$. For each induced subdiagram $\Lambda'\subset \Lambda$, we define \emph{relative coset complex} $\Delta_{\Lambda,\Lambda'}$, and types of vertices and simplices in $\Delta_{\Lambda,\Lambda'}$ in the same way as before. 
\begin{definition}
	\label{def:admissible parabolic system}
We call $(G,\{H_s\}_{s\in S})$ an
\emph{admissible parabolic system over $\Lambda$}, if the following properties hold simultaneously:
\begin{enumerate}
\item	 For any $T\subset S$, the natural homomorphism from the following group to $G_T$ is an isomorphism:
	\[
		\left(\underset{s\in T}{*}H_s\right)
		\Big/
		\left\langle\!\left\langle
		\mathcal R_{st}:\,s,t\in T,\ s\ne t
		\right\rangle\!\right\rangle;
	\]
	\item $H_s$ and $H_t$ commute if $s,t$ are not adjacent in $\Lambda$;
	\item For any $U,T_1,T_2,T_3\subset S$, we have $G_U\cap (G_{T_1}G_{T_2}G_{T_3})=G_{U\cap T_1}G_{U\cap T_2}G_{U\cap T_3}$.
\end{enumerate}
If $\Lambda$ is a marked diagram with marked edge set $E$, then we require in addition that for any edge $e=st\in E$, the coset complex for $G_{\{s,t\}}$ is a graph with girth $\ge 12$.
\end{definition}

In Definition~\ref{def:admissible parabolic system} (3), we could assume instead $U,T_1,T_2,T_3\subsetneq S$, as if one of these sets are $S$, the property is trivial.

\begin{lem}
	\label{lem:group}
Suppose $(G,\{H_s\}_{s\in S})$ is an admissible parabolic system over $\Lambda$ (marked or unmarked). Then $\Delta_S$ is a regular simplicial complex of type $S$ modeled on $\Lambda$ and it satisfies condition $C_4$.
\end{lem}

\begin{proof}
By Definition~\ref{def:admissible parabolic system} with $T_2=T_3=\emptyset$, we have $G_U\cap G_{T_1}=G_{U\cap T_1}$ for any $U,T_1\subset S$. So simplices of type $\hat T$ are in one to one correspondence with left cosets of $G_{S\setminus T}$ in $G$, moreover, for any simplex $\sigma\subset \Delta_S$ of type $\hat T$, $\lk(\sigma,\Delta_S)$ can be naturally identified with $\Delta_{S\setminus T}$, which is the never of the covering of $G_{R}$ (with $R=S\setminus T$) by left cosets of $\{G_{R\setminus \{s\}}\}_{s\in S}$. When $|R|\ge 2$, $\{G_{R\setminus\{s\}}\}_{s\in S}$ generates $G_R$; and when $R\ge 3$, by Definition~\ref{def:admissible parabolic system} (1), each defining relation of $G_R$ is supported in one of $\{G_{R\setminus\{s\}}\}_{s\in S}$. Hence by \cite[Thm 2.4]{abels1993higher}, $\Delta_R$ is connected when $|R|\ge 2$ and $\Delta_R$ is simply-connected when $|R|\ge 3$. This implies that $\Delta_S$ is regular.

The equation~\ref{eq:join} follows from $G_U\cap G_{T_1}=G_{U\cap T_1}$ and Definition~\ref{def:admissible parabolic system} (2).

Now we show $\Delta_S$ is flag. It suffices to show that for any $T_1,T_2,\ldots,T_k\subset S$ and $g_1,g_2,\ldots,g_k\in G$, if $\{g_iG_{T_i}\}_{i=1}^k$ pairwise have nonempty intersection, then they have nonempty common intersection. 
It suffice to prove this with $k=3$, as the passage from $k=3$ to more general case is identical to the proof of \cite[Prop 4.5]{godelle2012k}, using that $G_U\cap G_{T_1}=G_{U\cap T_1}$. When $k=3$, we can assume without of loss of generality that $g_1,g_2$ is the identity element, hence the desired property follows from Definition~\ref{def:admissible parabolic system} (3) by setting $T_3=\emptyset$.

Condition C$_4$ follows from Definition~\ref{def:admissible parabolic system} (3).
\end{proof}

The following combines Lemma~\ref{lem:group}, Theorem~\ref{thm:cycle reduction Bestvina}, Theorem~\ref{thm:cycle redution contractible}, Theorem~\ref{thm:G2 general} and Corollary~\ref{cor:reducion together}.
\begin{thm}
	\label{thm:group}
Suppose $(G,\{H_s\}_{s\in S})$ is an admissible parabolic system over $\Lambda$. Let $\Delta_S$ be the associated coset complex. 
\begin{enumerate}
	\item if Assumption (2) of Theorem~\ref{thm:cycle reduction Bestvina} holds, then the conclusion of Theorem~\ref{thm:cycle reduction Bestvina} holds with $G$ acting on $\Delta^P_S$ cocompactly by isometries;
	\item if Assumption (2) of Theorem~\ref{thm:cycle reduction Bestvina} holds, then $\Delta_S$ is contractible;
	\item if $\Lambda$ is a marked tree and Assumption (2) of Theorem~\ref{thm:G2 general} holds, then the conclusion of Theorem~\ref{thm:G2 general} holds with $G$ acting on $\Delta_{\Lambda,\Lambda_0}$ cocompactly;
	\item if $\Lambda$ is a marked diagram and Assumption (2) of Corollary~\ref{cor:reducion together} holds, then the conclusion of Corollary~\ref{cor:reducion together} holds with $G$ acts cocompactly by isometries on all the relevant complexes.
\end{enumerate}
\end{thm}

A concrete class where the framework of Definition~\ref{def:admissible parabolic system} applies is the class of \emph{Shephard groups}. 
Recall that an \emph{extended Coxeter diagram} is a Coxeter diagram with each vertex labeled by an element in $\{2,3,4,\ldots\}\cup \{\infty\}$ such that any two adjacent vertices connected by an odd edge in the diagram share the same vertex label. We still use $A_\Lambda$ and $W_\Lambda$ to denote the corresponding Artin group and Coxeter group, ignoring the vertex labels. Let $\mathfrak{S}_\Lambda$ be the associated \emph{Shephard group}, which is the quotient
\[
\mathfrak{S}_\Lambda=A_\Lambda/\langle\!\langle s_i^{p_i}\rangle\!\rangle_{i\in I}
\]
where $\{s_i\}_{i\in I}$ denotes the collection of vertices of $\Lambda$ whose label $p_i$ is finite. Let $S$ be the collection of all vertices of $\Lambda$. A \emph{standard parabolic subgroup} $\mathfrak{S}_T $ of $\mathfrak{S}_\Lambda$ is the subgroup of $\mathfrak{S}_\Lambda$ generated by $T\subset S$. 
A Shephard group $\mathfrak{S}_\Lambda$ is \emph{admissible} if
\begin{enumerate}
	\item for any induced subdiagram $\Lambda'\subset\Lambda$, the natural homomorphism $\mathfrak{S}_{\Lambda'}\to \mathfrak{S}_\Lambda$ is injective;
	\item For any $U,T_1,T_2,T_3\subset S$, we have 
	\begin{center}
	$\mathfrak{S}_U\cap (\mathfrak{S}_{T_1}\mathfrak{S}_{T_2}\mathfrak{S}_{T_3})=\mathfrak{S}_{U\cap T_1}\mathfrak{S}_{U\cap T_2}\mathfrak{S}_{U\cap T_3}$.
	\end{center}
\end{enumerate}

Each extended Coxeter diagram $\Lambda$ determined a marked diagram, by forgetting vertex labels and edge labels, and mark the collection of edges with label $\ge 6$. The following is a direct consequence of the definition, together with \cite[Prop 4.3]{goldman20242} which implies the additional requirement at the end of Definition~\ref{def:admissible parabolic system}.

\begin{lem}
	\label{lem:shephard to general}
Suppose $\mathfrak{S}_\Lambda$ is an admissible Shephard group over an extended Coxeter diagram $\Lambda$ with vertex set $S$. Then $(\mathfrak{S}_\Lambda,\{H_s\}_{s\in S})$ is an admissible parabolic system over the marked diagram determined by $\Lambda$, where $H_s$ is the cyclic subgroup of $\mathfrak{S}_\Lambda$ generated by $s$.
\end{lem}

A Shephard group is \emph{even}, if its vertex labels are either even or infinity.

\begin{thm}(\cite{goldman2026})
	\label{lem:even shephard}
Even Shephard groups are admissible.
\end{thm}

The following is combination of Lemma~\ref{lem:shephard to general}, Theorem~\ref{lem:even shephard} and Theorem~\ref{thm:group}.

\begin{thm}
Suppose $\mathfrak{S}_\Lambda$ is an admissible Shephard group on an extended Coxeter diagram $\Lambda$ with its coset complex $\Delta_S$. Then the conclusion of Theorem~\ref{thm:group} holds. In particular, the conclusion of Theorem~\ref{thm:group} holds for any even Shephard group.
\end{thm}

%
%
%

\bibliographystyle{alpha}
\bibliography{mybib}
\end{document}